\documentclass[11pt,a4paper,oneside,reqno]{amsart}

\usepackage{amsmath, amsthm, amsopn, amssymb}
\usepackage{mathtools, thmtools, thm-restate}
\usepackage{bbm}
\usepackage{enumerate}
\usepackage{enumitem}
\usepackage{hyperref}
\usepackage{multirow, makecell}
\usepackage{caption} 
\usepackage{stmaryrd} 
\usepackage{xcolor}

\usepackage{geometry}
\usepackage{setspace}
\declaretheorem[name=Theorem,within=section]{thm}
\declaretheorem[name=Lemma,sibling=thm]{lemma}
\declaretheorem[name=Question,sibling=thm]{quest}
\declaretheorem[name=Proposition,sibling=thm]{prop}

\newtheorem{cor}[thm]{Corollary}

\declaretheorem[name=Claim,sibling=thm]{claim}
\declaretheorem[name=Subclaim,within=claim]{subclaim}

\newtheorem{obs}[thm]{Observation}

\declaretheorem[name=Definition,sibling=thm,style=definition]{dfn}
\declaretheorem[name=Plan,style=definition,numbered=no]{plan}
\declaretheorem[name=Algorithm,style=definition,numbered=no]{alg}

\newcommand{\bF}{\mathbf{F}}
\newcommand{\bG}{\mathbf{G}}

\newcommand{\NN}{\mathbb{N}}

\newcommand{\cA}{\mathcal{A}}
\newcommand{\cB}{\mathcal{B}}
\newcommand{\cC}{\mathcal{C}}
\newcommand{\cD}{\mathcal{D}}
\newcommand{\cE}{\mathcal{E}}
\newcommand{\cF}{\mathcal{F}}
\newcommand{\cG}{\mathcal{G}}
\newcommand{\cH}{\mathcal{H}}

\newcommand{\cM}{\mathcal{M}}
\newcommand{\cN}{\mathcal{N}}

\newcommand{\cQ}{\mathcal{Q}}
\newcommand{\cR}{\mathcal{R}}
\newcommand{\cS}{\mathcal{S}}
\newcommand{\cT}{\mathcal{T}}

\newcommand{\vell}{\mathbf{l}}

\renewcommand{\Pr}{\mathbb{P}}
\newcommand{\Ex}{\mathbb{E}}

\newcommand{\1}{\mathbbm{1}} 

\newcommand{\Bin}{\mathrm{Bin}}
\newcommand{\Pois}{\mathrm{Pois}}

\newcommand{\eps}{\varepsilon}

\newcommand{\br}[1]{\llbracket{#1}\rrbracket}

\renewcommand{\le}{\leqslant}
\renewcommand{\ge}{\geqslant}

\newcommand{\cFQL}{\cF_Q^{\mathrm{low}}}
\newcommand{\cHQL}{\cH_Q^{\mathrm{low}}}
\newcommand{\cFQH}{\cF_Q^{\mathrm{high}}}
\newcommand{\cHQH}{\cH_Q^{\mathrm{high}}}

\newcommand{\cHh}{\hat{\cH}}

\newcommand{\balpha}{\bar\alpha}

\newcommand{\vmin}{v_{\min}}

\newcommand{\cGb}{\bar{\cG}}
\newcommand{\cGh}{\hat{\cG}}

\newcommand{\Cpar}{C_{\partial}}
\newcommand{\Clow}{C_{\mathrm{low}}}
\newcommand{\chigh}{c_{\mathrm{high}}}
\newcommand{\Cth}{C_{\theta}}

\newcommand{\deficit}{\mathrm{def}}
\newcommand{\ext}{\mathrm{ext}}
\renewcommand{\int}{\mathrm{int}}
\newcommand{\crit}{\mathrm{crit}}
\newcommand{\core}{\mathrm{core}}
\newcommand{\maxcut}{\mathrm{maxcut}}

\author{Ilay Hoshen}
\address{School of Mathematical Sciences, Tel Aviv University, Tel Aviv 6997801, Israel}
\email{ilayhoshen@gmail.com}

\author{Wojciech Samotij}
\address{School of Mathematical Sciences, Tel Aviv University, Tel Aviv 6997801, Israel}
\email{samotij@tauex.tau.ac.il}

\title{Simonovits's theorem in random graphs}

\thanks{This research was supported by: the Israel Science Foundation grant 2110/22; the grant 2019679 from the United States--Israel Binational Science Foundation (BSF) and the United States National Science Foundation (NSF); and the ERC Consolidator Grant 101044123 (RandomHypGra).}

\begin{document}

\begin{abstract}
  Let $H$ be a graph with $\chi(H) = r+1$.
  Simonovits's theorem states that, if $H$ is edge-critical, the unique largest $H$-free subgraph of $K_n$ is its largest $r$-partite subgraph, provided that $n$ is sufficiently large.
  We show that the same holds with $K_n$ replaced by the binomial random graph $G_{n,p}$ whenever $H$ is also strictly $2$-balanced and $p \ge (\theta_H+o(1)) n^{-\frac{1}{m_2(H)}} (\log n)^{\frac{1}{e_H-1}}$ for some explicit constant $\theta_H$, which we believe to be optimal. This (partially) resolves a conjecture of DeMarco and Kahn.
\end{abstract}

\maketitle

\setcounter{tocdepth}{1}
\tableofcontents

\section{Introduction}

The celebrated theorem of Tur\'an~\cite{Tur41} states that, for every $r \ge 2$ and all $n$, the unique largest $K_{r+1}$-free subgraph of the complete graph $K_n$ is its largest $r$-partite subgraph.
The starting point of our work is the following generalisation of this result, proved by Simonovits~\cite{Sim68}.  We call a graph $H$ \emph{edge-critical} if $\chi(H \setminus e) < \chi(H)$ for some $e \in H$.

\begin{thm}[\cite{Sim68}]\label{thm:Simonovits}
  If a graph $H$ is edge-critical and $n$ is a large enough integer, then every largest $H$-free subgraph of $K_n$ is $(\chi(H) - 1)$-partite.
\end{thm}

We will say that a graph $G$ is \emph{$H$-Simonovits} if every largest $H$-free subgraph of $G$ is $(\chi(H)-1)$-partite, so that~\autoref{thm:Simonovits} can be concisely restated as: `If $H$ is edge-critical, then every sufficiently large complete graph is $H$-Simonovits.'
On the other hand, observe that when $H$ is not edge-critical, no graph with chromatic number at least $\chi(H)$ can be $H$-Simonovits; indeed, adding one edge to a $(\chi(H)-1)$-partite graph cannot introduce a copy of $H$ unless $H$ is edge-critical.

In this work, we study the following question, which, to the best of our knowledge, was first explicitly considered by Babai, Simonovits, and Spencer~\cite{BabSimSpe90}.

\begin{quest}
  \label{quest:BSS}
  Suppose that $H$ is an edge-critical graph.  For what $p$ is the binomial random graph $G_{n,p}$ a.a.s.\ $H$-Simonovits?
\end{quest}

The main result of~\cite{BabSimSpe90} was that, for every $\ell \ge 1$, the random graph $G_{n,p}$ is typically $C_{2\ell+1}$-Simonovits as long as $p \ge 1/2 - \eps_\ell$ for some (small) positive constant $\eps_\ell$ that depends only on $\ell$.
Answering a challenge raised by the authors of~\cite{BabSimSpe90}, Brightwell, Panagiotou, and Steger~\cite{BriPanSte12} proved that $G_{n,p}$ is a.a.s.\ $K_{r+1}$-Simonovits, for every $r \ge 2$, already when $p \ge n^{-c_r}$ for some (small) positive constant $c_r$.

What are the necessary conditions on $p$ in \autoref{quest:BSS}?
A standard deletion argument shows that, when $p \gg n^{-2}$ and $n^{v_F} p^{e_F} \ll n^2p$ for some nonempty subgraph $F \subseteq H$, the random graph $G_{n,p}$ a.a.s.\ contains an $H$-free subgraph with $(1-o(1))\binom{n}{2}p$ edges; on the other hand, if $p \gg n^{-1}$, then no $r$-partite subgraph of $G_{n,p}$ can have more than $(1-1/r+o(1))\binom{n}{2}p$ edges.
This simple reasoning shows that, in order for $G_{n,p}$ to be a.a.s.\ $H$-Simonovits, one has to assume that $p = \Omega(n^{-1/m_2(H)})$, where
\begin{equation}
  \label{eq:2-density}
  m_2(H) \coloneqq \max\left\{ \frac{e_F - 1}{v_F - 2} \colon F \subseteq H,\ e_F \ge 2\right\}
\end{equation}
is the so-called \emph{$2$-density} of $H$.

On the other hand, when $p \gg n^{-1/m_2(H)}$, then $G_{n,p}$ typically becomes \emph{approximately} $H$-Simonovits, even when $H$ is not edge-critical.  This was first conjectured by Kohayakawa, {\L}uczak, and R\"odl~\cite{KohLucRod97} and later proved in the breakthrough work of Conlon and Gowers~\cite{ConGow16}, under the technical assumption that $H$ is strictly $2$-balanced (see below), which was later removed by the second author~\cite{Sam14}, using an adaptation of the argument of Schacht~\cite{Sch16}.

\begin{thm}[\cite{ConGow16,Sam14}]
  \label{thm:Conlon-Gowers}
  For every nonbipartite graph $H$ and every $\beta > 0$, there exist a~positive $C$ such that, when $p \ge Cn^{-1/m_2(H)}$, then a.a.s.\ every largest $H$-free subgraph of $G_{n,p}$ can be made $(\chi(H) - 1)$-partite by removal of at most $\beta n^2p$ edges.
\end{thm}

Unfortunately, the methods of proof of \autoref{thm:Conlon-Gowers} are not sufficiently accurate for addressing the more delicate \autoref{quest:BSS}; in particular, they are insensitive to the assumption that the forbidden subgraph $H$ is edge-critical.  Only in subsequent tour de force work, DeMarco and Kahn showed that adding an extra polylogarithmic factor in the lower-bound assumption on $p$ suffices for $G_{n,p}$ to a.a.s.\ become \emph{exactly} $H$-Simonovits in the case where $H$ is a clique, first in the case $H = K_3$~\cite{DeMKah15Man} (where the corresponding approximate statement had been proved already in~\cite{KohLucRod97}) and then $H = K_{r+1}$ for all $r \ge 2$~\cite{DeMKah15Tur}.

\begin{thm}[\cite{DeMKah15Man, DeMKah15Tur}]
  \label{thm:DeMarco-Kahn}
  For every integer $r \ge 2$, there is a constant $C_r$ such that, if
  \[
    p \ge C_r n^{-\frac{1}{m_2(K_{r+1})}} (\log n)^{\frac{1}{e(K_{r+1}) - 1}},
  \]
  then $G_{n,p}$ is a.a.s. $K_{r+1}$-Simonovits.
\end{thm}

A key feature of \autoref{thm:DeMarco-Kahn} is that the lower-bound assumption on $p$ is best-possible, up to a multiplicative constant factor.
This is because, when $n^{-1} \ll p \le c_r n^{-1/m_2(K_{r+1})} (\log n)^{1/(e(K_{r+1})-1)}$ for a sufficiently small positive constant $c_r$, one can a.a.s.\ find in $G_{n,p}$ a subgraph $G$ with chromatic number larger than $r$ such that no edge of $G$ belongs to a copy of $K_{r+1}$ in $G_{n,p}$.
(This was first observed in~\cite{BriPanSte12} in the case $r=2$ and the subgraph being a copy of the $5$-cycle.)
In particular, as every maximal $K_{r+1}$-free subgraph of $G_{n,p}$ must contain all edges that do not belong to a copy of $K_{r+1}$, this means that no largest $K_{r+1}$-free subgraph of $G_{n,p}$ is $r$-partite.

The main result of our work is a generalisation of \autoref{thm:DeMarco-Kahn} from cliques to arbitrary edge-critical graphs that are strictly $2$-balanced.
A graph $H$ is called \emph{strictly $2$-balanced} if the maximum in the definition of the $2$-density $m_2(H)$, given in~\eqref{eq:2-density}, is achieved uniquely at $F = H$.
It is not difficult to check that all cliques, as well as all cycles, are strictly $2$-balanced, so our result is indeed a generalisation of \autoref{thm:DeMarco-Kahn}, as well as the results of Babai, Simonovits, and Spencer~\cite{BabSimSpe90}.

Our lower-bound assumption on $p$ involves an explicit constant that we now define.
Given integers $m$ and $r$, denote by $K_r(m)$ the $m$-blowup of $K_r$, that is, the balanced, complete $r$-partite graph with $rm$ vertices and denote by $K_r(m)^+$ the graph obtained from $K_r(m)$ by adding to it an arbitrary edge (contained in one of the $r$ colour classes).
Suppose now that $H$ is an edge-critical graph and note that $H \subseteq K_{\chi(H)-1}(m)^+$ for all $m \ge v_H$.
Denoting the number of copies of $H$ in a graph $G$ by $N(H, G)$, we let
\begin{equation}
  \label{eq:pi_H}
  \pi_H \coloneqq \lim_{m \to \infty} \frac{N\big(H, K_{\chi(H)-1}(m)^+\big)}{m^{v_H-2}} > 0
\end{equation}
and let $\theta_H$ be the constant satisfying
\begin{equation}
  \label{eq:theta_H}
  (\chi(H)-1)^{2-v_H} \pi_H \theta_H^{e_H-1} = 2 - \frac{1}{m_2(H)}.
\end{equation}
We are now ready to state the main result of this work.

\begin{thm}
  \label{thm:main}
  If $H$ is a nonbipartite, edge-critical, strictly $2$-balanced graph and
  \begin{equation}
    \label{eq:main-p-assumption}
    p \ge (\theta_H+\eps) \cdot n^{-\frac{1}{m_2(H)}} (\log n)^{\frac{1}{e_H - 1}},
  \end{equation}
  for some positive constant $\eps$, then $G_{n,p}$ is a.a.s.\ $H$-Simonovits.
\end{thm}

It is not hard to check that the function $p_H \coloneqq \theta_H \cdot n^{-\frac{1}{m_2(H)}} (\log n)^{\frac{1}{e_H-1}}$ is the sharp threshold for the property that, for some \emph{fixed} equipartition $\Pi = \{V_1, \dotsc, V_{\chi(H)-1}\}$ of the vertices of $K_n$, every edge $e$ of $G_{n,p}$ whose both endpoints lie in the same $V_i$ extends to a copy of $H$ in $G_{n,p}$ whose all remaining edges have endpoints in different $V_i$s (i.e., $H \subseteq e \cup (G_{n,p} \cap \ext(\Pi))$ using the notation introduced below).
We strongly believe that our lower-bound assumption on $p$ is optimal.
Our belief contradicts the prediction made by DeMarco and Kahn~\cite{DeMKah15Tur}, who suggested that it is enough to assume that $p$ is above the threshold for the (weaker) property that every edge of $G_{n,p}$ extends to some copy of $H$ (without the additional restriction that the copy of $H$ needs to cross the fixed equipartition $\Pi$).

One of the key ingredients in our proof of \autoref{thm:main} requires the assumption that $p = 1-\Omega(1)$.
Luckily, in the complementary range $p = 1-o(1)$, \autoref{thm:main} is a straightforward consequence of the following result of Alon, Shapira, and Sudakov~\cite{AloShaSud09} and the fact that $\delta(G_{n,p}) \ge (p-o(1))n$ asymptotically almost surely.

\begin{thm}[\cite{AloShaSud09}]
  \label{thm:AlonShapiraSudakov}
  For every edge-critical graph $H$, there exists a positive constant $\mu$ such that every $n$-vertex graph with minimum degree at least $(1-\mu)n$ is $H$-Simonovits.
\end{thm}

We remark that~\cite[Theorem~6.1]{AloShaSud09} proves only the marginally weaker assertion that, in every $n$-vertex graph with minimum degree at least $(1-\mu)n$, the largest size of an $H$-free subgraph equals the largest size of an $r$-partite subgraph.  However, a minor adjustment of the proof of~\cite[Theorem~6.1]{AloShaSud09} yields the stronger \autoref{thm:AlonShapiraSudakov}.  For the sake of completeness, we include this modified argument in \autoref{sec:proof-AlonShapiraSudakov}.

\subsection{Related work}
\label{sec:related-work}

There is a closely related body of work concerning the structure of a random $H$-free graph (as opposed to the structure of largest subgraphs of a random graph).
This problem was first considered by Erd\H{o}s, Kleitman, and Rothschild~\cite{ErdKleRot76}, who proved that a random $K_3$-free graph is typically bipartite.
The analogous statement with $K_3$ replaced by $K_{r+1}$ for an arbitrary $r \ge 2$ (and bipartite with $r$-partite) was proved by Kolaitis, Pr\"omel, and Rothschild~\cite{KolProRot87}; this was further generalised by Pr\"omel and Steger~\cite{ProSte92}, who proved that a random $H$-free graph is typically $(\chi(H)-1)$-partite whenever $H$ is edge-critical.

In the past two decades, the results of~\cite{ErdKleRot76, KolProRot87, ProSte92} have been extended into the sparse regime, where interesting threshold phenomena emerge.
For a graph $H$ and integers $m, n$ satisfying $0 \le m \le \mathrm{ex}(n,H)$, let $\cF_{n,m}(H)$ denote the family of all $H$-free graphs with vertex set $\br{n}$ and precisely $m$ edges.
The following theorem combines the results of Osthus, Pr\"omel, and Taraz~\cite{OstProTar03} (the case where $H$ is an odd cycle) and Balogh, Morris, Samotij, and Warnke~\cite{BalMorSamWar16} (the case where $H$ is a clique).

\begin{thm}[\cite{BalMorSamWar16, OstProTar03}]
  \label{thm:OPT-BMSW}
  If $H$ is an odd cycle or a clique and $F_{n,m} \in \cF_{n,m}(H)$ is chosen uniformly at random, then
  \[
    \lim_{n \to \infty} \Pr\big(\text{$F_{n,m}$ is $(\chi(H)-1)$-partite}\big)
    =
    \begin{cases}
      1, & m \ge (1+\eps) m_H(n), \\
      0, & n \ll m \le (1-\eps) m_H(n),
    \end{cases}
  \]
  where $m_H(n) \coloneqq \theta_H' n^{-\frac{1}{m_2(H)}} (\log n)^{\frac{1}{e_H-1}}$ for some explicit constant $\theta_H' > 0$.
\end{thm}

Finally, we mention that Engelberg, Samotij, and Warnke~\cite{EngSamWar} have recently extended \autoref{thm:OPT-BMSW} to all edge-critical graphs; however, their result determines the threshold for the property that $F_{n,m}$ is $(\chi(H)-1)$-partite only up to a constant factor.

\subsection{Notation}
\label{sec:notation}

We briefly discuss the notation that will be used throughout the paper.
An \emph{$r$-cut} (or simply a \emph{cut}) in a graph is an ordered $r$-partition $\Pi = (V_1, \dotsc, V_r)$ of its vertex set.
We say that $\Pi$ is \emph{$\delta$-balanced} if each of its parts has size $(1 \pm \delta) \cdot n/r$, where $n$ is the number of vertices of the graph.
Following DeMarco and Kahn~\cite{DeMKah15Tur}, given a tuple $\Pi = (V_1, \dotsc, V_r)$ of pairwise-disjoint sets of vertices (not necessarily a cut), we will denote by $\ext(\Pi)$ the set of all pairs of vertices meeting two distinct $V_i$s (the `external' edges of $\Pi$) and by $\int(\Pi) \coloneqq \binom{V_1 \cup \dotsb \cup V_r}{2} \setminus \ext(\Pi)$ the set of all pairs of vertices that are contained in a single $V_i$ (the `internal' edges of $\Pi$).
For a graph $G$ and a family of cuts~$\cC$,
\[
  b_{\cC}(G) \coloneqq \max_{\Pi \in \cC} e(G \cap \ext(\Pi))
\]
is the largest number of edges of $G$ that cross some cut in $\cC$.
The \emph{deficit} of a cut $\Pi \in \cC$ in $G$ with respect to the family $\cC$ is the difference
\[
  \deficit_{\cC}(\Pi; G) \coloneqq b_{\cC}(G) - e(G \cap \ext(\Pi))
\]
between the number of edges of $G$ that cross $\Pi$ and the number of edges of $G$ that cross a largest cut from $\cC$. An \emph{$\br{r}$-coloured graph} is a graph whose vertices are coloured with colours from $\br{r}$; this colouring does not have to be proper.
For an $\br{r}$-coloured graph $Q$ and $k \in \br{r}$, we denote by $V^k(Q)$ the vertices of $Q$ with colour $k$. We say that an $r$-tuple $\Pi = (V_1, \dotsc, V_r)$ of pairwise disjoint sets of vertices (e.g., an $r$-cut) is \emph{compatible} with $Q$ if $V^k(Q) \subseteq V_k$ for every $k \in \br{r}$.

Suppose that $\cG$ is a hypergraph on a finite vertex set $V$.
Given a subset $W \subseteq V$ of vertices of $\cG$, we write $\cG[W]$ to denote the subhypergraph of $\cG$ induced by $W$, i.e., $\cG[W] \coloneqq \{A \in \cG : A \subseteq W\}$.
Further, $\nu(\cG)$ denotes the matching number of $\cG$, that is, the largest size of a collection of pairwise-disjoint edges of $\cG$.
Finally, given a vertex $e \in V$, the link hypergraph of $e$ (the neighbourhood of $e$) is the hypergraph
\[
  \partial_e\cG \coloneqq \{A \setminus \{e\} : e \in A \in \cG\};\footnote{This notation seems natural if one identifies $\cG$ with the polynomial $\sum_{A \in \cG} \prod_{e \in A} e$.}
\]
we further let $\partial \cG \coloneqq \bigcup_{e \in V} \partial_e\cG$.
Since we will often consider induced subhypergraphs in various link graphs, we use the convention that the operators $\partial_e$ and $\partial$ bind stronger than the operation of taking induced subhypergraphs, that is, $\partial \cG[W] = (\partial \cG)[W]$.

\section{Outline of the proof}
\label{section:outline}

Suppose that $H$ is an edge-critical and strictly 2-balanced graph with $\chi(H) = r + 1$, where $r \ge 2$.
Define
\[
  p_H \coloneqq \theta_H \cdot n^{-\frac{1}{m_2(H)}} (\log n)^{\frac{1}{e_H - 1}},
\]
let $\eps > 0$ be an arbitrary constant and suppose that $G \sim G_{n,p}$ for some $p \ge (1+\eps) p_H$.
We may clearly assume that $\eps < 1$ and, in light of~\autoref{thm:AlonShapiraSudakov}, that $p \le p_0$ for some $p_0 < 1$ that depends only on $H$.
We will also assume that $G$ satisfies a number of properties that are typical of graphs with density $p$, ranging from the simple requirement that all vertex degrees are $(1 \pm o(1))np$ to much more intricate conditions concerning the distribution of edges and copies of various subgraphs of $H$.
This will be made more precise in \autoref{sec:random-graphs}.

Let $F$ be a largest $H$-free subgraph of~$G$ and let $\Pi_F = (V_1, \dotsc, V_r)$ be an $r$-cut that maximises $e(\ext(\Pi_F) \cap F)$ and has the largest $e(F[V_1])$ among all such $r$-cuts;
clearly, we may choose such $\Pi_F$ for every $H$-free graph $F \subseteq K_n$ in some canonical way.
Note that we may assume that $e(F[V_1]) > 0$, as otherwise $F$ is $r$-partite and we have nothing left to prove.
On the other hand, in light of \autoref{thm:Conlon-Gowers}, we may assume that $e(F[V_1]) \le e(\int(\Pi_F) \cap F) \ll n^2 p$.

Since every $r$-partite subgraph of $G$ is trivially $H$-free, the following must be true for every family $\cC$ of $r$-cuts that includes $\Pi_F$:
\[
  e(G \cap \ext(\Pi_F)) + \deficit_{\cC}(\Pi_F;G) = b_{\cC}(G) \le e(F) = e(F \cap \ext(\Pi_F)) + e(F \cap \int(\Pi_F)).
\]
In particular, we will obtain the desired contradiction (to the assumption that $e(F[V_1]) > 0$) if we manage to show that, for some family $\cC$ of $r$-cuts that includes $\Pi_F$,
\begin{equation}
  \label{eq:contradiction-goal}
  e\big((G \setminus F) \cap \ext(\Pi_F)\big) > e(F \cap \int(\Pi_F)) - \deficit_{\cC}(\Pi_F;G).
\end{equation}
How can one bound $e\big((G \setminus F) \cap \ext(\Pi_F)\big)$ from below? The following definition is key.

\begin{dfn}
  Given an $r$-tuple $\Pi$ of pairwise-disjoint sets of vertices and a graph $Q$, we say that a copy $K$ of $H$ (in $K_n$) is \emph{$Q$-supported and $\Pi$-crossing} if $K \subseteq Q \cup \ext(\Pi)$.
\end{dfn}

The key point is that at least one edge of every copy of $H$ in $G$ that is $F$-supported and $\Pi_F$-crossing must belong to $(G \setminus F) \cap \ext(\Pi_F)$, as otherwise $F$ would contain a copy of $H$.
In particular, for every $Q \subseteq F$, the graph $(G \setminus F) \cap \ext(\Pi_F)$ must have at least as many edges as the largest size of a matching in the subhypergraph of
\[
  \cF_Q \coloneqq \{K \setminus Q : \text{$K$ is a copy of $H$ in $K_n$}\}
\]
that is induced by $G \cap \ext(\Pi_F)$, that is,
\[
  e\big((G \setminus F) \cap \ext(\Pi_F)\big) \ge \nu\big(\cF_Q[G \cap \ext(\Pi_F)]\big).
\]
Since~\eqref{eq:contradiction-goal} holds vacuously whenever $\deficit_\cC(\Pi_F; G) > e(F \cap \int(\Pi_F))$, it suffices to show that $\nu\big(\cF_Q[G \cap \ext(\Pi)]\big) > e(F \cap \int(\Pi_F))$ for \emph{some} $Q \subseteq F$ and \emph{every} cut $\Pi \in \cC$ satisfying $\deficit_{\cC}(\Pi; G) \le e(F \cap \int(\Pi_F))$.  This discussion naturally leads one to formulating the following plan for the proof of \autoref{thm:main}.

\begin{plan}
  We will construct
  \begin{itemize}
  \item
    a family $\cQ$ of $\br{r}$-coloured subgraphs of $K_n$ and
  \item
    sequences $(d_Q)_{Q \in \cQ}$ of integers and $(\cC_Q)_{Q \in \cQ}$ of $r$-cuts
  \end{itemize}
  such that $G$ has the following two properties asymptotically almost surely:
  \begin{enumerate}[label=(P\arabic*)]
  \item
    \label{item:plan-P1}
    Every largest $H$-free subgraph $F \subseteq G$ contains some $Q \in \cQ$ (coloured by $\Pi_F$) such that $d_Q \ge e(F \cap \int(\Pi_F))$ and $\Pi_F \in \cC_Q$.
  \item
    \label{item:plan-P2}
    For every $Q \in \cQ$ with $Q \subseteq G$ and every $r$-cut $\Pi \in \cC_Q$ with $\deficit_{\cC_Q}(\Pi;G) \le d_Q$, we have $\nu\big(\cF_Q[G \cap \ext(\Pi)]\big) > d_Q$.
  \end{enumerate}  
\end{plan}

\subsection{Constants}
\label{sec:constants}

There are a number of interdependent constants that will appear throughout our proof.
We have already introduced $\eps$ and $p_0$, which appear in the assumed bounds $(1+\eps)p_H \le p \le p_0$ on the density.
The remaining constants are:
\begin{itemize}
\item
  $\hat{C}$, $\Cth$, $\Cpar$, $\Clow$ : constants that depend only on $H$ that we use in place of complicated, explicit expressions;
\item
  $\alpha = \alpha(r, p_0, \eps)$ : the parameter from the definition of rigidity given in \autoref{sec:rigidity-correlation};
\item
  $\delta = \delta(r, p_0, \alpha)$ : the parameter from the definition of a balanced cut;
\item
  $Z = Z(r, p_0, \alpha)$ : the constant from the statement of \autoref{prop:algorithm-result};
\item
  $\kappa = \kappa(H, \hat{C}, \Cth, \Clow, \eps, \delta)$ : the parameter from the definition of the family $\cQ_L$ of low-degree graphs given in \autoref{sec:constructing-graphs-Q};
\item
  $\eta = \eta(H, \Cpar, Z, \eps, \kappa)$ : the parameter from the definition of the family $\cQ_H$ of high-degree graphs given in \autoref{sec:constructing-graphs-Q};
\item
  $\chigh = \chigh(H, \eta)$ : the constant from the statement of~\autoref{lemma:high-degree-bounds}.
\end{itemize}
Finally, we will also denote by $\beta = \beta(n) \ll 1$ the function that is implicit in the statement of \autoref{thm:Conlon-Gowers} invoked with $p \ge p_H \gg n^{-1/m_2(H)}$.
That is, asymptotically almost surely, for every largest $H$-free subgraph of $F \subseteq G$, we have $e(F \cap \int(\Pi_F)) \le \beta n^2p \ll n^2p$.

\subsection{Constructing the family $\cQ$}
\label{sec:constructing-graphs-Q}

In this section, we construct the family $\cQ$ and the sequences $(d_Q)_{Q \in \cQ}$ of integers and $(\cC_Q)_{Q \in \cQ}$ of $r$-cuts and show that they satisfy property~\ref{item:plan-P1} from the above plan.
We start by defining the notions of `low-degree' and `high-degree' graphs.

\begin{dfn}
  Denote by $\cQ_L$ the set of all graphs $Q \subseteq K_n$ with $\Delta(Q) \le \kappa np / \log n$ whose all vertices are coloured $1$.
  Further, denote by $\cQ_H$ the family of all graphs $Q \subseteq K_n$ that contain an independent set $X_Q \subseteq V^1(Q)$ of size $o(n)$ that dominates all edges of~$Q$ and such that every vertex $v \in X_Q$ has exactly $\eta np$ neighbours in each of the $r$ colour classes; we will refer to the vertices of $X_Q$ as the \emph{centre vertices} of $Q$ and denote their number by $k(Q)$.
  Finally, we let $\cQ \coloneqq \cQ_L \cup \cQ_H$.
\end{dfn}

The following two lemmas will help us choose a graph $Q \in \cQ_L \cup \cQ_H$ for every largest $H$-free subgraph $F \subseteq G$, provided that $G$ has some typical properties.  Given a graph $F$, a vertex $v \in V(F)$, and a subset $W$ of vertices of $F$, we will denote by $\deg_F(v, W)$ the number of neigbours of $v$ in $W$.

\begin{lemma}
  \label{lemma:Vizing}
  For every graph $I$ and integer $d \le \Delta(I)$, there is a subgraph $Q \subseteq I$ satisfying $\Delta(Q) = d$ and $e(Q) \ge \frac{d}{\Delta(I)+1} \cdot e(I)$.
\end{lemma}
\begin{proof}
  Denote $\Delta(I)$ by $D$.
  By Vizing's theorem, we can find a proper edge-colouring $\varphi \colon E(I) \to \br{D+1}$;
  without loss of generality,
  \begin{equation}
    \label{eq:Vizing-property}
    |\varphi^{-1}(1)| \ge \dotsb \ge |\varphi^{-1}(D + 1)|.
  \end{equation}
  Let $Q \coloneqq \varphi^{-1}(1) \cup \dotsb \cup \varphi^{-1}(d)$ and note that $e(Q) \ge \frac{d}{D + 1}  \cdot e(I)$ follows from~\eqref{eq:Vizing-property}.
  Further, $\Delta(Q) \le d$, as $\varphi$ induces a proper edge-colouring of $Q$ with $d$ colours.
  If $\Delta(Q) < d$, we can add to $Q$ further edges of $I$ until $\Delta(Q) = d$.
\end{proof}

\begin{lemma}
  \label{lemma:F-unfriendly-partition}
  If $\Pi = (V_1, \dotsc, V_r)$ is a largest $r$-cut in a graph $F$, then
  \begin{equation}
    \label{eq:F-unfriendly-partition}
    \deg_F(v, V_1) \le \deg_F(v, V_i) \text{ for all $v \in V_1$ and $i \in \br{r}$}.
  \end{equation}
\end{lemma}
\begin{proof}
  Indeed, if~\eqref{eq:F-unfriendly-partition} were not true for some $v \in V_1$ and $i \in \br{r}$, then moving a vertex $v$ from $V_1$ to $V_i$ would result in a cut $\Pi'$ with $e(\ext(\Pi') \cap F) > e(\ext(\Pi) \cap F)$, contradicting the maximality of $\Pi$.
\end{proof}

\begin{prop}
  \label{prop:Q-subgraph-F}
  Suppose that $F \subseteq K_n$ and let $\Pi_F = (V_1, \dotsc, V_r)$ be its canonical largest $r$-cut.  There exists $Q_F \subseteq F$, which we $\br{r}$-colour according to $\Pi_F$, such that one of the following holds:
  \begin{enumerate}[label={($\cQ_{\arabic*}$)}]
  \item
    \label{item:low-degree}
    $Q_F \in \cQ_L$ and $e(Q_F) \ge e(F[V_1])/2$;
  \item
    \label{item:medium-degree}
    $Q_F \in \cQ_L$ and $e(Q_F) \ge \max\left\{\frac{\kappa np}{\log n}, \frac{\kappa}{4\eta\log n} \cdot e(F[V_1])\right\}$;
  \item
    \label{item:high-degree}
    $Q_F \in \cQ_H$ and $k(Q_F) \ge e(F[V_1])/(16\Delta(F))$.
  \end{enumerate}
\end{prop}
\begin{proof}
  Let $I_L$ be a largest subgraph of $I \coloneqq F[V_1]$ of maximum degree at most $2\eta np$.
  We consider two cases:
  
  \smallskip
  \noindent
  \textit{Case 1. $e(I_L) \ge e(I)/2$.}
  Let $d \coloneqq \kappa np / \log n$.
  If $\Delta(I_L) \le d$, then we simply choose $Q_F \coloneqq I_L$, as it satisfies~\ref{item:low-degree}. Otherwise, \autoref{lemma:Vizing} applied to $I_L$ yields a subgraph $Q \subseteq I_L$ with $\Delta(Q) = d$ and
  \[
    e(Q) \ge \frac{\kappa}{2\eta \log n} \cdot e(I_L) \ge \frac{\kappa}{4\eta\log n} \cdot e(I).
  \]
  Since $e(Q) \ge \Delta(Q) = d$ as well, we may take $Q_F \coloneqq Q$, as it satisfies~\ref{item:medium-degree}.

  \smallskip
  \noindent
  \textit{Case 2. $e(I_L) < e(I)/2$.}
  Let $Y$ be the set of vertices whose $I$-degree is larger than $2 \eta np$ and let $\tilde{Q} \subseteq I$ be the graph containing all edges with at least one endpoint in $Y$.
  Since at least one endpoint of every edge of $I \setminus I_L$ lies in $Y$, we have
  \[
    e(\tilde{Q}) \ge \frac{1}{2} \sum_{v \in Y} \deg_{\tilde{Q}}(v) \ge \frac{1}{2} \sum_{v \in Y} \deg_{I \setminus I_L}(v) \ge \frac{e(I \setminus I_L)}{2} \ge \frac{e(I)}{4}.
  \]
  Let $\Sigma = (W_1, W_2)$ be a largest $2$-cut of $\tilde{Q}$.
  By maximality, for both $j \in \{1,2\}$, every vertex $v \in W_j$ has at least as many neighbours in $W_{3-j}$ as in $W_j$, see \autoref{lemma:F-unfriendly-partition}.
  In particular, for each $j \in \{1,2\}$, the subgraph of $\tilde{Q} \cap \ext(\Sigma)$ induced by $(W_j \cap Y, W_{3-j})$ is a union of stars of degree at least $\eta np$ each.
  Let $Q$ be the larger of these two graphs; without loss of generality, $Q$ is the graph induced by $(W_1 \cap Y, W_2)$.
  We have
  \[
    e(Q) \ge \frac{e_{\tilde{Q}}(W_1 \cap Y, W_2) + e_{\tilde{Q}}(W_1, W_2 \cap Y)}{2} \ge \frac{e(\tilde{Q} \cap \ext(\Sigma))}{2} \ge \frac{e(\tilde{Q})}{4} \ge \frac{e(I)}{16}.
  \]
  Finally, let $Q_F$ be the graph obtained by arbitrarily adding, for each $v \in W_1 \cap Y$, a set of $\eta np$ $Q$-neighbours of $v$ in $W_2 \subseteq V_1$ and, for each $i \in \{2, \dotsc, r\}$, a set of $\eta np$ $F$-neighbours of $v$ in $V_i$.
  This is possible, as assumption~\eqref{eq:F-unfriendly-partition} implies that, for every $v \in Y$ and $i \in \{2, \dotsc, r\}$,
  \[
    \deg_F(v,V_i) \ge \deg_F(v, V_1) = \deg_{\tilde{Q}}(v) \ge \deg_Q(v) \ge \eta np.
  \]
  Note that $Q_F \in \cQ_H$ (the dominating independent set is $X \coloneqq W_1 \cap Y$) and that $k(Q_F) = |W_1 \cap Y| \ge e(Q) / \Delta(Q) \ge re(Q)/\Delta(F) \ge re(I)/(16\Delta(F))$, where the second inequality follows as $\deg_F(v) = \sum_{i=1}^r \deg_F(v, V_i) \ge r \deg_Q(v)$ for every vertex $v$.
\end{proof}

We denote by $\cQ_L^1$ and $\cQ_L^2$ the sets of graphs corresponding to cases~\ref{item:low-degree} and~\ref{item:medium-degree} in the above proposition.  More precisely, let
\[
  \cQ_L^1 \coloneqq \{Q \in \cQ_L : e(Q) < \kappa np / \log n\}
  \quad
  \text{and}
  \quad
  \cQ_L^2 \coloneqq \{Q \in \cQ_L : e(Q) \ge \kappa np / \log n\}.
\]

\begin{dfn}
  \label{dfn:choice-of-d_Q}
  Define, for every $Q \in \cQ_L \cup \cQ_H$,  
  \[
    d_Q \coloneqq
    \begin{cases}
      8re(Q),\ & Q \in \cQ_L^1 \text{ and } p > \Cth p_H, \\
      \min\{\beta n^2p, \sqrt{\eta} e(Q) \log n\},\ & Q \in \cQ_L^2 \text{ or } (Q \in \cQ_L^1 \text{ and } p \le \Cth p_H), \\
      32k(Q)np ,\ &Q \in \cQ_H.
    \end{cases}
  \]
\end{dfn}

\begin{prop}
  If $G$ satisfies the assertion of \autoref{thm:Conlon-Gowers} and $\Delta(G) \le 2np$,
  then $d_{Q_F} \ge e(F \cap \int(\Pi_F))$ for every largest $H$-free $F \subseteq G$, provided $n$ is sufficiently large.
\end{prop}
\begin{proof}
  Let $F \subseteq G$ be a largest $H$-free subgraph of $G$.
  Since $\Pi_F = (V_1, \dotsc, V_r)$ has the largest value of $e(F[V_1])$ among all largest $r$-cuts of $F$, we have $e(F \cap \int(\Pi_F)) \le re(F[V_1])$.
  Moreover, since $G$ is assumed to satisfy the assertion of \autoref{thm:Conlon-Gowers}, we also know that $e(F \cap \int(\Pi_F)) \le \beta n^2p$.
  Consequently, it is enough to verify that $d_{Q_F} \ge re(F[V_1])$ under the assumption that $d_{Q_F} < \beta n^2p$.
  
  If $Q_F \in \cQ_L^1$, then $Q_F$ must satisfy~\ref{item:low-degree} in \autoref{prop:Q-subgraph-F} and thus
  \[
    d_{Q_F} \ge \min\left\{8re(Q_F), \sqrt{\eta} e(Q_F) \log n\right\} = 8re(Q_F) \ge 4re(F[V_1]),
  \]
  provided that $n$ is sufficiently large.
  If $Q_F \in \cQ_L^2$, then $Q_F$ must satisfy one of~\ref{item:low-degree} or~\ref{item:medium-degree} in \autoref{prop:Q-subgraph-F} and thus
  \[
    d_{Q_F} = \sqrt{\eta} \cdot e(Q_F) \log n \ge \sqrt{\eta} \cdot \min\left\{\frac{\kappa}{4\eta}, \frac{\log n}{2} \right\} \cdot e(F[V_1])
  \]
  and the desired inequality follows as $\eta \le \eta(\kappa)$ and $n$ is large.
  Finally, if $Q \in \cQ_H$, then
  \[
    d_{Q_F} = 32k(Q_F)np \ge 16k(Q_F) \Delta(Q_F) \ge re(F[V_1]),
  \]
  where the first inequality follows as $Q_F \subseteq F \subseteq G$ and $\Delta(G) \le 2np$.
\end{proof}

Finally, for every graph $Q \in \cQ_L \cup \cQ_H$, we let $\cC_Q$ be the family of all $\delta$-balanced cuts that are compatible with $Q$.  Since $Q_F$ inherits its $r$-colouring from $\Pi_F$, it is clearly compatible with this cut.  The following proposition shows that, asymptotically almost surely, $\Pi_F$ is also $o(1)$-balanced for every largest $H$-free subgraph $F$ of $G$.

\begin{prop}
  \label{prop:PiF-balanced}
  For every constant $\delta > 0$, a.a.s.\ the cut $\Pi_F$ is $\delta$-balanced for every largest $H$-free subgraph $F \subseteq G$.
\end{prop}
\begin{proof}
  Assume by contradiction that $\Pi_F$ is not $\delta$-balanced.
  It is easy to check that, for some positive $c_\delta$ that depends only on $\delta$,
  \[
    \ext(\Pi_F) \le \left(1-\frac{1}{r}-c_\delta\right)\binom{n}{2}.
  \]
  Furthermore,
  \begin{equation}
    \label{eq:F-upper}
    e(F) = e(F \cap \ext(\Pi_F)) + e(F \cap \int(\Pi_F)) \le e(G \cap \ext(\Pi_F)) + e(F \cap \int(\Pi_F)).
  \end{equation}
  Suppose now that assertions of \autoref{thm:Conlon-Gowers} and \autoref{lemma:cuts-concentration} hold in $G$.
  Since every $r$-partite subgraph of $G$ is $H$-free, letting $\Pi$ be some fixed largest $r$-cut of $K_n$, we have
  \[
    e(F) \ge e(G \cap \ext(\Pi)) \ge \ext(\Pi) \cdot p - o(n^2p) \ge \left(1 - \frac{1}{r} - o(1)\right)\binom{n}{2}p.
  \]
  On the other hand, by the assumed conclusion of \autoref{lemma:cuts-concentration},
  \[
    e(G \cap \ext(\Pi_F)) \le \ext(\Pi_F) \cdot p + o(n^2p) \le \left(1 - \frac{1}{r} - c_\delta+o(1)\right)\binom{n}{2}p;
  \]
  moreover, $e(F \cap \int(\Pi_F)) = o(n^2p)$ by the assumed conclusion of \autoref{thm:Conlon-Gowers}.
  Substituting these three estimates into~\eqref{eq:F-upper} yields a contradiction.
\end{proof}

\subsection{Large matchings in cuts with small deficits}
\label{sec:large-matchings-cuts}

In this section, we outline the proof of~\ref{item:plan-P2} from the above plan, which constitutes the bulk of this paper.
Since we can afford a union bound over all $Q \in \cQ$, we will treat $Q$ as fixed.

Let $\Pi \in \cC_Q$ be a $\delta$-balanced $r$-cut that is compatible with $Q$.
The probability of the event $\nu\big(\cF_Q[G \cap \ext(\Pi)]\big) \le d_Q$ can be bounded from above using a version of Janson's inequality that we state (and prove) as \autoref{thm:Janson-matchings} below in terms of the quantities $\mu_p(\cF_Q[\ext(\Pi)])$ and $\Delta_p(\cF_Q[\ext(\Pi)])$, defined in~\autoref{sec:Janson-inequality}, that should be familiar to all users of Janson's inequality.
Having said that, in order to obtain bounds on these parameters that would translate to an error probability that is sufficiently small to allow a union bound over all $Q \in \cQ$, we have to replace $\cF_Q$ by its (carefully chosen) subhypergraph $\cF_Q'$.
Whereas the construction of an appropriate $\cF_Q'$ in the `low-degree' case $Q \in \cQ_L$, presented in \autoref{sec:low-degree-case}, is rather straightforward, the construction of $\cF_Q'$ in the `high-degree' $Q \in \cQ_H$, presented in \autoref{sec:high-degree-case}, is a subtle argument that crucially uses the assumption that $G$ is a typical sample from $G_{n,p}$, which gives us control on the interactions between neighbourhoods of the vertices in $X_Q$ (see \autoref{lemma:tilde-G-hypergraph} below).
Our presentation here is inspired by~\cite{BalMorSamWar16, EngSamWar}.

A major challenge in executing the above strategy is that, unless $d_Q \gg n$, we cannot afford a union bound over all $\Pi \in \cC_Q$.
This is no accident -- when $Q$ is small and $p$ is not very close to one, one will find many cuts $\Pi \in \cC_Q$ for which $\cF_Q[G \cap \ext(\Pi)]$ is empty.
Luckily, \ref{item:plan-P2} does not require us to show that $\nu\big(\cF_Q[G \cap \ext(\Pi)]\big) > d_Q$ for all $\Pi \in \cC_Q$ but only for cuts $\Pi$ with small deficit.
This is where we employ and adapt the beautiful ideas of DeMarco and Kahn~\cite{DeMKah15Tur} that are centred around the concept of \emph{rigidity}.

Roughly speaking, a graph $G$ is \emph{rigid} if there is a collection $S = \{S_1, \dotsc, S_r\}$ of pairwise-disjoint vertex sets of size at least $(1-4r\alpha) \cdot n/r$ each such that $\ext(S) \subseteq \ext(\Pi)$ for \emph{every} maximum $r$-cut of $G$; a canonically chosen collection $S$ with this property is called the \emph{core} of $G$ and denoted by $\core(G)$.
An ingenious argument employing Harris's inequality, due to DeMarco and Kahn~\cite{DeMKah15Tur}, can be used to show that the probability that $G$ is rigid and $\nu\big(\cF_Q[G \cap \ext(\core(G))]\big) \le d_Q$ is not larger than the maximum value of the probability that $\nu\big(\cF_Q[G \cap \ext(S)]\big) \le d_Q$ over all collections $S$ of potential cores (i.e., collections of $r$ pairwise-disjoint sets of size at least $(1-4r\alpha) \cdot n/r$ each).
In other words, the DeMarco--Kahn correlation inequality reduces the problem of bounding the lower tail of $\nu\big(\cF_Q[G \cap \ext(\core(G))]\big)$ to the problem of bounding the lower tail of $\nu\big(\cF_Q[G \cap \ext(S)]\big)$ for a \emph{fixed} `almost' $r$-cut $S$.
\autoref{sec:rigidity-correlation} contains the formal definition of rigidity and the statement and proof of the DeMarco--Kahn correlation inequality.

Since $\ext(\core(G)) \subseteq \ext(\Pi)$, and thus $\cF_Q[G \cap \ext(\Pi)] \supseteq \cF_Q[G \cap \ext(\core(G))]$, for every largest $r$-cut $\Pi$ of $G$, the argument described in the previous paragraph suffices to control all cuts with zero deficit.
What to do about cuts with larger deficit?
Following DeMarco and Kahn~\cite{DeMKah15Tur}, we show that, asymptotically almost surely, $G$ has the following property:
For every $\Pi$ with $\deficit(\Pi; G) \le d_Q$, there are `many' graphs $G'$ that are rigid, have $\Pi$ among their largest $r$-cuts, and are `close' to $G$ in terms of the number of edges in their symmetric difference.
The precise notion of both `many' and `close' are too technical to be stated here, but the bottom line is that the notion of proximity implies that, for each $G'$,
\[
  \nu\big(\cF_Q[G \cap \ext(\Pi)]\big) \ge \nu\big(\cF_Q[G' \cap \ext(\core(G'))]\big) - O(d_Q)
\]
whereas the notion of multitude yields that the probability that $\nu\big(\cF_Q[G \cap \ext(\Pi)]\big) \le d_Q$ for \emph{some} cut $\Pi$ with $\deficit(\Pi; G) \le d_Q$ is bounded from above by
\[
  \Pr\left(\nu\big(\cF_Q[G' \cap \ext(\core(G'))]\big) \le O(d_Q)\right) \cdot \mathrm{err}(d_Q),
\]
where $\mathrm{err}(d_Q)$ is an error term that is, roughly speaking, of the form $\exp(O(d_Q))$.
\autoref{prop:algorithm-result} is a rather concise statement that encapsulates this.
Our proof of the proposition, which takes the bulk of \autoref{sec:algorithm}, is a very delicate switching argument that is based on the work of DeMarco and Kahn~\cite{DeMKah15Tur}, but departs from it in significant ways in order to provide a better estimate on $\mathrm{err}(d_Q)$.

Finally, the proof of \autoref{thm:main} culminates in \autoref{sec:proof}, where the various ingredients developed in earlier sections are finally mixed together.
Our argument there stumbles on yet another technical (described at the start of \autoref{sec:proof}) issue that occurs when $Q \in \cQ_L^2$ and $p > \Cth p_H$.  We solve this issue by replacing the family $\cF_Q'$ with its random sparsification, see \autoref{sec:sparsification}.  The final stretch is mere four pages of simple calculations.

\section{Tools and preliminaries}
\label{sec:tools-preliminaries}

\subsection{Janson's inequality}
\label{sec:Janson-inequality}

In order to successfully execute the plan sketched in~\autoref{section:outline}, we will need to bound the probability that a binomial random subset of vertices of a given hypergraph induces a subhypergraph that does not have a large matching.
We will derive a suitable lower-tail estimate for the matching number of a random induced subhypergraph, \autoref{thm:Janson-matchings} below, from the well-known inequality of Janson~\cite{Jan90};  our derivation follows the arguments of~\cite[Section~8.4]{AloSpe16}.

Given a set $V$ and a real $p \in [0,1]$, we denote by $V_p$ the random subset of $V$ obtained by independently retaining each element of $V$ with probability $p$.  Further, given a hypergraph $\cH$ with vertex set $V$, we define the following two quantities:
\[
  \mu_p(\cH) \coloneqq \sum_{A \in \cH} p^{|A|}
  \qquad
  \text{and}
  \qquad
  \Delta_p(\cH) \coloneqq \sum_{\substack{A, B \in \cH \\ A \neq B, A \cap B \neq \emptyset}} p^{|A \cup B|},
\]
where the second sum is over \emph{unordered} pairs of edges;  in other words, $\mu_p(\cH)$ is just the expected number of edges of $\cH[V_p]$ and $\Delta_p(\cH)$ is the expected number of pairs of distinct edges of $\cH[V_p]$ that intersect.

\begin{thm}[{Janson's inequality~\cite{Jan90}}]
  Let $\cH$ be a hypergraph on a finite vertex set $V$.  For all $p \in [0,1]$,
  \[
    \Pr(\cH[V_p] = \emptyset) \le \exp\big(-\mu_p(\cH)+\Delta_p(\cH)\big).
  \]
\end{thm}

Before we state our version of Janson's inequality for matchings, we recall a well-known lower-tail estimate for the Poisson distribution.

\begin{prop}
  \label{prop:lower-tail-Poisson}
  For every nonnegative real $\mu$ and all $\alpha \in [0,1]$,
  \[
    \Pr\big(\Pois(\mu) \le \alpha \mu\big) \le \exp\left(-\big(1-\alpha \log(e/\alpha)\big) \cdot \mu \right).
  \]
\end{prop}

\begin{thm}
  \label{thm:Janson-matchings}
  Let $\cH$ be a hypergraph on a finite vertex set $V$ and let $\alpha, \eta, p \in [0,1]$.
  Writing $\mu$ and $\Delta$ for $\mu_p(\cH)$ and $\Delta_p(\cH)$, we have
  \[
    \Pr\big(\nu(\cH[V_p]) \le \alpha \mu\big) \le \exp\big(-(1 - \alpha \log(e/\alpha) - \alpha p - \eta) \cdot \mu + (1 +2\alpha p/\eta) \cdot \Delta\big).
  \]
\end{thm}

\begin{proof}
  An edge $A \in \cH$ will be called \emph{$\eta$-good} if
  \[
    \delta(A) \coloneqq \sum_{\substack{B \in \cH \\ A \neq B, A \cap B \neq \emptyset}} p^{|B \setminus A|} \le \frac{2\Delta}{\eta \mu}.
  \]
  Let $\cH' \subseteq \cH$ be the hypergraph obtained by removing from $\cH$ all non-$\eta$-good edges.  Observe that
  \[
    2\Delta = \sum_{A \in \cH} p^{|A|} \delta(A) \ge \sum_{A \in \cH \setminus \cH'} p^{|A|} \delta(A) \ge \mu_p(\cH \setminus \cH') \cdot \frac{2\Delta}{\eta \mu},
  \]
  which implies that
  \[
    \mu_p(\cH') = \mu_p(\cH) - \mu_p(\cH \setminus \cH') \ge (1 - \eta) \mu.
  \]

  Denote by $\cM$ the collection of all matchings of size at most $\alpha \mu$ in $\cH'$.  Given an $M \in \cM$, let $P_M$ denote the probability that $M$ is a \emph{maximal} matching in $\cH'[V_p]$.  It follows from Janson's inequality that, letting $\cH_M' \coloneqq \cH' - \bigcup M$ be the collection of all edges of $\cH'$ that are disjoint from the union of all edges of $M$,
  \[
    \begin{split}
      P_M & \le \Pr\left(\bigcup M \subseteq V_p \text{ and } \cH_M'[V_p] = \emptyset\right) = \Pr\left(\bigcup M \subseteq V_p\right) \cdot \Pr(\cH_M'[V_p] = \emptyset)\\
      & = \prod_{A \in M} p^{|A|} \cdot \Pr(\cH_M'[V_p] = \emptyset) \le \prod_{A \in M} p^{|A|} \cdot \exp\left(-\mu_p(\cH_M') + \Delta_p(\cH_M')\right).
    \end{split}
  \]
  As, clearly, $\Delta_p(\cH_M') \le \Delta$ and
  \[
    \begin{split}
      \mu_p(\cH') - \mu_p(\cH_M') & \le \sum_{A \in M} \sum_{\substack{B \in \cH' \\ A \cap B \neq \emptyset}} p^{|B|} \le \sum_{A \in M} \left(p^{|A|} + \delta(A) \cdot \max_{\substack{B \in \cH' \\ A \cap B \neq \emptyset}} p^{|B \cap A|}\right) \\
      &  \le \sum_{A \in M} \left(p + \delta(A) \cdot p\right) \le |M| \cdot p \cdot \left(1+\frac{2\Delta}{\eta \mu}\right) \le \alpha p \cdot \mu + \frac{2 \alpha p}{\eta} \cdot \Delta.
    \end{split}
  \]
  We may thus conclude that, for every $M \in \cM$,
  \[
    P_M \le \prod_{A \in M} p^{|A|} \cdot \exp\left(-(1 - \alpha p - \eta) \cdot \mu + \left(1 + \frac{2\alpha p }{\eta}\right)\cdot \Delta\right).
  \]
  
  Since $\cH' \subseteq \cH$, the inequality $\nu(\cH[V_p]) \le \alpha \mu $ implies that some $M \in \cM$ must be a maximal matching in $\cH'[V_p]$.  Therefore, by the union bound,
  \[
    \Pr\big(\nu(\cH[V_p]) \le \alpha \mu\big) \le \sum_{M \in \cM} \prod_{A \in M} p^{|A|} \cdot \exp\left(-(1 - \alpha p - \eta) \cdot \mu + \left(1 + \frac{2\alpha p }{\eta}\right)\cdot \Delta\right).
  \]
  In order to complete the proof, it suffices to observe that
  \[
    \sum_{M \in \cM} \prod_{A \in M} p^{|A|} \le \sum_{s \le \alpha \mu} \frac{1}{s!} \left(\sum_{A \in \cH'} p^{|A|}\right)^s \le \sum_{s \le \alpha \mu} \frac{\mu^s}{s!} = e^{\mu} \cdot \Pr\big(\Pois(\mu) \le \alpha \mu\big)
  \]
  and invoke \autoref{prop:lower-tail-Poisson}.
\end{proof}

The following two simplified versions of \autoref{thm:Janson-matchings} will be more convenient to use in our arguments.  The first version, which yields an essentially optimal upper bound on the probability that $\nu(\cH[V_p]) \ll \mu_p(\cH)$ in the case where $\Delta_p(\cH) \ll \mu_p(\cH)$,  is obtained from \autoref{thm:Janson-matchings} by letting $\alpha = \eta/2 = \gamma^2$.  The second version gives a non-optimal upper bound on the probability that $\cH[V_p]$ contains no large matchings, but it is applicable in the cases where $\Delta_p(\cH) = \Omega(\mu_p(\cH))$.

\begin{cor}
  \label{cor:Janson-matchings}
  Let $\cH$ be a hypergraph on a finite vertex set $V$ and let $p \in [0,1]$.
  For every $\gamma \le 1/10$, writing $\mu$ and $\Delta$ for $\mu_p(\cH)$ and $\Delta_p(\cH)$, we have
  \[
    \Pr\big(\nu(\cH[V_p]) \le \gamma^2 \mu\big) \le \exp\big(-(1 - \gamma)\mu + 2 \Delta\big).
  \]
\end{cor}

\begin{cor}
  \label{cor:extended-Janson-matchings}  
  Let $\cH$ be a hypergraph on a finite vertex set $V$ and let $p \in [0,1]$.
  Writing $\mu$ and $\Delta$ for $\mu_p(\cH)$ and $\Delta_p(\cH)$ and letting $\Lambda \coloneqq \min\{\mu, \mu^2 / \Delta\}$, we have
  \[
    \Pr\big(\nu(\cH[V_p]) \le \Lambda/1000\big) \le \exp(-\Lambda/10).
  \]
\end{cor}
\begin{proof}
  Assume first that $\Delta \le \mu/4$.  Since $\Lambda \le \mu$, it follows from~\autoref{cor:Janson-matchings} that
  \[
    \Pr\big(\nu(\cH[V_p]) \le \Lambda/100\big)
    \le \exp(-9\mu/10 + 2\Delta)
    \le \exp(-2\mu/5) \le \exp(-\Lambda/10).
  \]
  Assume now that $\Delta > \mu/4$, so that $q \coloneqq \mu/(4\Delta) < 1$.  Let $\cH_q$ be the random hypergraph obtained from $\cH$ by retaining each of its edges independently with probability $q$ and define the random variables $\mu' \coloneqq \mu_p(\cH_q)$ and $\Delta' \coloneqq \Delta_p(\cH_q)$.  By~\autoref{cor:Janson-matchings}, 
  \begin{equation}
    \label{eq:Janson-random-subhypergraph}
    \Pr\big(\nu(\cH_q[V_p]) \le \mu'/100 \mid \cH_q \big) \le \exp(-9\mu'/10 + 2\Delta').
  \end{equation}
  The key observation is that
  \[
    \Ex\left[\frac{9\mu'}{10} - 2\Delta'\right] = \frac{9q\mu}{10} - 2q^2\Delta = \left(\frac{9}{40} - \frac{2}{16}\right) \cdot \frac{\mu^2}{\Delta} = \frac{\mu^2}{10\Delta} \ge \frac{\Lambda}{10}.
  \]
  Since $\nu(\cH_q[V_p]) \le \nu(\cH[V_p])$ with probability one, evaluating~\eqref{eq:Janson-random-subhypergraph} on the (nonempty) event that $\mu' \ge {9\mu'}{10} - 2\Delta' \ge \Lambda/10$ gives the claimed inequality.
\end{proof}

\subsection{An upper-tail inequality}
\label{sec:upper-tail-inequality}

Our second tool is an upper-tail inequality for the number of edges that a binomial random subset induces in a uniform hypergraph.  Roughly speaking, it states that the probability that a $p$-random subset of vertices of an $n$-vertex, $\ell$-uniform hypergraph with $o(n^\ell)$ edges induces $\Omega(n^\ell p^\ell)$ edges is exponentially small in $np$.
Even though an estimate on this probability can be easily derived from~\cite[Lemma~3.6]{BalMorSamWar16}, which is a respective estimate for the hypergeometric distribution, we supply a simple, self-contained proof that is based on an elegant argument of Janson, Oleszkiewicz, and Ruci\'nski~\cite{JanOleRuc04} (see also~\cite{HarMouSam22}).

\begin{lemma}
  \label{lemma:hypergraphs_neighbors}
  For every $\alpha > 0$ and positive integer $\ell$, there is a $\rho > 0$ such that the following holds for every $n$-vertex, $\ell$-uniform hypergraph $\cH$ with $e(\cH) \le \rho n^\ell$.  For every $p \in [0,1]$, letting $R \sim V(\cH)_p$, we have
  \[
    \Pr\left(e(\cH[R]) \ge \alpha \cdot (np)^\ell\right) \le \exp\left(-\rho np\right).
  \]
\end{lemma}

\begin{proof}
  Given an $\alpha > 0$ and a positive integer $\ell$, set $\rho \coloneqq \frac{\min\{\alpha,1\}}{(2\ell+1)e}$ and let $\cH$ be an $n$-vertex, $\ell$-uniform hypergraph with $e(\cH) \le \rho n^\ell$.  Suppose that $R \sim V(\cH)_p$ and let $X \coloneqq e(\cH[R])$.  If we denote, for each $A \in \cH$, the indicator random variable of the event $A \subseteq R$ by $Y_A$, we may write, for every positive integer $k$,
  \begin{equation}
    \label{eq:X-moment}
    \begin{split}
      \Ex[X^k] & = \sum_{A_1, \dotsc, A_k \in \cH} \Ex[Y_{A_1} \dotsb Y_{A_k}] \\
      & = \sum_{A_1, \dotsc, A_{k-1} \in \cH} \Ex[Y_{A_1} \dotsb Y_{A_{k-1}}] \sum_{A_k \in \cH} \Ex[Y_{A_k} \mid Y_{A_1} \dotsb Y_{A_{k-1}} = 1] \\
      & = \sum_{A_1, \dotsc, A_{k-1} \in \cH}  \Ex[Y_{A_1} \dotsb Y_{A_{k-1}}] \cdot \Ex[X \mid Y_{A_1} \dotsb Y_{A_{k-1}} = 1] \\
      & \le \Ex[X^{k-1}] \cdot \max_{A_1, \dotsc, A_{k-1} \in \cH} \Ex[X \mid Y_{A_1} \dotsb Y_{A_{k-1}} = 1].
    \end{split}
  \end{equation}
  The following claim bounds the conditioned expectation in the above inequality:
  \begin{claim}
    \label{claim:high-moment-bound-X}
    For every $S \subseteq V(\cH)$ of size at most $\rho \ell np$,
    \[
      \Ex[X \mid S \subseteq R] \le (2\ell+1)\rho \cdot (np)^\ell.
    \]
  \end{claim}
  \begin{proof}
    Let $S$ be an arbitrary set of at most $\rho \ell n p$ vertices of $\cH$.  Since every $i$-element subset of $S$ is contained in at most $n^{\ell-i}$ edges of $\cH$, we have
    \[
      \Ex[X \mid S \subseteq R] - \Ex[X ]\le \sum_{i=1}^\ell \binom{|S|}{i} n^{\ell - i} p^{\ell - i} \le (np)^\ell \cdot \sum_{i = 1}^\ell \left(\frac{|S|}{np}\right)^i \le \frac{\rho \ell}{1-\rho \ell} \cdot (np)^\ell,
    \]
    where the last inequality follows from our assumption that $|S| \le \rho \ell np$.  Since $\rho \le 1/(2\ell)$, we have
    \[
      \Ex[X \mid S \subseteq R] \le \Ex[X] + \frac{\rho \ell}{1-\rho \ell} \cdot (np)^\ell \le e(\cH) \cdot p^\ell + 2\rho \ell \cdot (np)^\ell \le (2\ell+1) \rho \cdot (np)^\ell,
    \]
    as claimed.
  \end{proof}
  A straightforward inductive argument turns~\eqref{eq:X-moment} and \autoref{claim:high-moment-bound-X} into the bound
  \[
    \Ex[X^k] \le \big((2\ell+1) \rho\big)^k \cdot (np)^{\ell k},
  \]
  valid for all $k \le \lceil \rho n p \rceil$.  By Markov's inequality, letting $k = \lceil \rho n p\rceil$,
  \[
    \Pr\left(X \ge \alpha \cdot (np)^\ell\right) \le \frac{\Ex[X^k]}{\alpha^k \cdot (np)^{k\ell}} \le \left(\frac{(2\ell+1)\rho}{\alpha}\right)^k \le \exp(-\rho np),
  \]
  as claimed.
\end{proof}

\section{Typical properties of random graphs}
\label{sec:random-graphs}

In this section, we establish several key properties concerning the distribution of edges and of copies of various subgraphs of $H$ that the binomial random graph $G_{n,p}$ possesses with probability very close to one.
We begin with a non-probabilistic inequality that proves useful while estimating various expectations throughout the paper.

\begin{lemma}
  \label{lemma:2-balanced-condition}
  Let $H$ be a $2$-balanced graph and suppose that $p \ge C n^{-1/m_2(H)}$ for some $C > 0$.  Then,
  $n^{v_{H'}-2} p^{e_{H'}-1} \ge C^{e_{H'} - 1}$ for every nonempty subgraph $H' \subseteq H$.  Moreover, if $H$ is strictly $2$-balanced, then there exists a $\lambda > 0$ that depends only on $H$ such that $n^{v_{H'}-2} p^{e_{H'}-1} \ge C^{e_{H'} - 1} n^{\lambda}$ for every $H' \subseteq H$ with $1 < e_{H'} < e_H$.
\end{lemma}
\begin{proof}
  Let $H'$ be a nonempty subgraph of $H$.  Since the assertion of the lemma holds vacuously if $e_{H'} = 1$, we may assume that $e_{H'}>1$.  By our assumption on $p$,
  \[
    n^{v_{H'}-2} p^{e_{H'}-1} \ge n^{v_{H'}-2} \left(C n^{-\frac{1}{m_2(H)}}\right)^{e_{H'}-1} = C^{e_{H'} - 1} n^{v_{H'}-2 - \frac{e_{H'}-1}{m_2(H)}}.
  \]
  Now, recall that the definition of $m_2(H)$ implies that $v_{H'} - 2 \ge \frac{e_{H'}-1}{m_2(H)}$ and that, when $H$ is strictly $2$-balanced, this inequality is strict unless $H' = H$. Thus, the exponent of $n$ is non-negative and it is positive if $H$ is strictly $2$-balanced and $1 < e_{H'} < e_H$.
\end{proof}

\subsection{Distribution of edges}
\label{sec:distribution-edges}

The next lemma states that the number of edges in the intersection of $G_{n,p}$ with the graphs $\ext(\Pi)$ and $\int(\Pi)$ is typically close to its expectation not only for every fixed $r$-cut $\Pi$, but in fact for all sequences $\Pi$ of pairwise disjoint sets, irrespective of their length, simultaneously.

\begin{lemma}
  \label{lemma:cuts-concentration}
  If $p \gg n^{-1} \log n$, then a.a.s.\ $G \sim G_{n,p}$ satisfies the following for all $m$ and every sequence of pairwise disjoint sets $\Pi = (V_1, \dotsc, V_m)$:
  \[
    e(G \cap \ext(\Pi)) = \left(\ext(\Pi) \pm o(n^2)\right) \cdot p
    \quad
    \text{and}
    \quad
    e(G \cap \int(\Pi)) = \left(\int(\Pi) \pm o(n^2)\right) \cdot p.
  \]
\end{lemma}
\begin{proof}
  Note first that, for every sequence $\Pi = (V_1, \dotsc, V_m)$ of pairwise disjoint sets of vertices, we have $\ext(\Pi) \cup \int(\Pi) = \int(\{V_1 \cup \dotsb \cup V_m\})$.
  Further, there are at most $(n+1)^n$ different such sequences.
  Consequently, it suffices to show that, for each sequence $\Pi$ and all fixed $\zeta > 0$,
  \[
    \Pr\left(\big|e(G \cap \int(\Pi)) - e(\int(\Pi)) \cdot p\big| \ge \zeta n^2 p\right) \le n^{-2n}
  \]
  for all sufficiently large $n$.
  Since $e(G \cap \int(\Pi)) \sim \Bin(e(\int(\Pi)), p)$, the above estimate easily follows from standard tail estimates for binomial random variables and our assumption that $n^2 p \gg n \log n$.
\end{proof}

\subsection{The number of subgraphs leaning on an edge or a vertex}
\label{sec:numb-subgr-lean}

Our next lemma is an upper-tail estimate for the number of copies of all graphs of the form $H \setminus f$, where $f \in H$, that contain a given edge of (resp.\ a given vertex of) $G_{n,p}$.
Its proof utilises the high-moment argument of Janson, Oleszkiewicz, and Ruci\'nski~\cite{JanOleRuc04}.
We recall that $\partial_e$ and $\partial$ bind stronger than the operation of taking induced subhypergraphs and thus $\partial\partial_e\cH[G_{n,p}] = (\partial\partial_e\cH)[G_{n,p}]$ and $\partial \cH_v [G_{n,p}] = (\partial \cH_v)[G_{n,p}]$.

\begin{lemma}
  \label{lemma:partial}
  Let $H$ be a strictly $2$-balanced graph and suppose that $p \ge n^{-1/m_2(H)}$.
  \begin{enumerate}[{label=(\roman*)}]
  \item
    \label{item:partial-edge}
    For every edge $e \in K_n$, letting $\cH$ denote all copies of $H$ in $K_n$,
    \[
      \Pr\left(\big|\partial\partial_e\cH[G_{n,p}]\big| \ge 4e_H^2n^{v_H-2}p^{e_H-2}\right) \le n^{-6}.
    \]
  \item
    \label{item:partial-vertex}
    For every vertex $v$, letting $\cH_v$ denote all copies of $H$ in $K_n$ that contain $v$,
    \[
      \Pr\left(\big|\partial\cH_v[G_{n,p}]\big| \ge 2v_He_Hn^{v_H-1}p^{e_H-1}\right) \le n^{-6}.
    \]
  \end{enumerate}
\end{lemma}
\begin{proof}
  Denote, for every edge $e \in K_n$ and every vertex $v \in V(K_n)$,
  \[
    X_e \coloneqq \big|\partial\partial_e\cH[G_{n,p}]\big|
    \qquad
    \text{and}
    \qquad
    X_v \coloneqq \big|\partial\cH_v[G_{n,p}]\big|
  \]
  and note that $\Ex[X_e] \le 2e_H^2n^{v_H-2}p^{e_H-1}$ and $\Ex[X_v] \le v_He_Hn^{v_H-1}p^{e_H-1}$.
  In particular, it suffices to show that the probability that either of $X_e$ or $X_v$ exceeds twice its expectation is at most $n^{-6}$.
  Since Markov's inequality implies that, for every positive integer $\ell$ and each nonnegative random variable $X$,
  \[
    \Pr\big(X \ge 2\Ex[X]\big) = \Pr\big(X^\ell \le (2\Ex[X])^\ell\big)\le \frac{\Ex[X^\ell]}{(2\Ex[X])^\ell},
  \]
  it suffices to show that $\Ex[X_e^\ell] \le (3\Ex[X_e]/2)^{\ell}$ for every $e \in K_n$ and all $\ell = O(\log n)$, and that the analogous bound holds for each $X_v$.

  To this end, for every $\omega \subseteq K_n$, write $Y_\omega$ for the indicator random variable of the event that $\omega \subseteq G_{n,p}$.
  We may now write
  \[
    X_e = \sum_{\omega \in \partial\partial_e \cH}Y_{\omega}
    \qquad
    \text{and}
    \qquad
    X_v = \sum_{\omega \in \partial \cH_v} Y_{\omega},
  \]
  and, further, for every positive integer $\ell$ and every edge $e$,
  \begin{equation}
    \label{eq:Xe-moment}
    \begin{split}
      \Ex[X_e^\ell] &= \sum_{\omega_1, \dotsc, \omega_\ell \in \partial\partial_e \cH} \Ex[Y_{\omega_1} \dotsm Y_{\omega_\ell}] \\
      &= \sum_{\omega_1, \dotsc, \omega_{\ell-1} \in \partial\partial_e\cH} \Ex[Y_{\omega_1} \dotsm Y_{\omega_{\ell-1}}] \sum_{\omega_\ell \in \partial\partial_e\cH} \Ex[Y_{\omega_\ell} \mid Y_{\omega_1} \dotsm Y_{\omega_{\ell-1}} = 1] \\
      &= \sum_{\omega_1, \dots, \omega_{\ell-1} \in \partial\partial_e\cH} \Ex[Y_{\omega_1} \dotsm Y_{\omega_{\ell-1}}] \cdot \Ex[X_e \mid \omega_1 \cup \dotsb \cup \omega_{\ell-1} \subseteq G_{n,p}];
    \end{split}
  \end{equation}
  similarly, for every vertex $v$,
  \begin{equation}
    \label{eq:Xv-moment}
    \Ex[X_v^\ell] = \sum_{\omega_1, \dotsc, \omega_{\ell-1} \in \partial\cH_v}
    \Ex[Y_{\omega_1} \dotsb Y_{\omega_{\ell-1}}] \cdot \Ex[X_v \mid \omega_1 \cup \dotsb \cup \omega_{\ell-1} \subseteq G_{n,p}].
  \end{equation}
  Our next two claims bound the conditioned expectations in the above sums.
  \begin{claim}
    \label{claim:Xe-seed}
    For any edge $e$ and all $G \subseteq K_n$ with $e(G) = O(\log n)$, 
    \[
      \Ex[X_e \mid G \subseteq G_{n,p}] - \Ex[X_e] \ll \Ex[X_e].
    \]
  \end{claim}
  \begin{claim}
    \label{claim:Xv-seed}
    For any vertex $v$ and all $G \subseteq K_n$ with $e(G) = O(\log n)$, 
    \[
      \Ex[X_v \mid G \subseteq G_{n,p}] - \Ex[X_v] \ll \Ex[X_v].
    \]
  \end{claim}
  \begin{proof}[Proof of \autoref{claim:Xe-seed}]
    Let $G$ be an arbitrary graph with $O(\log n)$ edges.
    For every nonempty $J \subseteq H$, set 
    \[
      \Omega_J \coloneqq \{\omega \in \partial\partial_e\cH \colon (\omega \cap G) \cup e \cong J\}
    \]
    and note that
    \[
      \Ex[X_e \mid G \subseteq G_{n,p}] - \Ex[X_e] \le \sum_{\substack{J \subseteq H \\ 1 < e_J < e_H}} |\Omega_J| \cdot p^{e_H - e_J -1}.
    \]
    It is not hard to see that, for every nonempty $J \subseteq H$,
    \[
      |\Omega_J| \le O\left(e(G)^{e_J - 1} n^{v_H - v_J}\right) = O\left(n^{v_H-v_J} (\log n)^{e_H}\right);
    \]
    indeed, every $\omega \in \Omega_J$ intersects $G$ in $e_J-1$ edges.  We conclude that
    \[
      \Ex[X_e \mid G \subseteq G_{n,p}] - \Ex[X_e] \le n^{v_H-2} p^{e_H-2} \cdot  \sum_{\substack{J \subseteq H \\ 1 < e_J < e_H}} n^{2 - v_J} p^{1 - e_J} \cdot O\big((\log n)^{e_H}\big).
    \]
    Finally, since $H$ is strictly $2$-balanced and $p \ge n^{-1/m_2(H)}$, \autoref{lemma:2-balanced-condition} implies that each term in the above sum is at most $n^{-\lambda}$, for some positive constant $\lambda$ that depends only on $H$.  This implies the claimed estimate, as $\Ex[X_e] = \Theta\left(n^{v_H-2} p^{e_H-2}\right)$.
  \end{proof}
  \begin{proof}[Proof of \autoref{claim:Xv-seed}]
    Let $G$ be an arbitrary graph with $O(\log n)$ edges.
    For every nonempty $J \subseteq H$, set 
    \[
      \Omega_J \coloneqq \{\omega \in \partial\cH_v \colon \omega \cap G \cong J\},
    \]
    where we always include $v$ in the vertex set of $\omega \cap G$ (so that $J$ might contain one isolated vertex).
    Note that
    \[
      \Ex[X_v \mid G \subseteq G_{n,p}] - \Ex[X_v] \le \sum_{\substack{J \subseteq H \\ 1 \le e_J < e_H}} |\Omega_J| \cdot p^{e_H - e_J -1}
    \]
    and, similarly as before, for every $J \subseteq H$,
    \[
      |\Omega_J| \le O\left(e(G)^{e_J} n^{v_H - v_J}\right) = O\left(n^{v_H-v_J} (\log n)^{e_H}\right).
    \]
    We conclude that
    \[
      \Ex[X_v \mid G \subseteq G_{n,p}] - \Ex[X_v] \le n^{v_H-1} p^{e_H-1} \cdot  \sum_{\substack{J \subseteq H \\ 1 < e_J < e_H}} n^{1 - v_J} p^{-e_J} \cdot O\big((\log n)^{e_H}\big).
    \]
    Finally, since $H$ is $2$-balanced and $p \ge n^{-1/m_2(H)}$, \autoref{lemma:2-balanced-condition} implies that each term in the above sum is at most $(np)^{-1}$.
    Since $np \ge n^{1-1/m_2(H)} \ge n^{\lambda}$ for some positive constant $\lambda$ that depends only on $H$ and since $\Ex[X_v] = \Theta\left(n^{v_H-1} p^{e_H-1}\right)$, the claimed estimate follows.
  \end{proof}

  By \autoref{claim:Xe-seed} and~\eqref{eq:Xe-moment}, we conclude that, for every $e \in K_n$ and all $\ell = O(\log n)$,
  \[
    \Ex[X_e^\ell] \le \sum_{\omega_1, \dots, \omega_{\ell-1} \in \partial\partial_e \cH} \Ex\left[Y_{\omega_1} \dotsm Y_{\omega_{\ell-1}}\right] \cdot (1+o(1)) \Ex[X_e]
    \le 3/2 \cdot \Ex[X_e^{\ell-1}] \cdot \Ex[X_e];
  \]
  this implies, via induction on $\ell$, that $\Ex[X_e^\ell] \le (3/2)^\ell \cdot \Ex[X_e]^\ell$ holds for all $\ell = O(\log n)$.
  Analogously, \autoref{claim:Xv-seed} and~\eqref{eq:Xv-moment} yield $\Ex[X_v^\ell] \le (3/2)^\ell \cdot \Ex[X_v]^{\ell}$ for every $v \in V(K_n)$ and all $\ell = O(\log n)$.
\end{proof}

\subsection{Interactions between neighbourhoods}
\label{sec:inter-betw-neighb}

Our next lemma concerns interactions between neighbourhoods of the centre vertices of graphs $Q \in \cQ_H$ such that $Q \subseteq G_{n,p}$.  Roughly speaking, it states that (a.a.s.) these neighbourhoods do not overlap significantly.  The precise statement involves a uniform hypergraph whose edges are subsets of the neighbourhoods of the centre vertices of a $Q \in \cQ_H$ that contain a prescribed number of vertices from every colour class of $Q$.  Suppose that $\vell = (\ell_1, \dotsc, \ell_r) \in \NN^r$ is a vector with $\ell \coloneqq \ell_1 + \dotsb + \ell_r \ge 1$.  For a graph $Q \in \cQ_H$ and a centre vertex $v \in X_Q$, denote by $\cN_Q^\vell(v)$ all the $\ell$-element sets $U$ of vertices such that
\[
  |U \cap N_Q(v) \cap V^k(Q)| = \ell_k \text{ for every } k \in \br{r}.
\]
The following lemma constructs a subhypergraph $\cG \subseteq \cN_Q^\vell \coloneqq \bigcup_{v \in X_Q} \cN_Q^\vell (v)$ whose edges are `well-distributed'.  For a nonnegative integer $j$, we write $\Delta_j(\cG)$ for the maximum number of edges of $\cG$ that contain any given set of $j$ vertices; in particular $\Delta_0(\cG) = e(\cG)$.

\begin{lemma}
  \label{lemma:tilde-G-hypergraph}
  For every $\vell = (\ell_1, \dotsc, \ell_r) \in \NN^r$ with $\ell \coloneqq \ell_1 + \dotsb + \ell_r \ge 1$, there exist positive constants $c$ and $C \ge 1$ that depend only on $\ell$ and $\eta$ such that a.a.s.\ the following holds for every $Q \in \cQ_H$ with $Q \subseteq G_{n,p}$, provided that $p \gg \log n / n$.   There is a hypergraph $\cG \subseteq \cN_Q^\vell$ with the following properties:
  \begin{enumerate}[{label=(\roman*)}]
  \item
    \label{item:tilde-G-hypergraph-edges}
    $e(\cG) \ge c \cdot \min\left\{n^\ell, k(Q) \cdot (np)^\ell \right\}$,
  \item
    \label{item:tilde-G-hypergraph-degrees}
    $\Delta_j(\cG) \le \max\left\{4(np)^{\ell-j},\  C \cdot \frac{e(\cG)}{n^j}\right\}$ for every $j \in \{0, \dotsc, \ell\}$.
  \end{enumerate}
\end{lemma}

We will use the following simple corollary of \autoref{lemma:tilde-G-hypergraph} to estimate the number of vertices in each $Q \in \cQ_H$ that is contained in $G_{n,p}$.

\begin{cor}
  \label{cor:Q-vmin}
  There exists a positive constant $c$ that depends only on $\eta$ such that when $p \gg \log n / n$, then a.a.s., for every $Q \in \cQ_H$ with $Q \subseteq G_{n,p}$ and every $k \in \br{r}$,
  \[
    |V^k(Q)| \ge \Big|\bigcup_{v \in V(Q)} N_Q(v) \cap V^k(Q)\Big| \ge c \cdot \min\{n, k(Q) \cdot np\}.
  \]
\end{cor}

\begin{proof}[Proof of~\autoref{lemma:tilde-G-hypergraph}]
  Given an ordered sequence $v_1, \dotsc, v_k$ of vertices, we will categorise them as \emph{good} or \emph{bad}, based only on the knowledge of their neighbourhoods in $G_{n,p}$.  This will be done in such a way that, for every $Q \in \cQ_H$ with $Q \subseteq G_{n,p}$ whose centre vertices are $v_1, \dotsc, v_k$, every good vertex will contribute at least $(\eta np/\ell)^\ell/2$ edges to $\cG$, unless $\cG$ already has $\Omega(n^{\ell})$ edges.  We will assume that
  \begin{align}\label{eq:G-maximum-ell-degree-bound}
    \Delta(G_{n,p})^\ell \le 2(np)^\ell,
  \end{align}
  which holds in $G_{n,p}$ a.a.s.
  We first define a few constants.  Let $\rho$ be the constant whose existence is asserted by~\autoref{lemma:hypergraphs_neighbors} invoked with $\alpha \coloneqq (\eta/\ell)^\ell/2$ and set $D \coloneqq 2^{\ell+1}/\rho$.
  
  The first vertex $v_1$ is good.  Fix some $i \in \br{k-1}$ and suppose that we have already categorised $v_1, \dotsc, v_i$ as good or bad.  Let $W_i$ denote the set of good vertices among $\{v_1, \dotsc, v_i\}$.  Define
  \[
    \cGh_i \coloneqq \bigcup_{v \in W_i} \binom{N_{G_{n,p}}(v)}{\ell}
  \]
  and, for every $j \in \br{\ell-1}$,
  \begin{align*}
    \Delta_j^{(i)} & \coloneqq \max\left\{2(np)^{\ell-j},\  \frac{De(\cGh_i)}{n^j}\right\}, \\
    M_j^{(i)} & \coloneqq \left\{ T \in \binom{\br{n}}{j} \colon \deg_{\cGh_i}T \ge \Delta_j^{(i)}\right\}.
  \end{align*}
  We further define the `closure' $\cGb_i$ of $\cGh_i$ by
  \begin{align*}
    \cGb_i \coloneqq \cGh_i \cup \bigcup_{j = 1}^{\ell-1} \left\{U \in \binom{\br{n}}{\ell} \colon U \supseteq T \text{ for some } T \in M_j^{(i)}\right\}.
  \end{align*}
  We call $v_{i+1}$ a good vertex if
  \begin{align}
    \label{align:good-vertex-inequality}
    \left|\binom{N_{G_{n,p}}(v_{i+1}) \setminus \{v_1, \dotsc, v_i\}}{\ell} \cap \cGb_i \right| \le \frac{1}{2} \left(\frac{\eta np}{\ell}\right)^\ell;
  \end{align}
  otherwise, call $v_{i+1}$ a bad vertex.  We stress again that the property that $v_{i+1}$ is a good/bad vertex depends only on the neighbourhoods in $G_{n,p}$ of the previously considered vertices $v_1, \dotsc, v_i$ (and their order).

  \begin{claim}
    \label{claim:cGb-versus-W}
    For every $i \in \br{k}$, either $e(\cGb_i) < \rho n^\ell$ or $|W_i| \ge \rho / (4p^\ell)$.
  \end{claim}
  \begin{proof}
    We first observe that, for every $j \in \br{\ell-1}$,
    \[
      e(\cGh_i) \ge \frac{|M_j^{(i)}|}{\binom{\ell}{j}} \cdot \Delta_j^{(i)} \ge \frac{|M_j^{(i)}|}{\binom{\ell}{j}} \cdot \frac{D e(\cGh_i)}{n^j},
    \]
    which yields the bound $|M_j^{(i)}| \le \binom{\ell}{j}n^j/D$.  Since every element of $M_j^{(i)}$ contributes at most $n^{\ell-j}$ edges to $\cGb_i$, we have
    \[
      \begin{split}
        e(\cGb_i) & \le |W_i| \cdot \binom{\Delta(G_{n,p})}{\ell} + \sum_{j=1}^{\ell - 1} |M_j^{(i)}| \cdot n^{\ell - j} \le |W_i| \cdot 2(np)^{\ell} + \sum_{j=1}^{\ell-1} \binom{\ell}{j} \cdot \frac{n^\ell}{D} \\
        & \le 2|W_i| \cdot (np)^{\ell} + \frac{2^\ell}{D} \cdot n^\ell = 2|W_i| \cdot (np)^\ell + \frac{\rho}{2} n^\ell.
      \end{split}
    \]    
    Finally, if $e(\cGb_i) \ge \rho n^\ell$, then the above inequality yields $|W_i| \ge \rho / (4p^\ell)$.
  \end{proof}

  \begin{claim}
    \label{claim:good-vertices}
    Asymptotically almost surely, for every $k$ and each sequence $v_1, \dotsc, v_k$ of distinct vertices of $G_{n,p}$, at least $\min\{k/2, \rho/(4p^\ell)\}$ of them are good.
  \end{claim}
  \begin{proof}
    Fix an arbitrary sequence $v_1, \dotsc, v_k$ of distinct vertices of $G_{n,p}$.  For each $i \in \br{k}$, let $\Sigma_i$ denote the $\sigma$-algebra generated by exposing the neighbourhoods of $v_1, \dotsc, v_i$.  Observe that conditioned on $\Sigma_i$, the set $N_{G_{n,p}}(v_{i+1}) \setminus \{v_1, \dotsc, v_i\}$ is still a $p$-random subset of $\br{n} \setminus \{v_1, \dotsc, v_i\}$.  We may thus use \autoref{lemma:hypergraphs_neighbors} to conclude that, for every $i$,
    \begin{equation}
      \label{eq:Pr-vi-bad}
      \Pr\big(\text{$v_{i+1}$ is bad and $e(\cGb_i) < \rho n^\ell$} \mid \Sigma_i \big) \le \exp(-\rho n p).
    \end{equation}

    Let $\cB$ denote the event that the sequence $v_1, \dotsc, v_k$ contains fewer than the claimed number of good vertices.  In particular, there are at least $k/2$ bad vertices among them and, for each $i \le k$, we have $|W_i| < \rho/(4p^\ell)$ and thus $e(\cGb_i) < \rho n^\ell$, by \autoref{claim:cGb-versus-W}.  By the union bound,
    \[
      \Pr(\cB) \le \sum_{\substack{I \subseteq \br{k-1} \\ |I| \ge k/2}} \Pr(\text{$v_{i+1}$ is bad and $e(\cGb_i) < \rho n^{\ell}$ for all $i \in I$}).
    \]
    Since the event that $v_{i+1}$ is bad and $e(\cGb_i) < \rho n^\ell$ is $\Sigma_{i+1}$-measurable, we may use the chain rule and the estimate~\eqref{eq:Pr-vi-bad} to bound each term in the sum from above by $\exp(-\rho |I| n p )$ and conclude that
    \[
      \Pr(\cB) \le 2^k \cdot \exp(-\rho k n p / 2) \le n^{-2k},
    \]
    where the last inequality follows from our assumption that $p \gg \log n / n$.  Taking the union bound over all $k$ and all sequences gives the desired result.
  \end{proof}

  Fix some $Q \in \cQ_H$, let $k \coloneqq k(Q)$, and let $v_1, \dotsc, v_k$ be an arbitrary ordering of the centre vertices of $Q$.  We will show that under the assumption that at least $\min\{k/2, \rho/(4p^\ell)\}$ of those $k$ vertices are good (which, by \autoref{claim:good-vertices}, holds simultaneously for all $Q \in \cQ_H$), we may build a hypergraph $\cG \subseteq \cN_Q^\vell$ with the desired properties.  We will do so iteratively, by first defining a sequence $\cG_0, \dotsc, \cG_k$ and then setting $\cG \coloneqq \cG_k$.

  First, we let $\cG_0$ be the empty hypergraph.  Second, suppose that $i \in \br{k}$ and that we have already defined $\cG_0, \dotsc, \cG_{i-1}$.  If $v_i$ is a good vertex, we obtain $\cG_i$ by adding to $\cG_{i-1}$ some $\frac{1}{2} \left(\frac{\eta np}{\ell}\right)^\ell$ sets $U \subseteq N_{G_{n,p}}(v_i) \setminus \{v_1, \dotsc, v_k\}$ that belong to $\cN_Q^\vell(v_i) \setminus \cGb_{i-1}$ (where we assume that $\cGb_0$ is empty).  Note that this is possible, since
  \[
    |\cN_Q^\vell(v_i)| = \prod_{j=1}^r \binom{|N_Q(v_i) \cap V^j(Q)|}{\ell_j} \ge \prod_{j=1}^r \left(\frac{\eta np}{\ell}\right)^{\ell_j} = \left(\frac{\eta np}{\ell}\right)^{\ell}
  \]
  and the assumption that $v_i$ is a good vertex implies that at most half of those sets belong to $\cGb_{i-1}$, see~\eqref{align:good-vertex-inequality}.  Finally, if $v_i$ is a bad vertex, we simply let $\cG_i \coloneqq \cG_{i-1}$.
  
  Since $\cG_{i-1} \subseteq \cGh_{i-1} \subseteq \cGb_{i-1}$ for every $i$, each good vertex contributes at least $\frac{1}{2}\left(\frac{\eta np}{\ell}\right)^\ell$ distinct edges to $\cG$, which means that $e(\cG_i) \ge \frac{|W_i|}{2}\left(\frac{\eta np}{\ell}\right)^\ell$.  This fact and our assumed lower bound on the number of good vertices proves assertion~\ref{item:tilde-G-hypergraph-edges} of the lemma with a~positive constant $c$ that depends only on $\eta$, $\ell$, and $\rho = \rho(\eta, \ell)$.  Further, by \eqref{eq:G-maximum-ell-degree-bound}, we have $e(\cGh_i) \le |W_i| \cdot 2(np)^\ell \le 4 \left(\frac{\ell}{\eta}\right)^\ell \cdot e(\cG_i)$.  Hence,
  \[
    \Delta_j^{(i)} \le \max\left\{2(np)^{\ell-j},\  \frac{4D \ell^\ell}{\eta^\ell} \cdot \frac{e(\cG_i)}{n^j}\right\}
  \]
  for each $i \in \br{k}$ and all $j \in \br{\ell-1}$.  When we add edges to $\cG_{i-1}$ to form $\cG_i$, we only increase the degrees of sets that do not belong to $\bigcup_{j=1}^{\ell - 1} M_j^{(i)}$.  Moreover, since we add only subsets of the neighbourhood of $v_i$ in $G_{n,p}$, assumption~\eqref{eq:G-maximum-ell-degree-bound} implies that, for every $T \in \binom{\br{n}}{j}$, we increase the degree of $T$ by at most $2(np)^{\ell - j}$. Thus, for every $i \in \br{k}$ and every $j \in \br{\ell - 1}$,
  \[
    \Delta_j(\cG_i) \le \Delta_j^{(i-1)} + 2(np)^{\ell-j} \le 2\Delta_j^{(i)}.
  \]
  This immediately gives~\ref{item:tilde-G-hypergraph-degrees} from the statement of the lemma with a~constant $C$ that depends only on $\ell$, $\eta$, $D = D(\ell, \rho)$, and $\rho = \rho(\eta, \ell)$; indeed, the claimed upper bounds on $\Delta_0(\cG) = e(\cG)$ and $\Delta_\ell(\cG) = 1$ hold vacuously.
\end{proof}

\subsection{Summary}
\label{sec:summary}

Suppose, as in \autoref{section:outline}, that $G \sim G_{n,p}$ for some $p \ge (1+\eps)p_H$.
We are finally ready to define the notion of `pseudo-randomness' that we will assume our graph $G$ to have.
We let $\cT$ be the intersection of the following events:
\begin{enumerate}[{label=(T\arabic*)}]
\item
  $\deg_G(v) = (1 \pm o(1)) np$ for every $v \in V(G)$;
\item
  $e(G \cap \int(\Pi)) = e(\int(\Pi)) \cdot p \pm o(n^2p)$ and $e(G \cap \ext(\Pi)) = e(\ext(\Pi)) \cdot p \pm o(n^2p)$ for every family $\Pi$ of pairwise-disjoint sets of vertices (see \autoref{lemma:cuts-concentration});
\item
  $\big|\partial\partial_e\cH[G]\big| \le 4e_H^2n^{v_H-2}p^{e_H-2}$ for every $e \in K_n$ and $\big|\partial\cH_v[G]\big| \le 2v_He_Hn^{v_H-1}p^{e_H-1}$ for every vertex $v$ (see \autoref{lemma:partial});
\item
  $e(F \cap \int(\Pi_F)) \le \beta n^2p$ for every largest $H$-free subgraph $F \subseteq G$ (see \autoref{thm:Conlon-Gowers});
\item
  $|V^k(Q)| \ge c \cdot \min \{n, k(Q) \cdot np\}$ for every $Q \in \cQ_H$ such that $Q \subseteq G$ and every $k \in \br{r}$ (see \autoref{cor:Q-vmin});
\item
  $G$ satisfies the assertion of \autoref{lemma:tilde-G-hypergraph}, and thus also the assertion of \autoref{lemma:high-degree-bounds};
\item
  if $p > 1/(10r^3)$, then $e(G \cap \ext(\{A,B\})) = (1+o(1))\ext(\{A,B\})p$ for all pairs of disjoint sets $A$ and $B$ with $|A||B| \gg n$.
\end{enumerate}
It is clear that $G \in \cT$ asymptotically almost surely.

\section{Low-degree case}
\label{sec:low-degree-case}

Recall that $\cH$ denotes the family of all copies of $H$ in $K_n$.
Given a graph $Q \subseteq K_n$, define
\[
  \cHQL \coloneqq \{A \in \cH : e(A \cap Q) = 1 \text{ and } e(Q[V(A)]) = 1\}.
\]
In other words, $\cHQL$ comprises only those copies of $H$ that contain exactly one edge of $Q$ and whose vertex set induces no additional edges of $Q$.  Note that this definition guarantees that $A \setminus Q \neq B \setminus Q$ whenever $A$ and $B$ are two distinct elements of $\cHQL$.  In particular, if we define
\[
  \cFQL \coloneqq \{A \setminus Q : A \in \cHQL\},
\]
then, for each graph $\omega \in \cFQL$, there is a unique $A \in \cHQL$ with $\omega = A \setminus Q$; we will denote this $A$ with $\omega^+$.

\begin{lemma}
  \label{lemma:FQL-lower-bound}
  For every $r$-tuple $S = (S_1, \dotsc, S_r)$ of pairwise-disjoint sets of vertices and every graph $Q$ with $V(Q) \subseteq S_1$, we have
  \[
    |\cFQL[\ext(S)]| \ge \left(1 - \frac{v_H^2 \cdot \Delta(Q)}{|S_1|-v_H}\right) \cdot e(Q) \cdot N(H, K^+_S),
  \]
  where $K_S^+ \coloneqq \ext(S) \cup uv$, where $u$ and $v$ are two distinct vertices of $S_1$.
\end{lemma}
\begin{proof}
  Let $e$ be a uniformly chosen random edge of $Q$ and let $G \coloneqq \ext(S) \cup e$.
  Note that $G \cong K_S^+$, as we assumed that $V(Q) \subseteq S_1$.  Further, let $A$ be a uniformly chosen random copy of $H$ in the random graph $G$; note that $e \in A$, as $H$ is edge-critical and $\chi(H)=r+1$.  Observe that $A \setminus e \in \cFQL[\ext(S)]$ if and only if $e$ is the only edge that $R \coloneqq V(A) \cap S_1$ induces in $Q$ and thus
  \[
    |\cFQL[\ext(S)]| = \Pr\big(e(Q[R]) = 1\big) \cdot e(Q) \cdot N(H, K_S^+).
  \]
  It now suffices to bound the above probability from below.

  To this end, note that $2 \le |R| \le v_H$ and that, conditioned on the event $|R|=i$, the set $R$ is a uniformly random $i$-element subset of $S_1$ containing $e$.  Fix some $i$, condition on the event $|R| = i$, let $e = \{w_1, w_2\}$, and let $\{w_3, \dotsc, w_i\}$ be a random ordering of the elements of $R \setminus e$.  The key observation is that, for every $j \ge 3$, the probability that $w_j$ is adjacent (in $Q$) to one of $w_1, \dotsc, w_{j-1}$ is $|N_Q(\{w_1, \dotsc, w_{j-1}\})| / (|S_1| - j + 1)$.  Consequently,
  \[
    \Pr\big(e(Q[R]) = 1 \mid |R| = i\big) \ge 1 - \sum_{j=3}^{i} \frac{(j-1) \cdot \Delta(Q)}{|S_1|-j+1} \ge 1 - \frac{v_H^2 \cdot \Delta(Q)}{|S_1|-v_H}.
  \]
  Since the final lower bound above holds for every value of $i$, it holds without the conditioning as well.
\end{proof}

\begin{lemma}
  \label{lemma:FQL-bounds}
  There exists a positive constant $\Clow$ such that the following holds for every $r$-tuple $S = (S_1, \dotsc, S_r)$ of pairwise-disjoint sets of vertices of size at least $n/(2r)$ each, every graph $Q \in \cQ_L$ with $V(Q) \subseteq S_1$, and all $p \ge n^{-1/m_2(H)}$:
  \begin{align*}
    \mu_p\big(\cFQL[\ext(S)]\big) & \ge (\pi_H-o(1)) \cdot e(Q) \cdot \big(\textstyle{\min_i|S_i|}\big)^{v_H-2} \cdot p^{e_H-1}, \\
    \frac{\Delta_p(\cFQL)}{\mu_p\big(\cFQL[\ext(S)]\big)} & \le \frac{\Clow\kappa\cdot n^{v_H-2}p^{e_H-1}}{\log n},
  \end{align*}
  where $\pi_H$ is the constant defined in~\eqref{eq:pi_H}.
\end{lemma}

\begin{proof}
  We will abbreviate $\cFQL$ by $\cF$ and $\cFQL[\ext(S)]$ by $\cF(S)$.  For an edge $e \in Q$, define
  \[
    \cF_e \coloneqq \{\omega \in \cF : \omega^+  = \omega \cup e\},
  \]
  so that $\cF \coloneqq \bigcup_{e \in Q} \cF_e$.  \autoref{lemma:FQL-lower-bound} and the assumption that $|S_1| \ge n/(2r)$ and $\Delta(Q) = o(n)$ imply that
  \begin{equation}
    \label{eq:mu-low-degree}
    \mu \coloneqq \mu_p(\cF(S)) = |\cF(S)| \cdot p^{e_H-1} \ge (1-o(1)) \cdot e(Q) \cdot N(H, K_S^+) \cdot p^{e_H-1},
  \end{equation}
  where $K_S^+$ is the graph defined in the statement of \autoref{lemma:FQL-lower-bound}.
  Further, since $K_S^+$ contains a subgraph isomorphic to $K_r(\min_i |S_i|)^+$, we have
  \[
    N(H, K_S^+) \ge N\big(H, K_r(\textstyle{\min_i |S_i|})^+\big) = (\pi_H-o(1)) \cdot \big(\textstyle{\min_i |S_i|}\big)^{v_H-2},
  \]
  where the last equality holds as $\min_i|S_i| \ge n/(2r)$, see the definition of $\pi_H$ given in~\eqref{eq:pi_H}.

  We now turn to bounding the correlation term $\Delta \coloneqq \Delta_p(\cF)$.
  By definition,
  \begin{equation}
    \label{eq:Delta-low-degree}
    \begin{split}
      \Delta & =\sum_{\omega_1 \in \cF} \sum_{\substack{\omega_2 \in \cF \setminus \{\omega_1\} \\ \omega_1 \cap \omega_2 \neq \emptyset}} p^{2(e_H-1)-e(\omega_1 \cap \omega_2)}  \\
      & \le |\cF| \cdot p^{2(e_H-1)} \cdot \max_{\omega_1 \in \cF} \sum_{\substack{\omega_2 \in \cF \setminus \{\omega_1\} \\ \omega_1 \cap \omega_2 \neq \emptyset}} p^{-e(\omega_1 \cap \omega_2)}.
    \end{split}
  \end{equation}
  Further, for every $\omega_1 \in \cF$, letting $e_1 \in Q$ be the unique edge such that $\omega_1 \in \cF_{e_1}$,
  \begin{equation}
    \label{eq:max-omega1}
    \sum_{\substack{\omega_2 \in \cF \setminus \{\omega_1\} \\ \omega_1 \cap \omega_2 \neq \emptyset}} p^{-e(\omega_1 \cap \omega_2)}
    = \sum_{\emptyset \neq H' \subsetneq H} \sum_{\substack{\omega_2 \in \cF \setminus \{\omega_1\} \\ \omega_1^+ \cap \omega_2^+ \cong H'}} p^{-e_{H'}+\1[\omega_2 \in \cF_{e_1}]}.
  \end{equation}

  The following claim will allow us to bound the right-hand side of the above inequality.

  \begin{claim}\label{omega_2_count}
    For each $e_1 \in Q$, every $\omega_1 \in \cF_{e_1}$, and all $H' \subseteq H$:
    \begin{enumerate}[label=(\roman*)]
    \item\label{item:omega-2-common-Q}
      $\left|\left\{\omega_2 \in \cF_{e_1} : \omega_1 \cap \omega_2 \neq \emptyset \wedge \omega_1^+ \cap \omega_2^+ \cong H' \right\}\right| \le  \1[e_{H'} > 1] \cdot  O\left(n^{v_H-v_{H'}}\right)$.
    \item
      $\left|\left\{\omega_2 \in \cF \setminus \cF_{e_1} : \omega_1 \cap \omega_2 \neq \emptyset \wedge \omega_1^+ \cap \omega_2^+ \cong H' \right\}\right| \le  O\left(\Delta(Q) \cdot n^{v_H-v_{H'}-1}\right).$
    \end{enumerate}
  \end{claim}

  \begin{proof}
    Fix some $e_1 \in Q$, $\omega_1 \in \cF_{e_1}$, and $H' \subseteq H$.  To see the first item, note that the set in the left-hand side of~\ref{item:omega-2-common-Q} is nonempty only if $e_{H'} > 1$; indeed, if $\omega_1 \cap \omega_2 \neq \emptyset$ for some $\omega_2 \in \cF_{e_1}$, then $\omega_1^+ \cap \omega_2^+$ contains $e_1$ and at least one more edge.  We may thus assume that $e_{H'} > 1$ and count the number of $\omega_2 \in \cF_{e_1}$ such that $\omega_1 \cap \omega_2 \neq \emptyset$ and $\omega_1^+ \cap \omega_2^+ \cong H'$ as follows:  First, there are at most $v_H^{v_{H'}-2}$ ways to choose the $v_{H'} - 2$ vertices of $\omega_1$ that certify $\omega_1^+ \cap \omega_2^+ \cong H'$.  Second, there are at most $n^{v_H - v_H'}$ options for the remaining vertices of $\omega_2$.  Finally, there are $O(1)$ choices to assign the roles that these vertices play in the copy of $H$.
    
    For the second item, we separately count those $\omega_2$ that belong to $\cF_{e_2}$ for some edge $e_2 \in Q \setminus \{e_1\}$ that shares an endpoint with $e_1$ and those $\omega_2$ that belong to $\cF_{e_2}$ for some edge $e_2 \in Q$ that is disjoint from $e_1$.

    There are at most $2\Delta(Q)$ edges $e_2$ of the former type.  For each such edge, we have at most $v_H^{v_{H'}-1}$ options for the $v_{H'} - 1$ vertices of $\omega_1$ that certify $\omega_1^+ \cap \omega_2^+ \cong H'$.  Moreover, there are at most $n^{v_H - v_{H'} - 1}$ ways to choose the remaining vertices of $\omega_2$.  Finally, there are $O(1)$ choices for the roles that these vertices play in the copy of $H$.  Thus, total the number of choices for $\omega_2$ in this case is $O\left(\Delta(Q) \cdot n^{v_H - v_{H'} - 1}\right)$. 
    
    There are at most $e(Q)$ edges $e_2$ that are disjoint from $e_1$.  For each such edge, there are now at most $v_H^{v_{H'}}$ options to choose the $v_{H'}$ vertices of $\omega_1$ that certify $\omega_1^+ \cap \omega_2^+ \cong H$ and at most $n^{v_H - v_{H'} - 2}$ options to choose the remaining vertices of $\omega_2$.  Finally, as before, there are $O(1)$ choices for the roles that these vertices play in the copy of $H$. Thus, the total number of choices for $\omega_2$ in this case is $O\left(e(Q) \cdot n^{v_H - v_{H'} - 2}\right)$.  The claimed bound follows as clearly $e(Q) \le \Delta(Q) n$.
  \end{proof}

  By~\eqref{eq:mu-low-degree}, \eqref{eq:Delta-low-degree}, \eqref{eq:max-omega1}, and \autoref{omega_2_count}, for some constant $C$ that depends only on $H$,
  \[
    \frac{\Delta}{\mu} \le C \cdot \frac{|\cF|}{|\cF(S)|} \cdot \sum_{\emptyset \neq H' \subsetneq H} \frac{n^{v_H-2} p^{e_H-1}}{n^{v_{H'}-2}p^{e_{H'}-1}} \cdot \left(\frac{\Delta(Q)}{np} + \1[e_{H'}>1]\right).
  \]
  \autoref{lemma:2-balanced-condition} and our assumption that $p \ge n^{-1/m_2(H)}$, imply that $n^{v_{H'}-2} p^{e_{H'}-1} \ge 1$ for all nonempty $H' \subseteq H$ and, further, that there exists a constant $\lambda > 0$ depending only on $H$ such that $n^{v_{H'}-2} p^{e_{H'}-1} \ge n^{\lambda}$ for all $H' \subsetneq H$ with $e_{H'} > 1$.
  Further, we clearly have $|\cF| \le e(Q) \cdot e_H n^{v_H-2}$, whereas we have shown above that $|\cF(S)| \ge e(Q) \cdot (\pi_H-o(1)) n^{v_H-2}$.
  We may thus conclude that, for some constant $C'$ that depends only on $H$,
  \[
    \frac{\Delta}{\mu} \le C' \cdot n^{v_H - 2} p^{e_H - 1} \cdot \left(\frac{\Delta(Q)}{np} + n^{-\lambda}\right)
    \le \frac{2\kappa C' \cdot n^{v_H-2}p^{e_H-1}}{\log n},
  \]
  where the final inequality holds as $\Delta(Q) \le \kappa np / \log n$ for every $Q \in \cQ_L$.
\end{proof}

\section{High-degree case}
\label{sec:high-degree-case}

Suppose that $Q \in \cQ_H$.  We start by defining a collection $\cHQH$ of copies of $H$ in $K_n$.
Let $X = X_Q$ be the set of centre vertices of $Q$ and, for every $v \in X$ and $k \in \br{r}$, write $N^k(v) \coloneqq N_Q(v) \cap V^k(Q)$ to denote the neighbours of $v$ that are coloured $k$.
Fix an edge $f \in H$ such that $\chi(H \setminus f) = r$, denote by $h$ one of the endpoints of $f$, and let $\ell \coloneqq \deg_Hh$ be the degree of $h$ in $H$.
Let $\varphi \colon H \setminus f \to \br{r}$ be a proper colouring of $H \setminus f$ such that $\varphi(h) = 1$.
We let $\cHQH$ be the collection of all copies of $H$ in $K_n$ that are constructed as follows:
\begin{enumerate}[{label=(\roman*)}]
\item
  map $h$ to some vertex $v \in X$,
\item
  map the $\ell$ neighbours of $h$ in $H$ into the sets $N^1(v), \dotsc, N^r(v)$, accordingly with the colouring $\varphi$,
\item
  map the remaining $v_H - \ell-1$ vertices of $H$ arbitrarily to the complement of $X$.
\end{enumerate}
Further, define
\[
  \cFQH \coloneqq \{A \setminus Q : A \in \cHQH\} \subseteq \cF_Q.
\]

\begin{lemma}\label{lemma:high-degree-bounds}
  There exists a positive constant $\chigh$ that depends only on $H$ and $\eta$ such that the following holds for each graph $Q \in \cQ_H$ such that $Q \subseteq G$ for some $G \in \cT$.
  There exists a family $\cF \subseteq \cFQH$ such that, for every $r$-tuple $(S_1, \dotsc, S_r)$ of pairwise disjoint sets that is compatible with $Q$ and satisfies $\min_j |S_j| \ge n/(2r)$,
  \begin{align*}
    \mu_p(\cF[\ext(S)]) & \ge \chigh \cdot \min\big\{k(Q) \cdot n^{v_H-1}p^{e_H}, n^2p\big\} \gg k(Q) \cdot np, \\
    \frac{\mu_p^2(\cF[\ext(S)])}{\Delta_p(\cF)} & \ge \chigh \cdot \min\big\{k(Q) \cdot n^{v_H-1}p^{e_H}, k(Q) \cdot n^{1+\lambda_H}p, n^2p\big\} \gg k(Q) \cdot np,
  \end{align*}
  where $\lambda_H$ is a positive constant that depends only on $H$.
\end{lemma}
\begin{proof}
  Fix some $Q \in \cQ_H$ and set $k = k(Q)$.  Let $\vell \coloneqq \big(|\varphi^{-1}(j) \cap N_H(h)|\big)_{j=1}^r$ and let $\cG \subseteq \cN_Q^\vell$ be the hypergraph from the statement of \autoref{lemma:tilde-G-hypergraph};
  we may find such a $\cG$ due to our assumption that $Q \subseteq G$ for some $G \in \cT$.
  Further, denote the two constants from the statement of \autoref{lemma:tilde-G-hypergraph} by $c_{\ref{lemma:tilde-G-hypergraph}}$ and $C_{\ref{lemma:tilde-G-hypergraph}}$ and observe that they both depend only on $\eta$ and $H$.  Recall from the definition of $\cN_Q^\vell$ that, for every edge $U \in \cG$, there is a $v \in X_Q$ that is adjacent in $Q$ to all $\ell$ vertices of $U$; denote some such $v$ by $v_U$ and let $E_h(U)$ be the set of $\ell$ edges connecting $v_U$ to the vertices of~$U$.  Finally, set
  \[
    \cF \coloneqq \{ \omega \subseteq K_n \setminus Q : \omega \, \dot\cup\, E_h(U) \cong H \text{ for some } U \in \cG \} \subseteq \cFQH.
  \]
  It is not hard to see that, for some absolute positive constants $c'$ and $c''$,
  \[
    |\cF[\ext(S)]| \ge c' \cdot e(\cG) \cdot \left(\min_j |S_j| - v_H - k - \eta np\right)^{v_H-\ell-1} \ge c'' \cdot e(\cG) \cdot n^{v_H - \ell - 1},
  \]
  and, consequently,
  \[
    \mu_p(\cF[\ext(S)]) = \sum_{\omega \in \cF} p^{e(\omega)} \ge c'' \cdot e(\cG) \cdot n^{v_H-\ell-1} p^{e_H-\ell}.
  \]
  Further, we claim that
  \[
    \begin{split}
      \Delta_p(\cF) & = \sum_{\substack{\omega, \omega' \in \cF \\ \omega \cap \omega' \ne \emptyset}} \mathbf{1}_{\omega \cup \omega' \subseteq G_{n,p}} \\
      & \le C' \cdot e(\cG) \sum_{j=0}^{\ell} \Delta_j(\cG) \sum_{\substack{\emptyset \ne H' \subseteq H - h\\ |V(H') \cap N_H(h)| = j}} n^{2v_H-2\ell-2-v_{H'}+j} p^{2e_H-2\ell-e_{H'}},
  \end{split}
  \]
  where $C'$ is a constant that depends only on $H$.  Indeed, we may enumerate all pairs of intersecting elements of $\cF$ as follows.  First, we choose an edge $U \in \cG$ that contains the images of the $\ell$ neighbours of $h$ in $H$ in its copy $\omega$.  Second, we choose the size $j \in \{0, \dotsc, \ell\}$ of the intersection of this set with the respective set in $\omega'$; there are at most $\binom{\ell}{j} \cdot \Delta_j(\cG)$ edges $U' \in \cG$ with $|U \cap U'| = j$.  Third, we choose a subgraph $H' \subseteq H$ which is isomorphic to the intersection of $\omega$ and $\omega'$.  Fourth, we then choose the $v_H - \ell - 1$ remaining vertices of $\omega$ and the $v_H -\ell - 1 - (v_{H'} - j)$ remaining vertices of $\omega'$.  Note that the union of $\omega$ and $\omega'$ has exactly $2(e_H - \ell) - e_{H'}$ edges.  The constant factor $C'$ accounts for the choices of the roles that the vertices of $\omega \cap \omega'$ play in $H'$.
  
  \begin{claim}
    For some positive constant $c$ that depends only on $\eta$ and $H$,
    \[
      \mu_p\big(\cF[\ext(S)]\big) \ge c \cdot \min\big\{kn^{v_H-1}p^{e_H}, n^2p\big\}.
    \]
  \end{claim}
  \begin{proof}
    It suffices to show that, for some $c > 0$ that depends only on $\eta$ and $H$,
    \[
      e(\cG) \cdot n^{v_H - \ell - 1} p ^{e_H - \ell} \ge c \cdot \min\big\{kn^{v_H-1}p^{e_H}, n^2p\big\}.
    \]
    We consider two cases, depending on which of the two terms achieves the minimum in \autoref{lemma:tilde-G-hypergraph}~\ref{item:tilde-G-hypergraph-edges}.  First, if $e(\cG) \ge c_{\ref{lemma:tilde-G-hypergraph}}k(np)^\ell$, then
    \[
      e(\cG) \cdot n^{v_H - \ell - 1} p^{e_H - \ell} \ge c_{\ref{lemma:tilde-G-hypergraph}}k \cdot n^{v_H-1} p^{e_H}.
    \]
    Second, if $e(\cG) \ge c_{\ref{lemma:tilde-G-hypergraph}}n^\ell$, then
    \[
      e(\cG) \cdot n^{v_H - \ell - 1} p^{e_H - \ell} \ge c_{\ref{lemma:tilde-G-hypergraph}} n^{v_H - 1} p^{e_H - \ell}.
    \]
    Since $v_H-1$ and $e_H-\ell$ are the numbers of vertices and edges of the graph $H-h \subseteq H$, \autoref{lemma:2-balanced-condition} implies that
    \[
      n^{v_H - 1} p^{e_H - \ell} \ge n^2 p.\qedhere
    \]
  \end{proof}

  \begin{claim}
    For some positive constants $c$ that depends only on $\eta$ and $H$,
    \[
      \frac{\mu_p\big(\cF[\ext(S)]\big)^2}{\Delta_p(\cF)} \ge c \cdot \min\big\{kn^{v_H-1}p^{e_H}, kn^{1+\lambda}p, n^2p\},
    \]
    where $\lambda$ is the constant from the statement of \autoref{lemma:2-balanced-condition}.
  \end{claim}
  \begin{proof}
    We first note that, for some positive constant $c'$,
    \[
      \frac{\mu_p\big(\cF[\ext(S)]\big)^2}{\Delta_p(\cF)} \ge \frac{c' \cdot  e(\cG)}{\max \left\{\Delta_j(\cG) \cdot n^{j-v_{H'}} p^{-e_{H'}} : \emptyset \ne H' \subseteq H - h \text{ with } |V(H') \cap N_H(h)| = j\right\}}.
    \]
    Fix some nonempty $H' \subseteq H - h$ and an integer $j \in \{0, \dotsc, \ell\}$.  We consider two cases, depending on which of the two terms achieves the maximum in \autoref{lemma:tilde-G-hypergraph}~\ref{item:tilde-G-hypergraph-degrees}.  First, if $\Delta_j(\cG) \le C_{\ref{lemma:tilde-G-hypergraph}} e(\cG)/n^j$, then, by \autoref{lemma:2-balanced-condition},
    \[
      \frac{e(\cG)}{\Delta_j(\cG) n^{j-v_{H'}} p^{-e_{H'}}} \ge \frac{n^{v_H'} p^{e_H'}}{C_{\ref{lemma:tilde-G-hypergraph}}} \ge \frac{n^2p}{C_{\ref{lemma:tilde-G-hypergraph}}}.
    \]
    We may thus suppose that $e(\cG)/n^j \le C_{\ref{lemma:tilde-G-hypergraph}} e(\cG) / n^j < \Delta_j(\cG) \le 4(np)^{\ell-j}$; note that this implies that $j \ge 1$, as $\Delta_0(\cG) = e(\cG)$.  We consider two further subcases, depending on which of the two terms achieves the minimum in \autoref{lemma:tilde-G-hypergraph}~\ref{item:tilde-G-hypergraph-edges}.  If $e(\cG) \ge c_{\ref{lemma:tilde-G-hypergraph}}k(np)^\ell$, then
    \[
      \frac{e(\cG)}{\Delta_j(\cG) n^{j-v_{H'}} p^{-e_{H'}}} \ge \frac{c_{\ref{lemma:tilde-G-hypergraph}}kn^{v_{H'}} p^{e_{H'}+j}}{4}.
    \]
    We now note that $v_{H'}+1$ and $e_{H'}+j$ are the numbers of vertices and edges of the graph $H' + h \subseteq H$.  If $v_{H'} + 1 = v_H$ and $e_{H'}+j = e_H$, then there is nothing left to prove.  Otherwise, the graph $H' + h$ is a proper subgraph of $H$ with at least two edges (as $H' \neq \emptyset$ and $j \ge 1$) and thus \autoref{lemma:2-balanced-condition} implies that $n^{v_{H'}} p^{e_{H'}+j} \ge n^{1+\lambda}p$ for some $\lambda > 0$ that depends only on $H$.  Finally, if $e(\cG) \ge c_{\ref{lemma:tilde-G-hypergraph}}n^\ell$, then
    \[
      \frac{e(\cG)}{\Delta_j(\cG) n^{j-v_{H'}} p^{-e_{H'}}} \ge \frac{c_{\ref{lemma:tilde-G-hypergraph}}n^{v_{H'}} p^{e_{H'}+j-\ell}}{4} \ge \frac{c_{\ref{lemma:tilde-G-hypergraph}}n^{v_{H'}}p^{e_{H'}}}{4} \ge \frac{c_{\ref{lemma:tilde-G-hypergraph}}n^2p}{4},
    \]
    where the last inequality follows from \autoref{lemma:2-balanced-condition}.
  \end{proof}
  In order to complete the proof of the lemma, we just point out that, by \autoref{lemma:2-balanced-condition} and our the assumption that $k = o(n)$,
  \[
    \min\big\{k(Q) \cdot n^{v_H-1}p^{e_H}, k(Q) \cdot n^{1+\lambda}p, n^2p\big\} \gg k(Q) \cdot np.\qedhere
  \]
\end{proof}

\section{Rigidity and correlation}
\label{sec:rigidity-correlation}

Recall that, for a graph $G$, a set of cuts $\cC$, and $\Pi \in \cC$, we defined
\begin{align*}
  b_{\cC}(G) &\coloneqq \max_{\Pi' \in \cC} e(\ext(\Pi') \cap G), \\
  \deficit_{\cC}(\Pi;G) &\coloneqq b_{\cC}(G) - e(\ext(\Pi) \cap G).
\end{align*}
Additionally, we now also set
\[
  \maxcut_{\cC}(G) \coloneqq \{\Pi \in \cC \colon e(G \cap \ext(\Pi)) = b_{\cC}(G)\}.
\]
Following DeMarco and Kahn~\cite{DeMKah15Tur}, we define an equivalence relation $\equiv_{\cC, G}$ on $V(G)$ by:
\begin{align*}
    x \equiv_{\cC,G} y \quad \Longleftrightarrow \quad \text{$x$ and $y$ are in the same part for every cut in $\maxcut_{\cC}(G)$}.
\end{align*}
The equivalence classes of $\equiv_{\cC, G}$ will be called $(\cC, G)$-\textit{components}, or simply \textit{components} if the identities of $\cC$ and $G$ are clear.
The following definition, again borrowed from DeMarco and Kahn~\cite{DeMKah15Tur}, is key.

\begin{dfn}
  Given a family $\cC$ of $r$-cuts, and a graph $G$, we say that $G$ is \emph{$\cC$-rigid} if the number of equivalent pairs of vertices under $\equiv_{\cC, G}$ is at least $(1-\alpha)n^2/(2r)$.
\end{dfn}

A key property of graphs $G$ that are rigid with respect to a family $\cC$ of balanced cuts is that all cuts $\Pi \in \maxcut_{\cC}(G)$ agree on most of the vertices of $G$.  This statement is made precise by the following proposition, which is essentially~\cite[Proposition~10.1]{DeMKah15Tur}.

\begin{prop}\label{prop:rigid-has-core}
  Suppose that $\cC$ is a family of balanced $r$-cuts.
  If $G$ is $\cC$-rigid, then there are distinct $(\cC, G)$-components $S_1, \dotsc, S_r$, each of size greater than $(1-4r\alpha) \cdot n/r$.
\end{prop}
\begin{proof}
  Let $\balpha \coloneqq 4r\alpha$ and let $\lambda n$ be the number of vertices of $G$ that belong to $(\cC, G)$-components of size at most $(1-\balpha) \cdot n/r$.
  Further, denote by $P$ the number of $\equiv_{\cC, G}$-equivalent pairs of vertices.
  Since $\cC$ contains only $\delta$-balanced cuts, no $(\cC, G)$-component can be larger than $(1+\delta) \cdot n/r$.
  Consequently,
  \[
    (1-\alpha) \frac{n^2}{r} \le 2P \le (1-\lambda)n \cdot (1+\delta)\frac{n}{r} + \lambda n \cdot (1-\balpha)\frac{n}{r} = (1+\delta-\lambda(\delta + \balpha)) \frac{n^2}{r},
  \]
  which implies that (since we may assume that $\delta \le \alpha$)
  \[
    \lambda \le \frac{\alpha+\delta}{\balpha+\delta} \le \frac{2\alpha}{\balpha} = \frac{1}{2r}.
  \]
  On the other hand, if fewer than $r$ components had size larger than $(1-\balpha) \cdot n/r$, then
  \[
    (1-\lambda) n \le (r-1) \cdot (1+\delta)\frac{n}{r} \le \left(1-\frac{1}{r}\right)\left(1+\frac{1}{2r}\right)n < \left(1-\frac{1}{2r}\right)n,
  \]
  as we may assume that $\delta \le 1/(2r)$.
  This completes the proof of the proposition.
\end{proof}

For a $\cC$-rigid graph $G$, we will call the (necessarily unique) collection $\{S_1, \dotsc, S_r\}$ in the assertion of \autoref{prop:rigid-has-core} the \emph{$\cC$-core} of $G$ and denote it by $\core_{\cC}(G)$.
Note that, in contrast to our convention for cuts, we think of the core as unordered.
We will say that (the vertex set of) a coloured graph $Q$ is \emph{contained in the core} if each of the sets $V^1(Q), \dotsc, V^k(Q)$ of vertices of $Q$ coloured $1, \dotsc, r$, respectively, is contained in a distinct element of the core.

Our next lemma, which is the key result of this section, is a variant of \cite[Lemma~10.2]{DeMKah15Tur}.

\begin{lemma}\label{lemma:correlation-argument}
  Let $Q$ be an $\br{r}$-coloured graph, let $\cC$ be a family of balanced $r$-cuts compatible with $Q$, and let $\xi \le 1$.
  Suppose that, for each tuple $S=(S_1, \dotsc, S_r)$ that is contained in some cut from $\cC$ and satisfies $\min_i |S_i| \ge (1-4r\alpha) \cdot n/r$, we have an event $F(S_1, \dotsc, S_r)$ that is decreasing, determined by $\ext(S)$, and satisfies
  \[
    \Pr\big(F(S_1, \dots, S_r) \mid Q \subseteq G_{n,p}\big) \le \xi.
  \]
  Let $\cR$ be the event that $G_{n,p}$ is $\cC$-rigid, with $\{S_1, \dotsc, S_r\} = \core_{\cC}(G)$, labeled so that $V^i(Q) \subseteq S_i$ for every $i \in \br{r}$, and $F(S_1, \dotsc, S_r)$ holds. Then $\Pr(\cR \mid Q \subseteq G_{n,p}) \le r! \cdot \xi$.
\end{lemma}
\begin{proof}
  For an ordered tuple of pairwise-disjoint vertex sets $S = (S_1, \dots, S_r)$ that is compatible with $Q$ and satisfies $\min_i |S_i| \ge (1-4r\alpha)n/r$, denote by $E(\{S_1, \dots, S_r\})$ the event that $G_{n,p}$ is $\cC$-rigid, with core $\{S_1, \dotsc, S_r\}$.
  We claim that this event is increasing in $\ext(S)$. Indeed, if $G \in E(\{S_1, \dotsc, S_r\})$, then adding to $G$ an edge $e \in \ext(S)$ does not change the set of largest cuts of $G$; in particular, the graph $G \cup e$ is rigid and has the same core as $G$.
  By Harris's inequality,
  \begin{multline*}
    \Pr\big(E(\{S_1, \dotsc, S_r\}) \cap F(S_1, \dots S_r) \mid Q \subseteq G_{n,p}\big) \\
    \le \Pr\big(E(\{S_1, \dots, S_r\}) \mid Q \subseteq G_{n,p}\big) \cdot \Pr\big(F(S_1, \dots S_r) \mid Q \subseteq G_{n,p}\big).
  \end{multline*}
  Consequently,
  \begin{align*}
    \Pr(\cR \mid Q \subseteq G_{n,p})
    & = \sum_{(S_1, \dotsc, S_r)} \Pr\big(E(\{S_1, \dots, S_r\}) \cap F(S_1, \dots S_r) \mid Q \subseteq G_{n,p}\big) \\
    & \le\sum_{(S_1, \dots, S_r)} \Pr(E(\{S_1, \dots, S_r\}) \mid Q \subseteq G_{n,p}) \cdot \xi
      \le r! \cdot \xi,
  \end{align*}
  where the sums are over all ordered tuples $(S_1, \dotsc, S_r)$ as above and the last inequality follows since the core is unique and its elements can be ordered in no more than $r!$ different ways.
\end{proof}

\section{Approximating cuts with small deficit by cores}
\label{sec:algorithm}

Recall that, for a $Q \in \cQ$, we defined $\cC_Q$ to be the family of all $\delta$-balanced $r$-cuts that are compatible with $Q$.
It will be convenient to define, for an $\br{r}$-coloured graph $Q$,
\[
  \vmin(Q) \coloneqq \min\{|V^k(Q)| : k \in \br{r} \text{ and } V^k(Q) \neq \emptyset\}.
\]
Let $\cB_{Q, \cF}(d, x)$ denote the set of all graphs $G \supseteq Q$ that admit an $r$-cut $\Pi \in \cC_Q$ with $\deficit_{\cC_Q, G}(\Pi) \le d$ and $\nu\big(\cF[G \cap \ext(\Pi)]\big) \le x$, cf.~\ref{item:plan-P2} in the plan of the proof presented in \autoref{section:outline}.
The following statement will allow us to bound the probability of $\cB_{Q, \cF_Q'}(d_Q, d_Q)$ from above in terms of the probability of the event that $\nu\big(\cF_Q'[G_{n,p} \cap \ext(\core_{\cC_Q}(G_{n,p}))]\big) \le d_Q$, which we will be able to estimate, in view of \autoref{lemma:correlation-argument}, using Janson's inequality and the calculations performed in \autoref{sec:low-degree-case} and~\autoref{sec:high-degree-case}.

\begin{prop}\label{prop:algorithm-result}
  Suppose that an $\br{r}$-coloured graph $Q \in \cQ$, integers $d$, $D$, $m$, and $x$, and a family $\cF$ of subgraphs of $K_n \setminus Q$ satisfy:
  \[
    d, e(Q) \ll n^2p
    \qquad
    \text{and}
    \qquad
    d \le \vmin(Q) \cdot \min\{np, \delta/(2r^3) \cdot n\}.
  \]
  Denote
  \[
    \xi \coloneqq \Pr\big(G_{n,p} \text{ is $\cC_Q$-rigid and } \nu(\cF[G_{n,p} \cap \ext(\core_{\cC_Q}(G_{n,p}))]) \le x + m + r^2 (d+1) \mid Q \subseteq G_{n,p}\big).
  \]
  Then, for any positive constant $K$, the probability that $\min\big\{ \binom{n}{2}, |\partial\cF[G_{n,p}]|\big\} \le Dm/p$ and $G_{n,p} \in \cT \cap \cB_{Q, \cF}(d, x)$ is at most
  \[
    p^{e(Q)} \cdot D^d \cdot \left( Z^d \cdot \xi + e^{-K d \log(n^2p/d)} \right),
  \]
  where $Z$ depends only on $\alpha$, $p_0$, and $r$.
\end{prop}

\begin{proof}
  Let
  \[
    \cG \coloneqq \left\{G \subseteq K_n : G \in \cT \cap \cB_{Q, \cF}(d,x) \text{ and } \min\big\{\textstyle{\binom{n}{2}},  |\partial\cF[G]|\big\} \le Dm/p \right\};
  \]
  we may clearly assume that $\cG$ is nonempty, as otherwise there is nothing to prove.
  Further, set
  \[
    \balpha \coloneqq \frac{\alpha}{4}
    \qquad \text{and} \qquad
    \gamma \coloneqq \frac{\balpha}{6r},
  \]
  let $\Gamma$ be a large constant (that depends on $\alpha$, $\delta$, $p_0$, and $r$), and set
  \begin{align*}
    L \coloneqq \Gamma K d \log (n^2p/d) \ll n^2p.
  \end{align*}
  In what follows, maximum cuts, critical edges, cores, and deficits are defined relative to $\cC_Q$; in particular, we write $\core$ and $\maxcut$ in place of $\core_{\cC_Q}$ and $\maxcut_{\cC_Q}$.
  Further, following DeMarco and Kahn~\cite{DeMKah15Tur}, given a graph $G$, let
  \[
    \crit(G) \coloneqq \bigcap_{\Pi \in \maxcut(G)} \ext(\Pi) \cap G
  \]
  be the set of \emph{critical} edges of $G$ and note that $\crit(G) \supseteq \ext(\core(G)) \cap G$.

  Pick $\bG_0 \sim G_{n,p}$ conditioned on $G_{n,p} \in \cG$, so that, for every $G_0 \in \cG$,
  \begin{equation}
    \label{eq:bG0-distribution}
    \Pr(\bG_0 = G_0) = \frac{\Pr(G_{n,p} = G_0)}{\Pr(G_{n,p} \in \mathcal{G})}
  \end{equation}
  and let $\bF_0$ be the empty graph.  Further, let $\Pi \in \cC_Q$ be an arbitrary $r$-cut witnessing the fact that $\bG_0 \in \cB_{Q, \cF}(d,x)$.  
  We will now define a random sequence $\big((\bG_i, \bF_i)\big)_{i=0}^T$, where $T \le L$ is a random stopping time.

  \begin{alg}
    For $i = 0, \dotsc, L-1$, do the following. Let 
    \[
      X_i \coloneqq \{\omega \in \cF : \omega \subseteq \bG_i \cap \crit(\bG_i \cup \bF_i) \text{ and } \omega \cap \int(\Pi) \ne \emptyset\}.
    \]
    and let $U_i \coloneqq \bigcup_{\omega \in X_i} \omega \cap \int(\Pi)$.
    
    \begin{enumerate}[{label=(\alph*)}]
    \item
      \label{item:alg-remove-H-copies}
      If $|U_i| \ge m$, then set $\bG_{i+1} \coloneqq \bG_i \setminus e$ for a~uniformly chosen edge $e \in U_i$.
      
    \item
      \label{item:alg-remove-int}
      Otherwise, if $e\big(\crit(\bG_i \cup \bF_i) \cap \int(\Pi)\big) \ge \gamma n^2 p$, then set $\bG_{i+1} \coloneqq \bG_i \setminus e$ for a~uniformly chosen $e \in \big(\bG_i \cap \crit(\bG_i \cup \bF_i) \cap \int(\Pi)\big) \setminus Q$.
      
    \item
      \label{item:alg-remove-int-minus-crit}
      Otherwise, if $\bG_i \cup \bF_i$ is not rigid, then set $\bG_{i+1} \coloneqq \bG_i \setminus e$ for uniformly chosen $e \in \left(\bG_i \cap \int(\Pi)\right) \setminus (\crit(\bG_i \cup \bF_i) \cup Q)$.
      
    \item
      \label{item:alg-add-edge}
      Otherwise, if $\bG_i \cup \bF_i$ is rigid but $Q$ is not contained in the core, do the following.
      Suppose that the core is $\{S_1, \dotsc, S_r\}$ and let $k \in \br{r}$ be an index such that the (nonempty) set $V^k(Q)$ and some $S_j$ are in the same part of at least one but not all cuts in $\maxcut(\bG_i \cup \bF_i)$.
      We then set $\bF_{i+1} \coloneqq \bF_i \cup e$ for a uniformly chosen edge $e \in \ext(\Pi) \setminus (\bG_i \cup \bF_i)$ that connects $V^k(Q)$ and the union of all such $S_j$.
      
    \item
      \label{item:alg-end}
      Otherwise, $\bG_i \cup \bF_i$ is rigid and $Q$ is contained in the core.  Set $T=i$ and stop the algorithm.
    \end{enumerate}
    If the algorithm defined $\bG_L$ and $\bF_L$ (instead of stopping in~\ref{item:alg-end} with $i < L$), set $T = L$.
  \end{alg}

  A sequence $\big((G_i, F_i)\big)_{i=0}^t$ will be called \textit{legal} if it can be produced by the above algorithm, that is, if with nonzero probability $T = t$ and $(\bG_i, \bF_i) = (G_i, F_i)$ for all $i$.
  The following lemma describes several key properties of every legal sequence.
  
  \begin{lemma}
    \label{lemma:sequence-properties}
    Every legal sequence $\big((G_i, F_i)\big)_{i=0}^t$ has the following properties:
    \begin{enumerate}[label=(\roman*)]
    \item
      \label{item:sequence-properties-Q}
      The graph $G_t$ contains $Q$.
    \item
      \label{item: sequence-properties-number-of-edges}
      For every $i \in \{0, \dotsc, t\}$, we have $i = e(G_0) - e(G_i) + e(F_i)$ and $G_i \cap F_i = \emptyset$.
    \item
      \label{item: sequence-properties-steps-a-b}
      At most $d$ steps are of types~\ref{item:alg-remove-H-copies} or \ref{item:alg-remove-int}.
    \item
      \label{item:sequence-properites-add-edge}
      At most $r^2(d+1)$ steps are of type~\ref{item:alg-add-edge} and thus $e(F_i) \le r^2(d+1)$ for every $i$.
    \end{enumerate}
  \end{lemma}
  
  \begin{proof}
    The first item holds as $Q \subseteq G_0$ by the definition of $\cG$ and, while choosing an edge $e$ to be removed from $G_i$ while defining $G_{i+1}$ in steps~\ref{item:alg-remove-H-copies}--\ref{item:alg-remove-int-minus-crit}, we always make sure that $e \notin Q$ (recall that each graph in $\cF$ is disjoint from $Q$ by assumption).
    
    The second item holds as $F_0$ is empty and each step of the algorithm either removes an edge from $G_i$ or adds to $F_i$ an edge that does not belong to $G_i$.  This means, in particular, that $e(G_i) - e(G_{i+1}) + e(F_{i+1}) - e(F_i) = 1$ for each $i \in \{0, \dotsc, t-1\}$;
    the identity $i = e(G_0) - e(G_i) + e(F_i)$ easily follows, as $e(F_0) = 0$.

    For the third item, notice that every step of type~\ref{item:alg-remove-H-copies} or~\ref{item:alg-remove-int} removes a critical edge of $G_i \cup F_i$ that lies in $\int(\Pi)$ and therefore the deficit of $\Pi$ with respect to $G_i \cup F_i$ drops by one.
    Moreover, none of the other steps increases this deficit, as we either remove an edge from $\int(\Pi)$ or add an edge to $\ext(\Pi)$.
    Since the deficit of $\Pi$ with respect to $G_0 \cup F_0$ is at most $d$ and it is always nonnegative, there can be at most $d$ steps of types~\ref{item:alg-remove-H-copies} and~\ref{item:alg-remove-int}.

    The fourth item is an easy consequence of the following stronger property:
    Every block of consecutive steps of type~\ref{item:alg-add-edge} has length at most $r^2$ and the step immediately following it cannot be of type~\ref{item:alg-remove-int-minus-crit}.  To see this, assume first that the step $i \to i+1$ is of type~\ref{item:alg-add-edge}, that is, $G_i \cup F_i$ is rigid, say with core $\{S_1, \dotsc, S_r\}$, and $Q$ is not contained in the core.  Let $j, k \in \br{r}$ be the indices defined in~\ref{item:alg-add-edge}.  Observe that:
    \begin{enumerate}[label=(\arabic*)]
    \item
      \label{item:sequence-properties-1}
      Since the sets $V^k(Q)$ and $S_j$ are in different parts of some maximum cut in $G_i \cup F_i$, adding any edge connecting $V^k(Q)$ to $S_j$ can only shrink the set of maximum cuts; consequently $\maxcut(G_{i+1} \cup F_{i+1}) \subseteq \maxcut(G_i \cup F_i)$.
      Moreover, no cut that puts $V^k(Q)$ and $S_j$ in the same part can be a maximum cut in $G_{i+1} \cup F_{i+1}$.
    \item
      Further, suppose that $i' > i$ is such that all the steps $i \to \dotsb \to i'$ are of type~\ref{item:alg-add-edge}.
      As $\maxcut(G_{i'} \cup F_{i'}) \subseteq \maxcut(G_i \cup F_i)$, by~\ref{item:sequence-properties-1}, the graph $G_{i'} \cup F_{i'}$ is also rigid and the elements of its core can be labeled as $S_1', \dotsc, S_r'$ so that $S_j \subseteq S_j'$ for every $j \in \br{r}$.
      Further, $S_j' \supseteq S_j$ and $V^k(Q)$ cannot be in the same part of a maximum cut of $G_{i'} \cup F_{i'}$ and thus the pair $j$ and $k$ of indices cannot be chosen in step $i' \to i'+1$.
    \end{enumerate}

    The second of the above two observations implies that there can be at most $r^2$ steps of type~\ref{item:alg-add-edge} in a row.  Moreover, since the graph $G_{i+1} \cup F_{i+1}$ produced by a step $i \to i+1$ of type~\ref{item:alg-add-edge} is rigid (as $G_i \cup F_i$ is rigid and $\maxcut(G_{i+1} \cup F_{i+1}) \subseteq \maxcut(G_i \cup F_i)$), the step $i+1 \to i+2$ cannot be of type~\ref{item:alg-remove-int-minus-crit}.
  \end{proof}

  Denote by $T_a$, $T_b$, $T_c$, and $T_d$ the random variables counting the number of steps of types \ref{item:alg-remove-H-copies}, \ref{item:alg-remove-int}, \ref{item:alg-remove-int-minus-crit}, and \ref{item:alg-add-edge}, respectively, in the sequence $\big((\bG_i, \bF_i)\big)_{i=0}^T$ defined above.  Further, let $T_d^1, \dotsc, T_d^r$ be the random variables counting the number of times that a step of type~\ref{item:alg-add-edge} was executed with $k = 1, \dotsc, r$, respectively, so that $T_d = T_d^1 + \dotsb + T_d^r$.
  Since each of these random variables takes values in $\{0, \dotsc, L\}$ and $\sum_{t=0}^L (2(t+1))^{-2} \le 1$, there are $t_a$, $t_b$, $t_c$, and $t_d^1, \dotsc, t_d^r$ such that
  \begin{multline*}
    \Pr\big((T_a, T_b, T_c, T_d^1, \dotsc, T_d^r) = (t_a, t_b, t_c, t_d^1, \dotsc, t_d^r)\big) \\
    \ge \left(2^{r+3}(t_a+1)(t_b+1)(t_c+1)(t_d^1+1)\dotsb(t_d^r+1)\right)^{-2} \ge (2t+2r+6)^{-2r-6},
  \end{multline*}
  where $t \coloneqq t_a + t_b + t_c + t_d^1 + \dotsb + t_d^r$.
  Denote the above event by $W$ and let $\cS$ the corresponding set of legal sequences.  For every $\Sigma \in \cS$, denote by $A_\Sigma$ the event that $\Sigma$ was produced by the algorithm.  Finally, write $t_d \coloneqq t_d^1 + \dotsb + t_d^r$, note that $t = t_a + t_b+t_c+t_d$ is the common length of all the sequences in $\cS$ (minus one), and define
  \begin{align*}
    \cE \coloneqq \left\{G_t \cup F_t : \big((G_i, F_i)\big)_{i=0}^t \in \cS\right\}.
  \end{align*}

  The following two lemmas imply the assertion of the proposition.
  
  \begin{lemma}
    \label{lemma:alg-rigid-graph}
    If $t < L$, then every $G \in \cE$ is rigid and satisfies
    \[
      \nu(\cF[G \cap \ext(\core(G))]) \le x + m +r^2(d+1).
    \]
  \end{lemma}

  \begin{lemma}
    \label{lemma:alg-prob-estimate}
    We have
    \[
      \Pr(G_{n,p} \in \cG) \le p^{e(Q)} \cdot D^d \cdot
      \begin{cases} 
        Z^d \cdot \Pr(G_{n,p} \in \cE \mid Q \subseteq G_{n,p}), &  t < L, \\
        e^{-K d \log (n^2p/d)}, &  t=L.
      \end{cases}
    \]
  \end{lemma}

  \begin{proof}[Proof of \autoref{lemma:alg-rigid-graph}]
    Let $G \in \cE$ and write $G = G_t \cup F_t$, where $(G_t, F_t)$ is the terminal element of some sequence in $\cS$.  Let $\cM$ be a largest matching in $\cF[G \cap \ext(\core(G))]$. Remove from $\cM$ all $\omega \in \cF$ that intersect $F_t$ and denote the resulting matching by $\cM'$. Since $e(F_t) \le r^2 (d+1)$, by \autoref{lemma:sequence-properties}~\ref{item:sequence-properites-add-edge}, we have $|\cM'| \ge |\cM| - r^2 (d+1)$.  Further, remove from $\cM'$ all $\omega \in \cF$ that intersect $\int(\Pi)$ and denote the resulting matching by $\cM''$.  Since $(G_t, F_t)$ does not satisfy the condition of step~\ref{item:alg-remove-H-copies} in the algorithm, every matching in $\cF[G_t \cap \ext(\core(G))] \subseteq \cF[G_t \cap \crit(G)]$ contains fewer than $m$ subgraphs that intersect $\int(\Pi)$ and hence $|\cM''| \ge |\cM'| - m \ge |\cM| - r^2(d+1) - m$.  Note that $\cM''$ is a matching in $\cF[G_t \cap \ext(\Pi)] \subseteq \cF[G_0 \cap \ext(\Pi)]$ and thus $|\cM''| \le x$, as $G_0 \in \cB_{Q, \cF}(d, x)$ and $\Pi$ witnesses this fact.  It follows that
    \[
      \nu(\cF[G \cap \ext(\core(G))]) = |\cM| \le |\cM''| + m + r^2(d+1) \le x + m + r^2 (d+1),
    \]
    as desired.
  \end{proof}

  \begin{proof}[Proof of \autoref{lemma:alg-prob-estimate}]
    Recall the definitions from the paragraph preceding the statement of \autoref{lemma:alg-rigid-graph}.  We have
    \begin{equation}
      \label{eq:PrW-sum-over-Sigma}
      (2t+2r+6)^{-2r-6} \le \Pr(W)  = \sum_{\Sigma = ((G_i, F_i))_i \in \cS} \Pr(\bG_0 = G_0) \cdot \Pr(A_\Sigma \mid \bG_0 = G_0).
    \end{equation}
    Since, for every sequence $\big((G_i, F_i)\big)_{i=0}^t \in \cS$,
    \[
      e(G_t \cup F_t) = e(G_t) + e(F_t) = e(G_0) + 2e(F_t) - t = e(G_0) + 2t_d - t,
    \]
    see~\autoref{lemma:sequence-properties}~\ref{item: sequence-properties-number-of-edges}, we have
    \[
      \Pr(\bG_0 = G_0) \stackrel{\eqref{eq:bG0-distribution}}{=} \frac{\Pr(G_{n,p} = G_0)}{\Pr(G_{n,p} \in \cG)} = \frac{\Pr(G_{n,p} = G_t \cup F_t)}{\Pr(G_{n,p} \in \cG)} \cdot \left(\frac{p}{1-p}\right)^{t-2t_d}.
    \]
    Consequently, multiplying both sides of~\eqref{eq:PrW-sum-over-Sigma} by $\Pr(G_{n,p} \in \cG)$, we obtain
    \begin{equation}
      \label{eq:Pr-cG-upper}
      \begin{split}
        \frac{\Pr(G_{n,p} \in \cG)}{(2t+2r+6)^{2r+6}} & \le \left(\frac{p}{1-p}\right)^{t-2t_d} \sum_{\Sigma = ((G_i, F_i))_i \in \cS} \Pr(G_{n,p} = G_t \cup F_t) \cdot \Pr(A_\Sigma \mid \bG_0 = G_0) \\
        & \le \left(\frac{p}{1-p}\right)^{t-2t_d} \cdot \Pr(G_{n,p} \in \cE) \cdot \underbrace{\max_{G \in \cE} \sum_{\substack{\Sigma = ((G_i, F_i))_i \in \cS \\ G_t \cup F_t = G}} \Pr(A_\Sigma \mid \bG_0 = G_0)}_{\star}.
      \end{split}
    \end{equation}
    
    The next two claims will allow us to bound the maximum in the right-hand side of the above inequality.
    Set $t_{\bar{c}} \coloneqq t_a + t_b + t_d = t-t_c$ and $N \coloneqq n^2/2$.
    
    \begin{claim}
      \label{claim:upper-bound-for-sequences}
      For each $G \in \cE$, the number of sequences $\big((G_i, F_i)\big)_{i=0}^t \in \cS$ such that $G_t \cup F_t = G$ is at most
      \[
        \left(\frac{3et}{t_{\bar{c}}}\right)^{t_{\bar{c}}} \left(\frac{Dm}{p}\right)^{t_a} N^{t_b+t_{\bar{c}}} \left(\frac{(1-\alpha/2)N(1-p)}{r}\right)^{t_c-t_{\bar{c}}} \prod_{k=1}^r \big(2r^3e|V^k(Q)|np\big)^{t_d^k}.
      \]
    \end{claim}

    \begin{claim}
      \label{claim:Pr-Sigma}
      For every $\Sigma = \big((G_i, F_i)\big)_{i=0}^t \in \cS$,
      \begin{align*}
        \Pr(A_\Sigma \mid \bG_0 = G_0) \le \left(\frac{1}{m}\right)^{t_a} \left(\frac{1}{\gamma N p}\right)^{t_b} \left(\frac{r}{(1-\balpha)Np}\right)^{t_c} \prod_{k=1}^r \left(\frac{4r}{(1-p_0)|V^k(Q)|n}\right)^{t_d^k}.
      \end{align*}
    \end{claim}

    Before proving \autoref{claim:upper-bound-for-sequences} and~\autoref{claim:Pr-Sigma}, we first show how they imply the assertion of the lemma.  By the two claims, the maximum $\star$ in the right-hand side of~\eqref{eq:Pr-cG-upper} can be bounded from above as follows (recall that $t_{\bar{c}} = t_a+t_b+t_d$):
    \[
      \begin{split}
        \star & \le \left(\frac{3et}{t_{\bar{c}}}\right)^{t_{\bar{c}}} \left(\frac{D}{p}\right)^{t_a} \left(\frac{1}{\gamma p}\right)^{t_b} \left(\frac{(1-\alpha/2)(1-p)}{(1-\balpha)p}\right)^{t_c} \left(\frac{r}{(1-\alpha/2)(1-p)}\right)^{t_{\bar{c}}} \left(\frac{8r^4e p}{1-p_0}\right)^{t_d} \\
        & \le \left(\frac{C_\star t}{t_{\bar{c}}}\right)^{t_{\bar{c}}} D^{t_a} \left(\frac{1-\alpha/2}{1-\balpha}\right)^t  \cdot \left(\frac{1-p}{p}\right)^{t_a+t_b+t_c-t_d},
      \end{split}
    \]
    where $C_\star$ depends only on $\alpha$, $\balpha$, $\gamma$, $p_0$, and $r$ (recall our assumption that $p \le p_0 < 1$).  Substituting this estimate back into~\eqref{eq:Pr-cG-upper} and using the identity $t_a + t_b + t_c - t_d = t - 2t_d$, we obtain
    \[
      \frac{\Pr(G_{n,p} \in \cG)}{(2t+2r+6)^{2r+6}} \le \left(\frac{C_\star t}{t_{\bar{c}}}\right)^{t_{\bar{c}}} D^{t_a} \left(\frac{1-\alpha/2}{1-\balpha}\right)^t \cdot \Pr(G_{n,p} \in \cE).
    \]
    Since $t_a \le d$, by \autoref{lemma:sequence-properties}~\ref{item: sequence-properties-steps-a-b}, denoting
    \[
      \tau \coloneqq \frac{t_{\bar{c}}}{t}
      \qquad
      \text{and}
      \qquad
      \lambda \coloneqq \log\left(\frac{1-\balpha}{1-\alpha/2}\right) > 0,
    \]
    we obtain
    \begin{equation}
      \label{eq:Pr-Gnp-G-final}
      \frac{\Pr(G_{n,p} \in \cG)}{\Pr(G_{n,p} \in \cE)} \le D^d \cdot (2t+2r+6)^{2r+6} \cdot \exp\left(\tau t \log \frac{C_\star}{\tau} - \lambda t\right).
    \end{equation}
    Observe that, since $x \log (1/x) \to 0$ as $x \to 0$, there must be some $\tau_0 = \tau_0(C_\star, \lambda) = \tau_0(\alpha, \balpha, \gamma, p_0, r)$ such that $\tau \log(C_\star/\tau) \le \lambda/2$ for all $\tau \le \tau_0$.
    We also note that items~\ref{item: sequence-properties-steps-a-b} and~\ref{item:sequence-properites-add-edge} in \autoref{lemma:sequence-properties} imply that
    \begin{equation}
      \label{eq:tau-upper}
      \tau t = t_{\bar{c}} = t_a +t_b+t_d \le (r^2+1)(d+1) \le 4r^2d.
    \end{equation}
    In order to estimate the right-hand side of~\eqref{eq:Pr-Gnp-G-final}, we consider two cases.

    \smallskip
    \noindent
    \textit{Case 1 ($t = L$).}
    In this case, $\tau \le 4r^2/\Gamma \le \tau_0$, since $\Gamma \ge \Gamma(\tau_0, r) = \Gamma(\alpha, \balpha, \gamma, p_0, r)$.
    Consequently,
    \[
      \exp\left(\tau t \log \frac{C_\star}{\tau} - \lambda t\right) \le \exp\left(-\frac{\lambda t}{2}\right)
      = \exp\left(-\frac{\lambda L}{2}\right),
    \]
    Finally, as $L = \Gamma K d \log(n^2p/d) \ge \Gamma$ and $\Gamma \ge \Gamma(\lambda,r) = \Gamma(\alpha, \balpha, r)$, we have, by~\eqref{eq:Pr-Gnp-G-final},
    \[
      \Pr(G_{n,p} \in \cG)\le D^d \cdot (2L+2r+6)^{2r+6} \cdot e^{-\lambda L /2} \le D^d \cdot e^{-\lambda L / 3} \le D^d \cdot e^{-Kd \log(n^2p/d)}.
    \]
    
    \smallskip
    \noindent
    \textit{Case 2 ($t < L$).}
    Since $\tau \log(C_\star/\tau) \le \lambda/2$ when $\tau \le \tau_0$ and $\log(C_\star/\tau) \le \log(C_\star/\tau_0)$ otherwise,
    \[
      \exp\left(\tau t \log \frac{C_\star}{\tau} - \lambda t\right) \le \exp\left(\tau t \log\frac{C_\star}{\tau_0}-\frac{\lambda t}{2}\right)
      \le \left(\frac{C_\star}{\tau_0}\right)^{4r^2d} \exp\left(-\frac{\lambda t}{2}\right),
    \]
    where the last inequality uses~\eqref{eq:tau-upper}.
    Finally, since $(2t+2t+6)^{2r+6} \exp(-\lambda t/2)$ is bounded by a constant that depends only on $r$ and $\lambda$, we may conclude from~\eqref{eq:Pr-Gnp-G-final} that
    \[
      \Pr(G_{n,p} \in \cG) \le D^d \cdot Z^d \cdot \Pr(G_{n,p} \in \cE)
    \]
    for some $Z = Z(\alpha, \balpha, \gamma, p_0, r)$.  Finally, since every graph in $\cE$ contains $Q$, by \autoref{lemma:sequence-properties}~\ref{item:sequence-properties-Q}, we have $\Pr(G_{n,p} \in \cE) = p^{e(Q)} \cdot \Pr(G_{n,p} \in \cE \mid Q \subseteq G_{n,p})$
    
    \begin{proof}[Proof of \autoref{claim:upper-bound-for-sequences}]
      Fix some $G \in \cE$.  We bound the number of different ways to construct a sequence $\big((G_i, F_i)\big)_{i=0}^t \in \cS$ such that $G = G_t \cup F_t$.

      First, we should choose the types of the $t$ steps $0 \to \dotsc \to t$ in the above algorithm.  There are $t_{\bar{c}}$ steps of type other than~\ref{item:alg-remove-int-minus-crit} and thus there are at most $\binom{t}{t_{\bar{c}}} 3^{t_{\bar{c}}}$ choices for all the types.

      Second, we bound the number of different choices for the sequence $(F_i)_{i=0}^t$.  In order to choose the partition $G = G_t \cup F_t$, we need to decide which $t_d$ edges of $G$ belong to $F_t$.  By construction, $F_t$ can be partitioned into sets of sizes $t_d^1, \dotsc, t_d^r$ such that the $k$th part contains only edges incident with $V^k(Q)$.
      Since $G_t \subseteq G_0 \in \cT$, the number of such edges is at most
      \[
        |V^k(Q)| \cdot \Delta(G_t) + e(F_t)
        \le |V^k(Q)| \cdot 2np + r^2(d+1) \le 2r^2 |V^k(Q)|np,
      \]
      where the first inequality follows from \autoref{lemma:sequence-properties}~\ref{item:sequence-properites-add-edge} and the second inequality is true because $d \le \vmin(Q) \cdot np$.
      Since the sequence $(F_i)_{i=0}^t$ increases by one edge at steps $i \to i+1$ that were designated to be of type~\ref{item:alg-add-edge} and remains constant otherwise, the number of possible sequences $(F_i)_{i=0}^t$ is at most
      \[
        \binom{t_d}{t_d^1, \dotsc, t_d^r} \cdot \prod_{k=1}^r  \big(2r^2|V^k(Q)|np\big)^{t_d^k} \le \prod_{k=1}^r \big(2r^3|V^k(Q)|np\big)^{t_d^k}.
      \]

      Finally, we bound the number of ways to choose the sequence $(G_i)_{i=0}^t$.  Recall that this sequence remains constant at steps $i \to i+1$ that were designated to be of type~\ref{item:alg-add-edge} and decreases by one edge at all the remaining steps.  We may thus reconstruct it from the knowledge of $G_t$ as follows:
      \begin{itemize}
      \item
        For any $i$, the number of ways to build $G_i$ from $G_{i+1}$ is not more than $N$.
      \item
        If the step $i \to i+1$ is of type~\ref{item:alg-remove-H-copies}, then in order to build $G_i$ from $G_{i+1}$, we need to add an edge which completes an element of $\cF$ with the rest of $G_{i+1}$.
        There are clearly no more than $\min\big\{\binom{n}{2}, |\partial \cF[G_{i+1}]|\big\}$ such edges;
        since $G_{i+1} \subseteq G_0 \in \cG$, and thus $\partial \cF[G_{i+1}] \subseteq \partial \cF[G_0]$, this number is at most $Dm/p$.
      \item
        If both the steps $i \to i+1 \to i+2$ are of type~\ref{item:alg-remove-int-minus-crit}, then $G_{i+1} \cup F_{i+1}$ is not rigid and $b(G_i \cup F_i) = b(G_{i+1} \cup F_{i+1})$, as the edge removed from $G_i$ while forming $G_{i+1}$ was not critical in $G_i \cup F_i$.
        This means that, in order to build $G_i$ from $G_{i+1}$, we need to add an edge so that $b(G_{i+1} \cup F_{i+1})$ does not increase, i.e., an edge whose both endpoints belong to the same $(\cC_Q, G_{i+1} \cup F_{i+1})$-component (see \autoref{sec:rigidity-correlation}).
        In other words, denoting by $\Pi'$ the set of such components, we must add an edge of $\int(\Pi')$.
        Since $G_{i+1} \cup F_{i+1}$ is not rigid, we have $e(\int(\Pi')) \le (1-\alpha)n^2/(2r)$ and thus the number of possible edges is at most
        \[
          e\big(\int(\Pi') \setminus G_{i+1}\big) \le e(\int(\Pi')) - e\big(G_0 \cap \int(\Pi')\big) + e(G_0) - e(G_{i+1}).
        \]
        Since $G_0 \in \cG \subseteq \cT$, we have $e\big(G_0 \cap \int(\Pi')\big) = e(\int(\Pi')) \cdot p \pm o(n^2p)$.
        Furthermore, $e(G_0) - e(G_{i+1}) \le i+1 \le L \ll n^2p$.
        Consequently,
        \[
          e\big(\int(\Pi') \setminus G_{i+1}\big) \le e(\int(\Pi')) \cdot (1-p) + o(n^2p) \le (1-\alpha/2)n^2/(2r) \cdot (1-p),
        \]
        where the last inequality follows as $1-p \ge 1-p_0 > 0$.
      \end{itemize}
      Summarising, since there are at least $t_c-t_{\bar{c}}$ indices $i$ such that the steps $i \to i+1 \to i+2$ are both of type~\ref{item:alg-remove-int-minus-crit}, we may bound the number of sequences $(G_i)_{i=0}^t$ by
      \[
        \left(\frac{Dm}{p}\right)^{t_a} \cdot \left(\frac{(1-\alpha/2)n^2(1-p)}{2r}\right)^{t_c-t_{\bar{c}}} \cdot N^{t_b+t_{\bar{c}}}.
      \]
      
      The claimed bound now follows by multiplying the three above bounds, for the numbers of choices of steps and the sequence $(F_i)_{i=0}^t$ and $(G_i)_{i=0}^t$, and using the standard estimate $\binom{t}{t_{\bar{c}}} \le (et/t_{\bar{c}})^{t_{\bar{c}}}$.
    \end{proof}

    Since, for every $i$, the random pair $(\bG_{i+1}, \bF_{i+1})$ is obtained from $(\bG_i, \bF_i)$ by removing from $\bG_i$ or adding to $\bF_i$ an edge chosen uniformly from some collection, in order to prove~\autoref{claim:Pr-Sigma}, it suffies to establish the following lower bounds on the sizes of these collections, depending on the type of the step.

    \begin{subclaim}
      For every step of type~\ref{item:alg-remove-H-copies}, there are at least $m$ choices.
    \end{subclaim}
    
    \begin{subclaim}
      For every step of type~\ref{item:alg-remove-int}, there are at least $\gamma n^2 p / 2$ choices.
    \end{subclaim}
    \begin{proof}
      If $(\bG_i, \bF_i)$ falls into~\ref{item:alg-remove-int}, then the number of choices for the edge to be removed from $\bG_i$ is at least
      \begin{multline*}
        e\left(\big(\bG_i \cap \crit(\bG_i \cup \bF_i) \cap \int(\Pi)\big) \setminus Q\right) \\
        \ge e\big(\crit(\bG_i \cup \bF_i) \cap \int(\Pi)\big) - e(\bF_i) - e(Q) \ge \gamma n^2p - e(\bF_i) - e(Q).
      \end{multline*}
      Finally, our assumptions and \autoref{lemma:sequence-properties}~\ref{item:sequence-properites-add-edge} imply that $e(\bF_i) \le r^2(d+1) \ll n^2p$ and $e(Q) \ll n^2p$.
    \end{proof}
    
    \begin{subclaim}
      For every step of type~\ref{item:alg-remove-int-minus-crit}, there are at least $\frac{(1-\balpha)n^2p}{2r}$ choices.
    \end{subclaim}
    \begin{proof}
      If $(\bG_i, \bF_i)$ falls into~\ref{item:alg-remove-int-minus-crit}, then the number of choices for the edge to be removed from $\bG_i$ is at least
      \[
        e\big(\bG_i \cap \int(\Pi)\big) - e\big(\crit(\bG_i \cup \bF_i ) \cap \int(\Pi)\big) - e(Q).
      \]
      Since $\bG_0 \in \cT$, we have
      \[
        e\big(\bG_i \cap \int(\Pi)\big) \ge e\big(\bG_0 \cap \int(\Pi)\big) - i \ge e(\int(\Pi)) \cdot p - o(n^2p) - i.
      \]
      Crucially, since $(\bG_i, \bF_i)$ did not fall into~\ref{item:alg-remove-int}, we have $e\big(\crit(\bG_i \cup \bF_i) \cap \int(\Pi)\big) < \gamma n^2p$.
      Finally, since $i + e(Q) \le L + e(Q) \ll n^2p$ and $\Pi \in \cC_Q$ is balanced (and $\delta$ is small as a function of $\alpha$ and thus also of $\gamma$),
      we may conclude that the number of choices is at least
      \[
        e(\int(\Pi)) p - 2 \gamma n^2 p \ge \frac{n^2p}{2r} - 3\gamma n^2p = (1-\balpha) \cdot \frac{n^2p}{2r},
      \]
      as (sub)claimed.
    \end{proof}

    \begin{subclaim}
      For every step of type~\ref{item:alg-add-edge}, there are at least $(1-p_0)|V^k(Q)|n/(4r)$ choices.
    \end{subclaim}
    \begin{proof}
      Suppose that $(\bG_i, \bF_i)$ falls into~\ref{item:alg-add-edge}, that is, that $\bG_i \cup \bF_i$ is rigid, with core $\{S_1, \dotsc, S_r\}$ and $Q$ is not contained in the core.
      This means that there are indices $j, k \in \br{r}$ such that the nonempty set $V^k(Q)$ and $S_j$ are in the same part of some but not all maximum cuts of $\bG_i \cup \bF_i$.
      Note crucially that, for each such $k$, there must be at least two different indices $j$;
      indeed, if $V^k(Q)$ and $S_j$ are not in the same part of some maximum cut, then $V^k(Q)$ shares its part with some $S_{j'}$ with $j' \neq j$.

      We now argue that $\min\{|S_j \cap V_k|, |S_{j'} \cap V_k|\} < \sqrt{2 \gamma} n$,
      where $V_k \supseteq V^k(Q)$ is the $k$th set in the (ordered) $r$-cut $\Pi$.
      If this were not true, then $\ext\big(\{S_j \cap V_k, S_{j'} \cap V_k\}\big)$ would be a set of at least $2\gamma n^2$ edges of $\int(\Pi)$.
      Moreover, since $\bG_0 \in \cT$, we would have
      \[
        e\big(\ext(\{S_j \cap V_k, S_{j'} \cap V_k\}) \cap \bG_0\big) \ge 2\gamma n^2p - o(n^2p)
      \]
      and thus
      \[
        e\big(\ext(\{S_j \cap V_k, S_{j'} \cap V_k\}) \cap \bG_i\big) \ge 2\gamma n^2p - i - o(n^2p) \ge \gamma n^2p.
      \]
      However, $\ext(\{S_j \cap V_k, S_{j'} \cap V_k\}) \cap \bG_i \subseteq \crit(\bG_i \cup \bF_i) \cap \int(\Pi)$, contradicting the assumption that the step $i \to i+1$ did not trigger case~\ref{item:alg-remove-int} in our algorithm.

      We may thus assume, without loss of generality, that $|S_j \cap V_k| < \sqrt{2\gamma} n$ and thus
      \begin{equation}
        \label{eq:Sj-Vk-lower}
        |S_j \setminus V_k| = |S_j| - |S_j \cap V_k| \ge \left(1-4r\alpha - \sqrt{2\gamma}r\right)\frac{n}{r} \ge \frac{n}{2r},
      \end{equation}
      where the last inequality holds as $\alpha$ (and thus $\gamma$) is small.      
      Since we may add to~$\bF_i$ any edge of $\ext(\{V^k(Q), S_j \setminus V_k\}) \setminus (\bG_i \cup \bF_i)$ and
      \[
        e(\ext(\{V^k(Q), S_j \setminus V_k\})) = |V^k(Q)| \cdot |S_j \setminus V_k| \ge |V^k(Q)| \cdot n/(2r),
      \]
      it is enough to argue that
      \begin{equation}
        \label{eq:bGbF-density-VkQ-Sj}
        e\big(\ext(\{V^k(Q), S_j \setminus V_k\})  \cap (\bG_i \cup \bF_i)\big) \le (1+p_0)/2 \cdot e\big(\ext(\{V^k(Q), S_j \setminus V_k\})\big).
      \end{equation}
      
      To this end, observe that
      \[
        e\big(\ext(\{V^k(Q), S_j \setminus V_k\})  \cap (\bG_i \cup \bF_i)\big) \le e\big(\ext(\{V^k(Q), S_j \setminus V_k\}) \cap \bG_0\big)+ e(\bF_i)
      \]
      and that, by~\autoref{lemma:sequence-properties}~\ref{item:sequence-properites-add-edge},
      \begin{equation}
        \label{eq:ebFi-upper}
        e(\bF_i) \le r^2(d+1) \le \vmin(Q) \cdot 2r^2d \le |V^k(Q)| \cdot \min\big\{2r^2np, \delta n / r\big\}.
      \end{equation}
      Consequently, the left-hand side of~\eqref{eq:bGbF-density-VkQ-Sj} is clearly not larger than $|V^k(Q)| \cdot (\Delta(\bG_0) + 2r^2np)$.
      Since $\Delta(\bG_0) \le 2np$, as $\bG_0 \in \cG \subseteq \cT$, the desired inequality holds if $p \le 1/(10r^3)$, say.
      We may thus assume for the remainder of the proof that $p > 1/(10r^3)$.
      If $|V^k(Q)| \gg 1$, then the assumption that $\bG_0 \in \cT$ implies that
      \[
        e\big(\ext(\{V^k(Q), S_j \setminus V_k\}) \cap \bG_0\big) \le  (1+o(1)) \cdot e\big(\ext(\{V^k(Q), S_j \setminus V_k\}) \big) \cdot p,
      \]
      which again implies~\eqref{eq:bGbF-density-VkQ-Sj}, as $e(\bF_i) \le r^2(d+1) \le \vmin(Q) \cdot \delta n / r \le |V^k(Q)| \cdot \delta n /r$, $p \le 1-p_0$, and $\delta \le \delta(r,p_0)$, so we may further assume that $|V^k(Q)| = O(1)$.
      But this means that $Q \in \cQ_L$ (every graph in $\cQ_H$ has $\Omega(np)$ vertices) and thus $k=1$ and $V^2(Q) = \dotsc = V^r(Q) = \emptyset$.
      
      Here, our argument is somewhat more involved:
      We will show that~\eqref{eq:bGbF-density-VkQ-Sj} must hold, or otherwise a~cut that places $V^1(Q)$ and $S_j$ in the same part cannot be a largest cut in $\bG_i \cup \bF_i$.      
      Let $\Pi' = (V_1', \dotsc, V_r') \in \maxcut(\bG_i \cup \bF_i)$ be an arbitrary maximum cut that puts $V^1(Q)$ and $S_j$ in the same part $V_1'$ and let $\Pi''$ be a cut obtained from $\Pi'$ by choosing $\ell \in \br{r}$ uniformly at random and swapping $V^1(Q)$ and a uniformly chosen random subset $R \subseteq V_\ell'$ of the same size (so that $\Pi'' = \Pi'$ with probability $1/r$).
      Since $\Pi''$ is compatible with $Q$ (for some ordering of its colour classes) and its colour classes have the same sizes as the colour classes of $\Pi'$, it belongs to $\cC_Q$ with probability one.
      Since $\Pi' \in \maxcut(\bG_i \cup \bF_i)$, the difference
      \[
        e\big(\ext(\Pi'') \cap (\bG_i \cup \bF_i)\big) - e\big(\ext(\Pi') \cap (\bG_i \cup \bF_i)\big)
      \]
      is nonpositive.
      This difference can be rewritten as (letting $U \coloneqq V^1(Q)$)
      \begin{multline*}
        \overbrace{e\big(\ext(\{U, V_1' \setminus U\}) \cap (\bG_i \cup \bF_i)\big)}^{E_1} - \overbrace{e\big(\ext(\{U,V_\ell' \setminus U\}) \cap (\bG_i \cup \bF_i)\big)}^{E_2} \\
        +\underbrace{e\big(\ext(\{R, V_\ell' \setminus R\}) \cap (\bG_i \cup \bF_i)\big)}_{E_3} - \underbrace{e\big(\ext(\{R,V_1' \setminus R\}) \cap (\bG_i \cup \bF_i)\big)}_{E_4}.
      \end{multline*}
      
      We now estimate the expected values of $E_1, \dotsc, E_4$.
      First, since we assumed that the sets $U$ and $S_j$ are disjoint and both contained in $V_1'$, we have
      $E_1 \ge e\big(\ext(\{U, S_j\})  \cap (\bG_i \cup \bF_i)\big)$.  Second, since each vertex appears in $V_\ell'$ with probability exactly $1/r$, we have
      \[
        \Ex[E_2] = \frac{e\big(\ext(\{U, V \setminus U\}) \cap (\bG_i \cup \bF_i)\big)}{r} \le \frac{|U| \cdot \Delta(\bG_0) + e(\bF_i)}{r} \le \frac{|U| \cdot (np + \delta n + o(n))}{r},
      \]
      where the last inequality follows as $\Delta(\bG_0) \le np + o(n)$, since $\bG_0 \in \cG \subseteq \cT$, and $e(\bF_i) \le |U| \cdot \delta n / r$, by~\eqref{eq:ebFi-upper}.
      Third, since for every $j \in \br{r}$,
      \[
        \sum_{v \in R} \deg_{\bG_i \cup \bF_i}(v, V_j') - |R|^2 \le  e\big(\ext(\{R,V_j' \setminus R\}) \cap (\bG_i \cup \bF_i)\big) \le \sum_{v \in R} \deg_{\bG_i \cup \bF_i}(v, V_j')
      \]
      and, recalling that $R$ is a uniformly chosen random $|U|$-element subset of $V_\ell'$,
      \[
        \Ex\left[\sum_{v \in R} \deg_{\bG_i \cup \bF_i}(v, V_j') \mid \ell \right] = \frac{|U|}{|V_\ell'|} \cdot \sum_{v \in V_\ell'} \deg_{\bG_i \cup \bF_i}(v, V_j'),
      \]
      we have
      \[
        \begin{split}
          \frac{\Ex[E_3-E_4]}{|U|}
          & \ge \frac{r-1}{r} \cdot \min_{j \neq 1} \left\{\frac{2e\big(\int(\{V_j'\}) \cap (\bG_i \cup \bF_i)\big)}{|V_j'|}  - \frac{e\big(\ext(\{V_1', V_j'\}) \cap (\bG_i \cup \bF_i)\big)}{|V_j'|}-|U|\right\} \\
          & \ge \frac{r-1}{r} \cdot \min_{j \neq 1} \frac{2e\big(\int(\{V_j'\}) \cap \bG_0\big) - e\big(\ext(\{V_1', V_j'\}) \cap \bG_0\big) - e(\bF_i) - 2i}{|V_j'|} - |U|.
        \end{split}
      \]
      Further, as $\bG_0 \in \cG \subseteq \cT$ and $e(\bF_i) \le i \le L \ll n^2p$,
      \[
        \frac{\Ex[E_3-E_4]}{|U|} \ge \frac{r-1}{r} \cdot \min_{j \neq 1} \left\{ |V_j'| \cdot p - |V_1'| \cdot p  - o(n) \right\} \ge -\frac{2\delta n}{r},
      \]
      where the second inequality follows since $\Pi'$ is a balanced cut.

      Recalling that $E_1-E_2+E_3-E_4$ is always nonpositive, we may conclude that
      \[
        e\big(\ext(\{U, S_j\})  \cap (\bG_i \cup \bF_i)\big) \le E_1 \le \Ex[E_2-E_3+E_4] \le \frac{|U| \cdot (np+3\delta n + o(n))}{r}
      \]
      and thus, by~\eqref{eq:Sj-Vk-lower},
      \[
        e\big(\ext(\{U, S_j\})  \cap (\bG_i \cup \bF_i)\big) \le \frac{p+3\delta+o(1)}{1-2r\alpha-\sqrt{2\gamma}} \cdot |U| \cdot |S_j \setminus V_k|.
      \]
      Since $p \le p_0 < 1$ and $\alpha$ (and thus $\gamma$) and $\delta$ are sufficiently small with respect to $1-p_0$ and $r$, we obtain
      \[
        e\big(\ext(\{U, S_j\})  \cap (\bG_i \cup \bF_i)\big) \le \frac{1+p_0}{2} \cdot |U| \cdot |S_j \setminus V_k| = \frac{1+p_0}{2} \cdot e\big(\ext(\{U, S_j \setminus V_k\})\big),
      \]
      which, recalling that $U = V^k(Q)$, implies~\eqref{eq:bGbF-density-VkQ-Sj}.
    \end{proof}
    The proof of~\autoref{claim:Pr-Sigma} and thus of \autoref{lemma:alg-prob-estimate} is finally complete.
  \end{proof}
  Since we have already shown how~\autoref{lemma:alg-rigid-graph} and \autoref{lemma:alg-prob-estimate} imply the assertion of the proposition, the proof is complete.
\end{proof}

\section{The proof}
\label{sec:proof}

Recall that our goal is to prove that a.a.s., for every $Q \in \cQ$ with $Q \subseteq G$ and every $r$-cut $\Pi \in \cC_Q$ with $\deficit_{\cC_Q}(\Pi; G) \le d_Q$, the family $\cF_Q[G \cap \ext(\Pi)]$ contains a matching of size exceeding $d_Q$, see~\ref{item:plan-P2} in the proof plan presented in \autoref{section:outline}.
Note that this is precisely the complement of the event $\cB_{Q, \cF_Q}(d_Q,d_Q)$ defined at the beginning of \autoref{sec:algorithm}.
For technical reasons, we will actually bound from above the probability of the event $\cB_{Q, \cF_Q'}(d_Q, d_Q) \supseteq \cB_{Q, \cF_Q}(d_Q, d_Q)$, where $\cF_Q' \subseteq \cF_Q$ is a (carefully chosen) subfamily of $\cF_Q$.
Gearing up towards an application of \autoref{prop:algorithm-result}, we will also choose, for each $Q \in \cQ$, integers $D_Q$ and $m_Q$ and define $\cD_Q$ to be the event $\min\big\{ \binom{n}{2}, |\partial\cF_Q'[G]|\big\} \le D_Qm_Q/p$, so that
\begin{multline*}
  \Pr\big(\cB_{Q, \cF_Q'}(d_Q, d_Q) \text{ for some $Q \in \cQ[G]$}\big)
  \le
  \Pr(\cT^c) + \Pr\big(\cD_Q^c \text{ for some $Q \in \cQ[G]$}\big) \\
  + \sum_{Q \in \cQ} \Pr\big(\cT \cap \cD_Q \cap \cB_{Q, \cF_Q'}(d_Q,d_Q)\big).
\end{multline*}

What we would like to do is take $\cF_Q' \coloneqq \cFQL$ for each $Q \in \cQ_L$ and $\cF_Q' \coloneqq \cFQH$ for all $Q \in \cQ_H$.
In some cases (when $Q \in \cQ_L^2$ and $p \gg p_H$), however, choosing such large families $\cF_Q'$ would make the event $\cD_Q^c$ likely, unless the product $D_Qm_Q$ is large, in which case the upper bound on the probability of $\cT \cap \cD_Q \cap \cB_{Q, \cF_Q'}(d_Q,d_Q)$ provided by \autoref{prop:algorithm-result} (which is increasing in both $D_Q$ and $m_Q$) would not be sufficiently strong for our purposes.
In order to strike a balance between the events $\cD_Q^c$ and $\cT \cap \cD_Q \cap \cB_{Q, \cF_Q'}(d_Q, d_Q)$ in the problematic cases, we will let $\cF_Q'$ be a (pseudo)random subfamily of $\cFQL$ with an appropriate density.

To this end, given a family $\cHh \subseteq \cH$ of copies of $H$ in $K_n$, define, for every $Q \in \cQ_L$,
\[
  \cFQL(\cHh) \coloneqq \{A \setminus Q : A \in \cHQL \cap \cHh\},
\]
cf.\ the definition of $\cFQL$ given at the beginning of \autoref{sec:low-degree-case}.
We will let $\cF_Q'$ be the family $\cFQL(\cHh)$ with $\cHh$ being (a typical sample from the distribution of) the binomial random hypergraph $\cH_q$, for an appropriately chosen $q$.
The reason why we define $\cF_Q'$ via a random subset $\cHh \subseteq \cH$, rather than simply letting $\cF_Q'$ be a random subset of $\cFQL$ or $\cFQH$, is to avoid an unnecessary union bound over all $Q \in \cQ$ while estimating the probability of $\cD_Q$ failing for some $Q \in \cQ[G]$.

The remainder of this section is organised as follows.
First, in \autoref{sec:sparsification}, we construct the family $\cF_Q'$ in the problematic case $Q \in \cQ_L^2$ and $p \gg p_H$.
Second, in \autoref{sec:choosing-parameters}, we define the parameters for the applications of \autoref{prop:algorithm-result} for all $Q \in \cQ$.
Finally, in~\autoref{sec:calculations}, we put everything together and conclude the proof of~\autoref{thm:main}.

\subsection{Sparsification}
\label{sec:sparsification}

In order to obtain an upper bound on the probability of $\cT \cap \cD_Q \cap \cB_{Q, \cF_Q'}(d_Q,d_Q)$ using \autoref{prop:algorithm-result}, we need to establish an upper bound on the probability that $G$ is $\cC_Q$-rigid and $\cF_Q'[\ext(\core_{\cC_Q}(G))]$ contains no large matching (see the definition of $\xi$ in the statement of \autoref{prop:algorithm-result}).
To this end, we will use \autoref{thm:Janson-matchings} in combination with \autoref{lemma:correlation-argument}, which supply an upper bound on this probability (the said $\xi$) that is expressed in terms of the quantities $\Delta_p(\cF_Q')$ and the minimum of $\mu_p(\cF_Q'[\ext(S)])$ over all $r$-tuples $S = (S_1, \dotsc, S_r)$ that are compatible with $Q$ and satisfy $\min_i|S_i| \ge (1-4r\alpha)n/r$.
Since the upper tail of $\Delta_p\big(\cFQL(\cHh)\big)$ is likely too heavy to allow a union bound over all $Q$, instead of choosing a single $\cHh$ for all $Q$, we sample a sequence of independent copies of $\cHh$.
It is straightforward to show that the probability that $\cFQL(\cHh)$ does not have the desired properties for \emph{all} random hypergraphs in the sequence is small enough to warrant a union bound over all $Q \in \cQ$;
this means that we can find at least one suitable $\cHh$ for each $Q \in \cQ$ and define $\cF'_Q \coloneqq \cFQL(\cHh)$.

While defining our families $\cF_Q'$, we will also need to make sure that $\cD_Q$ holds with probability close to one simultaneously for all $Q \in \cQ[G]$.
The following simple observation supplies an upper bound on the size of $\partial\cF'_Q[G]$ in terms of two simple parameters of $Q$ and a~random variable that depends only on the hypergraph $\cHh \subseteq \cH$ that we used to define $\cF'_Q$.
Recall that, for a vertex $v$, we write $\cH_v$ to denote the collection of all copies of $H$ in $K_n$ that contain $v$.

\begin{obs}
  \label{obs:partial-H}
  The following holds for every family $\cHh \subseteq \cH$ and every $G \subseteq K_n$:
  \begin{enumerate}
  \item
    For all $Q \in \cQ_L$, we have
    \[
      |\partial \cFQL(\cHh)[G]| \le e(Q) \cdot \max_{e \in K_n} |\partial\partial_e \cHh[G]|.
    \]
  \item
    For all $Q \in \cQ_H[G]$, we have
    \[
      |\partial \cFQH[G]| \le k(Q) \cdot \max_{v \in V(K_n)} |\partial\cH_v[G]|,
    \]
    where $k(Q)$ is the number of centre vertices of $Q$.    
  \end{enumerate}  
\end{obs}

The following lemma will allow us to define an appropriate family $\cF'_Q$ for all $Q \in \cQ_L^2$.

\begin{lemma}
  \label{lemma:choices-of-cH-low-deg}
  There exist constants $\hat{C}$ and $\Cpar$ that depend only on $H$ such that the following holds for every $q \in [0,1]$ satisfying
  \[
    q \ge \max_{Q \in \cQ_L^2} \Hat{C} \cdot \frac{n^{3-v_H}}{e(Q)}
    \qquad
    \text{and}
    \qquad
    n^{v_H-2} p^{e_H-2} q \ge \log n.
  \]
  There are $\cH_1, \dotsc, \cH_N \subseteq \cH$ with the following properties:
  \begin{enumerate}[label=(\Alph*)]
  \item
    \label{item:global-cH-property-low-deg}
    A.a.s., $\max_{e \in K_n} |\partial\partial_e\cH_i[G_{n,p}]| \le \Cpar n^{v_H-2}p^{e_H-2}q$ for every $i \in \br{N}$.
  \item
    \label{item:Q-good-low-deg}
    For every $Q \in \cQ_L^2$, there exists an $i \in \br{N}$ such that
    \begin{enumerate}[label=(B\arabic*)]
    \item
      \label{item:first-Q-good-low-deg}
      For every $r$-tuple $S = (S_1, \dotsc, S_r)$, with $V(Q) \subseteq S_1$, of disjoint subsets of vertices of size at least $n/(2r)$,
      \[
        \mu_p\big(\cFQL(\cH_i)[\ext(S)]\big) \ge (q/2) \cdot \mu_p(\cFQL[\ext(S)]).
      \]
    \item
      \label{item:second-Q-good-low-deg}
      $\Delta_p\big(\cFQL(\cH_i)\big) \le 2q^2 \cdot \Delta_p(\cFQL)$.
    \end{enumerate}
  \end{enumerate}
\end{lemma}

\begin{proof}[{Proof of~\autoref{lemma:choices-of-cH-low-deg}}]
  Let $G \sim G_{n,p}$.
  Given an edge $e \in K_n$ and a hypergraph $\cHh \subseteq \cH$, denote by $\cA_e^{\cHh}$ the event that
  \[
    |\partial\partial_e\cHh[G]| > \Cpar n^{v_H - 2} p^{e_H - 2} q
  \]
  and let $\cA^{\cHh} \coloneqq \bigcup_{e \in K_n} \cA_e^{\cHh}$.
  Further, for each $Q \in \cQ_L^2$, let $\cB_{Q,1}^{\cHh}$ denote the event that
  \[
    \mu_p\big(\cFQL(\cHh)[\ext(S)]\big) < (q/2) \cdot \mu_p(\cFQL[\ext(S)]).
  \]
  for some $r$-tuple $S = (S_1, \dotsc, S_r)$, with $V(Q) \subseteq S_1$, of disjoint subsets of vertices of size at least $n/(2r)$ and let $\cB_{Q,2}^{\cHh}$ be the event that $\Delta_p\big(\cFQL(\cHh)\big) > 2q^2 \cdot \Delta_p(\cFQL)$.
  The following two claims easily imply the assertion of the lemma.

  \begin{claim}
    \label{claim:global-cH-property-low-deg}
    $\Pr\big(\cA^{\cH_q} \mid \cH_q\big) < n^{-1}$ with probability at least $1 -n^{-3}$.
  \end{claim}
  \begin{claim}
    \label{claim:Q-property-low-deg}
    $\Pr(\cB_{Q,1}^{\cH_q} \cup \cB_{Q,2}^{\cH_q}) \le 2/3$ for every $Q \in \cQ_L^2$.
  \end{claim}

  Indeed, let $N = n^2$ and let $\cH_1, \dotsc, \cH_N$ be independent samples from the distribution of $\cH_q$.
  Note first that \autoref{claim:global-cH-property-low-deg} and a simple union bound imply that, with probability at least $1-n^{-1}$ (in the choices of $\cH_1, \dotsc, \cH_N$), we have $\Pr(\cA^{\cH_i}) < n^{-1}$ for all $i \in \br{N}$, that is, assertion~\ref{item:global-cH-property-low-deg} holds.
  Further, \autoref{claim:Q-property-low-deg} asserts that, for every fixed pair of $Q \in \cQ$ and $i \in \br{N}$, the probability that either \ref{item:first-Q-good-low-deg} or~\ref{item:second-Q-good-low-deg} fails for this pair is at most $2/3$.
  Since, for a given $Q$, these events are independent for all $i \in \br{N}$, the probability that either~\ref{item:first-Q-good-low-deg} or~\ref{item:second-Q-good-low-deg} fails simultaneously for all $i \in \br{N}$ is at most $(2/3)^N$.
  However, there are no more than $2^{\binom{n}{2}}$ graphs in $\cQ_L^2$ and thus the probability that~\ref{item:Q-good-low-deg} fails is not more than $2^{\binom{n}{2}} \cdot (2/3)^{n^2} \le (8/9)^{n^2/2}$.

  \begin{proof}[{Proof of \autoref{claim:global-cH-property-low-deg}}]
    Denote by $\cA$ the event that $\Pr\big(\cA^{\cH_q} \mid \cH_q\big) \ge n^{-1}$ and observe that
    $\Pr(\cA^{\cH_q}) \ge \Pr(\cA) \cdot n^{-1}$.
    On the other hand, we may also bound the probability of $\cA^{\cH_q}$ from above by first conditioning on $G$.
    To this end, let $\cA_e$ be the event that
    \[
      |\partial\partial_e\cH[G]| > \Cpar n^{v_H - 2} p^{e_H - 2} / (2v_H^4)
    \]
    and note that $\Pr(\cA_e) \le n^{-6}$ when $\Cpar$ is sufficiently large, by \autoref{lemma:partial}~\ref{item:partial-edge}.
    Consequently,
    \[
      \Pr(\cA^{\cH_q}) \le \sum_{e \in K_n} \Pr(\cA_e^{\cH^q})
      \le \sum_{e \in K_n} \left(\Pr(\cA_e) + \Pr\big(\cA_e^{\cH_q} \mid \cA_e^c\big)\right) \le \frac{1}{2n^4} + \sum_{e \in K_n} \Pr\big(\cA_e^{\cH_q} \mid \cA_e^c\big).
    \]
    Gearing up to bound $\Pr\big(\cA_e^{\cH_q} \mid \cA_e^c\big)$, define, for all $e \in K_n$ and $\cHh \subseteq \cH$,
    \[
      Z_e^{\cHh} \coloneqq \big|\big\{A \in \cHh: e \in A \text{ and } A \setminus(e \cup f) \subseteq G \text{ for some } f \in A\big\}\big|
    \]
    and observe that, for all $e \in K_n$,
    \[
      v_H^{-2} \cdot |\partial\partial_e \cHh[G]| \le Z_e^{\cHh} \le v_H^2 \cdot |\partial\partial_e \cHh[G]|.
    \]
    The advantage of working with $Z_e^{\cHh}$ is that, conditioned on $G$, the variable $Z_e^{\cH_q}$ is distributed like $\Bin(Z_e^{\cH}, q)$.
    Since $Z_e^{\cH} \le \Cpar n^{v_H-2}p^{e_H-2}/(2v_H^2)$ on $\cA_e^c$ whereas $Z_e^{\cH_q} > \Cpar n^{v_H-2}p^{e_H-2}q/v_H^2$ on $\cA_e^{\cH^q}$, it follows that (letting $m \coloneqq \Cpar n^{v_H-2}p^{e_H-2}/(2v_H^2) \ge \Cpar/(2v_H^2) \cdot q^{-1}\log n $)
    \[
      \Pr\big(\cA_e^{\cH_q} \mid \cA_e^c\big) \le \Pr\big(\Bin(m,q) > 2mq\big) \le \exp(-mq/8) \le n^{-6},
    \]
    where the last inequality holds provided that $\Cpar$ is sufficiently large.
    We conclude that
    \[
      \Pr(\cA) \le n \cdot \Pr(\cA^{\cH_q}) \le n \cdot \left(\frac{1}{2n^4} + \binom{n}{2} \cdot \frac{1}{n^6}\right) \le \frac{1}{n^3},
    \]
    as claimed.
  \end{proof}

  \begin{proof}[{Proof of \autoref{claim:Q-property-low-deg}}]  
    Fix an arbitrary $Q \in \cQ_L^2$.
    By construction, for each $\omega \in \cFQL$, there is a unique $A \in \cH$ such that $\omega = A \setminus Q$.
    In particular, we have
    \[
      \Ex\big[\Delta_p\big(\cFQL(\cH_q)\big)\big] = q^2 \cdot \Delta_p\big(\cFQL\big)
    \]
    and thus $\Pr(\cB_{Q,2}^{\cH_q}) \le 1/2$ by Markov's inequality.
    It is therefore enough to show that $\Pr(\cB_{Q,1}^{\cH_q}) \le 1/6$.
    To this end, note first that, since $\cFQL$ is a uniform hypergraph, $\cB_{Q,1}^{\cH_q}$ is the event that $|\cFQL(\cHh)[\ext(S)]| < (q/2) \cdot |\cFQL[\ext(S)]|$ for some $r$-tuple $S = (S_1, \dotsc, S_r)$, with $V(Q) \subseteq S_1$, of disjoint subsets of vertices of size at least $n/(2r)$.
    Further, observe that, for every $r$-tuple $S$ of interest,
    \[
      |\cFQL(\cH_q)[\ext(S)]| \sim \Bin\big(|\cFQL[\ext(S)]|, q\big)
    \]
    and $|\cFQL[\ext(S)]| \ge (\pi_H/2) \cdot e(Q) \cdot (n/(2r))^{v_H-2} \eqqcolon m$, by \autoref{lemma:FQL-bounds}.
    Since there are at most $(r+1)^n$ choices for the $r$-tuple $S$,
    \[
      \begin{split}
        \Pr(\cB_{Q,1}^{\cH_q}) & \le (r+1)^n \cdot \Pr\big(\Bin(m,q) < mq/2\big)
        \le (r+1)^n \cdot \exp(-mq/8) \\
        & \le (r+1)^n \cdot \exp(-\hat{C}\pi_H/(2r)^{v_H} \cdot n) \ll 1,
      \end{split}
    \]
    provided that $\hat{C}$ is sufficiently large.
  \end{proof}
  The proof of the lemma is now complete.
\end{proof}

\subsection{Choosing the parameters}
\label{sec:choosing-parameters}

We are finally ready to choose the parameters for our application of \autoref{prop:algorithm-result}.
We first apply \autoref{lemma:choices-of-cH-low-deg} with $q \coloneqq \hat{C}\log n / (\kappa n^{v_H-2}p^{e_H-1})$ to obtain hypergraphs $\cH_1, \dotsc, \cH_N \subseteq \cH$ as in the assertion of the lemma.
Note that we need to first verify that $n^{v_H-2}p^{e_H-2}q = \hat{C}/\kappa \cdot p^{-1} \log n \ge \log n$ and that, for every $Q \in \cQ_L^2$,
\[
  \frac{\hat{C}\log n}{\kappa n^{v_H-2} p^{e_H-1}} \ge \frac{\hat{C}n^{3-v_H}}{e(Q)},
\]
which holds since $e(Q) \ge \kappa np/\log n$ and $e_H \ge 2$.

We now choose the families $\cF_Q'$.
If either $Q \in \cQ_L$ and $p \le \Cth p_H$ or $Q \in \cQ_L^1$ and $p > \Cth p_H$, then we let $\cF_Q' \coloneqq \cFQL$.
If $Q \in \cQ_L^2$ and $p > \Cth p_H$, then we let $\cF_Q' \coloneqq \cFQL(\cH_{i(Q)})$, where $i(Q) \in \br{N}$ is the index from \autoref{lemma:choices-of-cH-low-deg}~\ref{item:Q-good-low-deg}.
Finally, if $Q \in \cQ_H$, then we let $\cF_Q'$ be the hypergraph $\cF \subseteq \cFQH$ from the assertion of \autoref{lemma:high-degree-bounds}.

The choices of the parameters $m_Q$ and $D_Q$ are described in \autoref{table:asd}, where we also recall the definition of $d_Q$ originally given in \autoref{dfn:choice-of-d_Q}.

\begin{table}
  \centering
  \begin{tabular}{ |c|c|c|c|c| }
    \hline
    & $q$ & $d_Q$ & $m_Q$ & $D_Q$ \\
    \hline
    \makecell{$p \le \Cth p_H$ \\ $Q \in \cQ_L$} &  & $\sqrt{\eta}e(Q)\log n \wedge \beta n^2p$ & $\kappa n^{v_H-2}p^{e_H-1} \cdot e(Q)$ & $\frac{4e_H^2}{\kappa}$ \\
    \hline
    \makecell{$p > \Cth p_H$ \\ $Q \in \cQ_L^1$} &  & $8re(Q)$ & $8re(Q)$ & $\frac{n^2 p}{d_Q}$ \\
    \hline
    \makecell{$p > \Cth p_H$ \\ $Q \in \cQ_L^2$} & $\frac{\hat{C}\log n}{\kappa n^{v_H-2}p^{e_H-1}}$ & $\sqrt{\eta} e(Q)\log n \wedge \beta n^2p$ & $\kappa n^{v_H-2} p^{e_H-1} q \cdot e(Q)$ & $\frac{\Cpar}{\kappa}$ \\
    \hline
    \makecell{$Q \in \cQ_H$} & & $32k(Q)np$ & $32k(Q)np$ & $\frac{2v_H \cdot M(k(Q))}{m_Q}$ \\
    \hline
  \end{tabular}
  \caption{
    \label{table:asd}
    The choices of the parameters. (For brevity, $a \wedge b$ denotes $\min\{a, b\}$ and $M(k) \coloneqq kn^{v_H-1}p^{e_H} \wedge n^2p$.)
  }
\end{table}

\subsection{Calculations}
\label{sec:calculations}

We first estimate the probability that $\cD_Q$ fails for some $Q \in \cQ[G]$.
(Recall that $\cD_Q$ is the event that $\min\big\{ \binom{n}{2}, |\partial\cF_Q'[G]|\big\} \le D_Qm_Q/p$.)

\begin{lemma}
  \label{lemma:cDQ}
  A.a.s., $\cD_Q$ holds for all $Q \in \cQ$ with $Q \subseteq G$.
\end{lemma}
\begin{proof}
  We prove this lemma separately for each of the four rows in~\autoref{table:asd}.
  First, if $p \le \Cth p_H$ and $Q \in \cQ_L$, then $\cF_Q' = \cFQL$ and thus on the event~$\cT$, by \autoref{obs:partial-H},
  \[
    |\partial\cF_Q'[G]| \le e(Q) \cdot \max_{e \in K_n} |\partial\partial_e\cH[G]| \le 4e_H^2n^{v_H-2}p^{e_H-2}\cdot e(Q) = D_Qm_Q/p.
  \]
  Second, if $p > \Cth p_H$ and $Q \in \cQ_L^1$, then $\binom{n}{2} < n^2 = D_Qm_Q/p$ and thus $\cD_Q$ holds almost surely.
  Third, if $p > \Cth p_H$ and $Q \in \cQ_L^2$, then $\cF_Q' = \cFQL(\cH_{i(Q)})$ and thus, by \autoref{lemma:choices-of-cH-low-deg}~\ref{item:global-cH-property-low-deg} and \autoref{obs:partial-H},
  \[
    |\partial\cF_Q'[G]| \le e(Q) \cdot \max_{e \in K_n} |\partial\partial_e \cH_{i(Q)}[G]| \le \Cpar n^{v_H-2}p^{e_H-2}q \cdot e(Q) = D_Qm_Q/p.
  \]
  Fourth, if $Q \in \cQ_H$, then $\cF_Q' \subseteq \cFQH$ and thus on the event~$\cT$, by~\autoref{obs:partial-H},
  \[
    |\partial\cF_Q'[G]| \le k(Q) \cdot \max_{v \in V(K_n)} |\partial\cH_v[G]| \le 2v_H \cdot k(Q) n^{v_H-1}p^{e_H-1};
  \]
  moreover, $\binom{n}{2} < n^2$.  Hence, $\min\big\{ \binom{n}{2}, |\partial\cF_Q'[G]|\big\} \le D_Qm_Q/p$.
\end{proof}

Moreover, denote by $\cG_Q$ the event $\cT \cap \cD_Q \cap \cB_{Q, \cF_Q'}(d_Q,d_Q) \cap \{Q \subseteq G\}$.

\begin{lemma}
  \label{lemma:sufficient-bound-cGQ}
  Suppose that there exists a positive constant $c$ such that
  \[
    \Pr(G_{n,p} \in \cG_Q) \le p^{e(Q)} e^{-(1+c) e(Q) \log (n^2p/e(Q))}
  \]
  for every $Q \in \cQ_L$ and that
  \[
    \Pr(G_{n,p} \in \cG_Q) \le p^{e(Q)} e^{-2k(Q)np}
  \]
  for all $Q \in \cQ_H$. Then
  \[
    \sum_{Q \in \cQ} \Pr\left(G_{n,p} \in \cG_Q\right) = o(1).
  \]
\end{lemma}
\begin{proof}
  Denote $N \coloneqq \binom{n}{2}$.  Since $e(Q) \le \beta N p$ for all $Q \in \cQ_L$ with $Q \subseteq G_{n,p}$ when $G_{n,p} \in \cG_Q \subseteq \cT$, we have
  \[
    \begin{split}
      \sum_{Q \in \cQ_L} \Pr(G_{n,p} \in \cG_Q)
      & \le \sum_{m=1}^{\beta Np} \binom{N}{m} p^{m} e^{-(1+c) m \log (Np/m)} \\
      & \le \sum_{m=1}^{\beta Np} \left(\frac{eNp}{m}\right)^m e^{-(1+c) m \log(Np/m)}
      = \sum_{m=1}^{\beta Np} e^{m-cm \log (Np/m)}.
    \end{split}
  \]
  Let $f(x) \coloneqq x - cx \log(Np/x)$ and note that $f'(x) = 1 + c - c\log(Np/x)$.
  In particular, if $x \in (0, \beta Np)$, then $f'(x) \le 1 + c - c\log(1/\beta) \le -1$, provided that $\beta$ is sufficiently small as a function of $c$ only.
  This implies that
  \[
    \sum_{Q \in \cQ_L} \Pr(G_{n,p} \in \cG_Q) \le \sum_{m =1}^{\beta Np} e^{f(m)} \le \sum_{m=1}^{\beta Np} e^{f(1)+1-m} \le \frac{e^{f(1)}}{1-1/e} \le O\big((Np)^{-c}\big).
  \]

  Since each $Q \in \cQ_H$ has a set of centre vertices that dominate all its edges,
  \[
    \begin{split}
      \sum_{Q \in \cQ_H} \Pr(G_{n,p} \in \cG_Q)
      & \le \sum_{k=1}^{n} \binom{n}{k} \left(\sum_{m \ge 0} \binom{n}{m} p^m \right)^k \cdot e^{-2knp} \\
      & = \sum_{k=1}^{n} \binom{n}{k} (1+p)^{nk} \cdot e^{-2knp} \le \sum_{k=1}^n \binom{n}{k} e^{-npk} \\
      & = (1+e^{-np})^n - 1 \le \exp(ne^{-np}) - 1 \le 2ne^{-np},
    \end{split}
  \]
  since $ne^{-np} \le 1$.
\end{proof}

Fix an arbitrary $Q \in \cQ$.  In order to bound $\Pr(G_{n,p} \in \cG_Q)$, we will apply \autoref{prop:algorithm-result} with $d = x = d_Q$, $m = m_Q$, $D = D_Q$, and $\cF = \cF_Q$.  To do this, however, we will first have to check that the various assumptions of the proposition are satisfied.  Since $e(Q) \le \beta n^2 p \ll n^2p$ and $\cF_Q$ is a family of subgraphs of $K_n \setminus Q$, it will suffice to prove the following lemma.

\begin{lemma}
  For every $Q \in \cQ$ satisfying $Q \subseteq G \in \cT$,
  \[
    d_Q \ll n^2p
    \qquad
    \text{and}
    \qquad
    d_Q \le \vmin(Q) \cdot \kappa np.
  \]
\end{lemma}
\begin{proof}
  The asymptotic inequality $d_Q \ll n^2p$ is straightforward to verify.  Indeed, recall that $e(Q) \ll n^2p$ for all $Q \in \cQ$ and $k(Q) \ll n$ for all $Q \in \cQ_H$.
  To see that the second inequality holds as well, consider first the case $Q \in \cQ_L$.
  Here, we have $\vmin(Q) = |V^1(Q)| \ge e(Q)/\Delta(Q) \ge e(Q)\log n / (\kappa np)$ and thus
  \[
    \frac{d_Q}{\vmin(Q)} \le \kappa np \cdot \max\left\{\sqrt{\eta}, \frac{8r}{\log n}\right\} \le \kappa np.
  \]
  Assume now that $Q \in \cQ_H$.  Our assumption that $Q \subseteq G \in \cT$ implies that $\vmin(Q) \ge c \cdot \min\{n, k(Q) np\}$ for some positive constant $c$ that depends only on $\eta$ and thus
  \[
    \frac{d_Q}{\vmin(Q)} \le c^{-1} \cdot \max\{32, 32k(Q)p\} \ll np,
  \]
  since $k(Q) \ll n$.
\end{proof}

Define
\[
  \nu_Q \coloneqq m_Q + (r^2+1)(d_Q+1).
\]
\autoref{prop:algorithm-result} with $K_Q = 3$ implies that, for some constant $Z = Z(\alpha, p_0, r)$,
\begin{equation}
  \label{eq:cGQ-bound}
  \Pr(G_{n,p} \in \cG_Q) \le p^{e(Q)} \cdot D_Q^{d_Q} \cdot \left(Z^{d_Q} \cdot \xi_Q + e^{-3d_Q\log(n^2p/d_Q)}\right),
\end{equation}
where
\[
  \xi_Q \coloneqq \Pr\big(G_{n,p} \text{ is $\cC_Q$-rigid and } \nu(\cF_Q[G_{n,p} \cap \ext(\core_{\cC_Q}(G_{n,p}))]) \le \nu_Q \mid Q \subseteq G_{n,p}\big).
\]
Gearing up towards bounding $\xi_Q$ from above, let $\cS_Q$ denote the collection of all tuples $S = (S_1, \dotsc, S_r)$ of pairwise disjoint sets of vertices that are compatible with $Q$ and satisfy $\min_i |S_i| \ge (1-4r\alpha)\cdot n/r$.
Since, for every $S \in \cS_Q$, the event $\nu(\cF_Q[G_{n,p} \cap \ext(S)]) \le \nu_Q$ is decreasing, determined by $\ext(S)$, and independent of the event $Q \subseteq G_{n,p}$, we may use \autoref{lemma:correlation-argument} to conclude that
\[
  \xi_Q \le r! \cdot \max_{S \in \cS_Q} \Pr\left(\nu\big(\cF_Q[G_{n,p} \cap \ext(S)]\big) \le \nu_Q\right).
\]
Since we will bound $\xi_Q$ from above with the use \autoref{cor:Janson-matchings} and \autoref{cor:extended-Janson-matchings}, it will be convenient to define, for each $S \in \cS_Q$,
\[
  \mu_Q \coloneqq \min_{S \in \cS_Q} \mu_p\big(\cF_Q[\ext(S)]\big),
  \quad
  \Delta_Q \coloneqq \Delta_p(\cF_Q),
  \quad
  \text{and}
  \quad
  \Lambda_Q \coloneqq \min\big\{\mu_Q, \mu_Q^2/\Delta_Q\big\}.
\]
The following two lemmas are the heart of this section.

\begin{lemma}
  \label{lemma:mu-Delta-nu-low-degree}
  The following holds for every $Q \in \cQ_L$:
  \begin{enumerate}[label=(\roman*)]
  \item
    \label{item:mu-Delta-nu-low-degree-sparse}
    If $p \le \Cth p_H$, then $\mu_Q \ge (1+\eps) e(Q) \log(n^2p)$ and $\Delta_Q, \nu_Q \le \eps^2/100 \cdot \mu_Q$.
  \item
    \label{item:mu-Delta-nu-low-degree-dense}
    If $p > \Cth p_H$, then $\Lambda_Q \ge 100r \cdot e(Q)\log (n^2p)$ and $\nu_Q \le \Lambda_Q / 1000$.
  \end{enumerate}
\end{lemma}
\begin{lemma}
  \label{lemma:mu-Delta-nu-high-degree}
  For every $Q \in \cQ_H$, we have $\Lambda_Q \gg k(Q)np \cdot \log D_Q$ and $\nu_Q \le \Lambda_Q/1000$.
\end{lemma}

We first show how to use the two lemmas to complete the proof of \autoref{thm:main}.
Note that it is enough to show that the assumptions of \autoref{lemma:sufficient-bound-cGQ} are satisfied.
In view of~\eqref{eq:cGQ-bound}, it suffices to estimate $(D_QZ)^{d_Q} \cdot \xi_Q$ and the quantity $\Xi_Q$ defined by
\[
  \Xi_Q \coloneqq D_Q^{d_Q} \cdot e^{-3d_Q\log(n^2p/d_Q)}.
\]

Since $D_Q \le n^2p/d_Q$ for all $Q \in \cQ$, we have
\[
  \Xi_Q \le \exp\big(-2d_Q\log(n^2p/d_Q)\big).
\]
Further, since $e(Q) \le d_Q \ll n^2p$ for all $Q \in \cQ_L$ and the function $x \mapsto x \log (a/x)$ is increasing on $(0,a/e)$, we further have
\[
  \Xi_Q \le \exp\big(-2e(Q) \log(n^2p/e(Q))\big).
\]
Similarly, since $k(Q)np \le d_Q \ll n^2p$ for all $Q \in \cQ_H$, we also have
\[
  \Xi_Q \le \exp\big(-k(Q)np \log(n/k(Q))\big) \le \exp(-3k(Q)np).
\]
In order to estimate $(D_QZ)^{d_Q} \cdot \xi_Q$, we split into three cases.

\smallskip
\noindent
\textit{Case 1. $p \le \Cth p_H$ and $Q \in \cQ_L$.}
By \autoref{lemma:mu-Delta-nu-low-degree}~\ref{item:mu-Delta-nu-low-degree-sparse}, we may apply \autoref{cor:Janson-matchings} with $\gamma \coloneqq \eps/10$ to obtain
\[
  \begin{split}
    \xi_Q & \le r! \cdot \exp\big(-(1-\eps/10) \cdot \mu_Q + 2\Delta_Q\big) \le r! \cdot \exp\big(-(1-\eps/10-\eps^2/50) \cdot \mu_Q\big) \\
    & \le r! \cdot \exp\big(-(1-\eps/7)(1+\eps) \cdot e(Q) \log(n^2p)\big) \le \exp\big(-(1+\eps/2) \cdot e(Q)\log(n^2p)\big).
  \end{split}
\]
Further,
\[
  (D_QZ)^{d_Q} \le \left(\frac{4e_HZ}{\kappa}\right)^{\sqrt{\eta} e(Q)\log n} \le \exp\left(\eps/4 \cdot e(Q)\log(n^2p)\right).
\]
We may thus conclude that
\[
  \frac{\Pr(\cG_Q)}{p^{e(Q)}}  \le (D_QZ)^{d_Q} \cdot \xi_Q + \Xi_Q  \le e^{-(1+\eps/4) e(Q) \log (n^2p)} + e^{-2e(Q)\log(n^2p/e(Q))},
\]
as needed.

\smallskip
\noindent
\textit{Case 2. $p > \Cth p_H$ and $Q \in \cQ_L$.}
By~\autoref{lemma:mu-Delta-nu-low-degree}~\ref{item:mu-Delta-nu-low-degree-dense} and \autoref{cor:extended-Janson-matchings}, we have
\[
  \xi_Q \le r! \cdot \exp(-\Lambda_Q/10) \le r! \cdot \exp\big(-10r \cdot e(Q) \log (n^2p)\big).
\]  
Further,
\[
  (D_QZ)^{d_Q} \le \max\left\{\left(\frac{Zn^2p}{8re(Q)}\right)^{8re(Q)}, \, \left(\frac{\Cpar Z}{\kappa}\right)^{\sqrt{\eta}e(Q)\log n}\right\} \le \exp\big(9re(Q)\log (n^2p)\big),
\]
We may thus conclude that
\[
  \frac{\Pr(\cG_Q)}{p^{e(Q)}}  \le (D_QZ)^{d_Q} \cdot \xi_Q + \Xi_Q  \le r! \cdot e^{-re(Q) \log (n^2p)} + e^{-2e(Q)\log(n^2p/e(Q))},
\]
as needed.

\smallskip
\noindent
\textit{Case 3. $Q \in \cQ_H$.}
In this case, $d_Q = 32k(Q)np$ and thus, by~\autoref{lemma:mu-Delta-nu-high-degree}, we may apply \autoref{cor:extended-Janson-matchings} to obtain
\[
  (D_QZ)^{d_Q} \cdot \xi_Q \le r! \cdot \exp\big(32\log(D_QZ) \cdot k(Q)np - \Lambda_Q/10\big) \le \exp(-3k(Q) np).
\]  
We conclude that
\[
  \frac{\Pr(\cG_Q)}{p^{e(Q)}}  \le (D_QZ)^{d_Q} \cdot \xi_Q + \Xi_Q  \le 2e^{-3k(Q)np},
\]
as needed.

\begin{proof}[Proof of~\autoref{lemma:mu-Delta-nu-low-degree}]
  Assume first that $p \le \Cth p_H$.  Since in this case we set $\cF_Q' = \cFQL$, it follows from \autoref{lemma:FQL-bounds} that
  \[
    \begin{split}
      \mu_Q
      & \ge (\pi_H-o(1)) \cdot e(Q) \cdot \big((1-4\alpha r) \cdot n/r\big)^{v_H-2} p^{e_H-1} \\
      & \ge (1-4\alpha r-o(1))^{v_H-2} \cdot e(Q) \cdot r^{2-v_H} \pi_H n^{v_H-2} \cdot ((1+\eps)p_H)^{e_H-1}.
    \end{split}
  \]
  Since $e_H \ge 3$ and $\alpha \ll \eps$, it follows from the definition of $p_H$ and~\eqref{eq:theta_H} that
  \[
    \begin{split}
      \mu_Q & \ge (1+2\eps) \cdot e(Q) \cdot r^{2-v_H} \pi_H \theta_H^{e_H-1} \cdot \log n \\
      & = (1+2\eps) \cdot e(Q) \cdot \left(2-\frac{1}{m_2(H)}\right) \cdot \log n
      \ge (1+\eps) \cdot e(Q) \cdot \log(n^2p),
    \end{split}
  \]
  where the last inequality follows as $p \le \Cth p_H = O(n^{-1/m_2(H)} \log n)$.
  Furthermore, \autoref{lemma:FQL-bounds} also yields
  \[
    \frac{\Delta_Q}{\mu_Q} \le \frac{\Clow \kappa n^{v_H-2}p^{e_H-1}}{\log n}
    \le \frac{\Clow \kappa n^{v_H-2}(\Cth p_H)^{e_H-1}}{\log n}
    = \Clow \kappa (\Cth\theta_H)^{e_H-1} \le \frac{\eps^2}{100}.
  \]
  Finally, note that
  \[
    \begin{split}
      \nu_Q & = m_Q + (r^2+1)(d_Q+1) \le \kappa e(Q) n^{v_H-2}p^{e_H-1} + 2r^2\sqrt{\eta}e(Q)\log n \\
      & \le \kappa (\Cth\theta_H)^{e_H-1} e(Q) \log n + 2r^2\sqrt{\eta}e(Q)\log n \le (\eps^2/100) \cdot \mu_Q.
    \end{split}
  \]

  Assume now that $p > \Cth p_H$ and $Q \in \cQ_L^1$.  Since in this case we also have $\cF_Q' = \cFQL$, we may again use \autoref{lemma:FQL-bounds} to conclude that
  \[
    \begin{split}
      \mu_Q & \ge (\pi_H-o(1)) \cdot e(Q) \cdot \big((1-4\alpha r) \cdot n/r\big)^{v_H-2} p^{e_H-1} \\
      & \ge (\pi_H/2) \cdot e(Q) \cdot r^{-v_H} n^{v_H-2} (\Cth p_H)^{e_H-1} \\
      & = (\pi_H/2) \cdot e(Q) \cdot r^{-v_H} \cdot (\Cth\theta_H)^{e_H-1} \cdot \log n \ge 100r \cdot e(Q) \log (n^2p).
    \end{split}
  \]
  and that
  \begin{equation}
    \label{eq:mu-squared-Delta-dense}
    \frac{\mu_Q^2}{\Delta_Q} = \frac{\mu_Q}{\Delta_Q / \mu_Q} \ge \frac{(\pi_H/2) \cdot e(Q) \cdot (n/(2r))^{v_H-2}p^{e_H-1}}{\kappa \Clow n^{v_H-2}p^{e_H-1} / \log n} \ge 100r \cdot e(Q) \log (n^2p).
  \end{equation}
  Finally, note that
  \[
    \nu_Q \le (r^2+2)(d_Q+1) \le 10r^3e(Q) \le r^3\Lambda_Q/\log n \le \Lambda_Q/1000.
  \]

  Last but not least, assume that $p > \Cth p_H$ and $Q \in \cQ_L^2$.
  By the definition of $\cF_Q'$ (see \autoref{lemma:choices-of-cH-low-deg}~\ref{item:first-Q-good-low-deg} and~\ref{item:second-Q-good-low-deg}) and \autoref{lemma:FQL-bounds},
  \[
    \begin{split}
      \mu_Q & \ge (q/2) \cdot \min_{S \in \cS}\mu_p(\cFQL[\ext(S)]) \ge q \cdot (\pi_H/4) \cdot e(Q) \cdot (2r)^{2-v_H} n^{v_H-2} p^{e_H-1} \\
      & \ge \frac{\hat{C} \pi_H}{(2r)^{v_H}\kappa} \cdot e(Q) \log n \ge 100r \cdot e(Q) \log (n^2p),
    \end{split}
  \]
  whereas, for every $S \in \cS_Q$, analogously to~\eqref{eq:mu-squared-Delta-dense},
  \[
    \frac{\mu_{Q,S}^2}{\Delta_Q} \ge \frac{(q/2)^2 \cdot \mu_p(\cFQL[\ext(S)])^2}{2q^2 \cdot \Delta_p(\cFQL)} \ge \frac{\pi_H \cdot e(Q) \log n}{16 \cdot (2r)^{v_H-2} \kappa \Clow} \ge 100r \cdot e(Q)\log(n^2p).
  \]
  Finally, note that
  \[
    \begin{split}
      \nu_Q & = m_Q + (r^2+1)(d_Q+1) \le \kappa n^{v_H-2} p^{e_H-1}q \cdot e(Q) + 5r^3\eta/\kappa \cdot e(Q) \log n \\
      & = (\hat{C} + 5r^3\eta/\kappa) \cdot e(Q)\log n \le 2\hat{C} \cdot e(Q) \log n,
    \end{split}
  \]
  where the last inequality holds as $\hat{C}$ is a constant that depends only on $H$ whereas $\eta \le \eta(H, \kappa)$.  Finally, since also $\Clow$ depends only on $H$ and $\kappa \le \kappa(H)$, is follows that $\nu_Q \le \mu_Q/1000$ as well as $\nu_Q \le \mu_Q^2/\Delta_Q$.  This completes the proof of the lemma.
\end{proof}
\begin{proof}[Proof of \autoref{lemma:mu-Delta-nu-high-degree}]
  Since $\cF_Q' \subseteq \cFQH$ is the hypergraph that satisfies the assertion of \autoref{lemma:high-degree-bounds}, we have
  \[
    \Lambda_Q \ge \chigh \cdot \min\big\{k(Q) \cdot n^{v_H-1}p^{e_H}, k(Q) \cdot n^{1+\lambda}p, n^2p\big\} \gg k(Q) \cdot np.
  \]
  Suppose first that the minimum above is achieved at $M \in \{k(Q) \cdot n^{v_H-1}p^{e_H}, n^2p\}$.
  Since $M \gg k(Q) np$ and $D_Q \le \frac{v_H M}{k(Q) np}$, we have
  \[
    \Lambda_Q \ge \frac{\chigh}{v_H} \cdot k(Q) np \cdot \frac{v_HM}{k(Q) np} \gg k(Q) np \cdot \log\left(\frac{v_HM}{k(Q)np}\right) \ge k(Q)np \cdot \log D_Q.
  \]
  Suppose now that the minimum above is achieved at $k(Q) \cdot n^{1+\lambda}p$.  In this case, since clearly $D_Q \le n^2$, we have
  \[
    \Lambda_Q \ge \chigh \cdot k(Q) n^{1+\lambda}p \gg k(Q)np \cdot \log D_Q.
  \]
  Finally, note that
  \[
    \nu_Q \le m_Q + (r^2+1)(d_Q+1) \le 40r^2k(Q) \cdot np \le \Lambda_Q/1000.
  \]
  This completes the proof of the lemma.
\end{proof}
    
\bibliographystyle{amsplain}
\bibliography{bibliography}

\appendix

\section{$H$-free subgraphs of dense graphs}
\label{sec:proof-AlonShapiraSudakov}

In this short section, we present a proof of \autoref{thm:AlonShapiraSudakov}.  Our argument is a minor adjustment of the proof of~\cite[Theorem~6.1]{AloShaSud09}.  It relies on the following generalisation of the well-known result of Andr\'{a}sfai, Erd\H{o}s, and S\'{o}s~\cite{AndErdSos74} due to Erd\H{o}s and Simonovits~\cite{ErdSim73}.
\begin{thm}[\cite{AndErdSos74, ErdSim73}]
  \label{thm:edge-critical-min-degree}
  Let $H$ be an edge-critical graph with $\chi(H) = r + 1 \ge 3$.  Every $n$-vertex, $H$-free graph with minimum degree exceeding $\frac{3r - 4}{3r - 1} \cdot n$ is $r$-partite.
\end{thm}

We will also make use of the following easy lemma.

\begin{lemma}
  \label{lemma:r-partite-subgraph}
  Let $r \ge 2$ be an integer and suppose $G'$ is an $r$-partite subgraph of a graph~$G$.  Then, $G$ has an $r$-partite subgraph of size at least $e(G') + \frac{r-1}{r} \cdot e(G - V(G'))$.
\end{lemma}

\begin{proof}[Proof of \autoref{thm:AlonShapiraSudakov}]
  Suppose that an $n$-vertex graph $G$ satisfies
  \[
    \delta(G) \ge \left(1 - \frac{3}{4(r-1)(3r-1)}\right) n + 1.
  \]
  Let $\Gamma$ be the largest $r$-partite subgraph of $G$ and let $F$ be the largest $H$-free subgraph of $G$.
  Since $\Gamma$ is $H$-free and by \autoref{lemma:r-partite-subgraph}, with $G' = \emptyset$,
  \[
    \begin{split}
      e(F) & \ge e(\Gamma) \ge \frac{r-1}{r} \cdot e(G) \ge \frac{r-1}{r} \left(\left(1 - \frac{3}{4(r-1)(3r-1)}\right) n + 1\right) \cdot \frac{n}{2} \\
    &= \frac{12r^2 -16r + 1}{8r(3r-1)} \cdot n^2 +\frac{r-1}{2r} \cdot n.
  \end{split}
  \]
  We now construct a sequence $F = F_n \supseteq F_{n-1} \supseteq \dotsb \supseteq F_s$ as follows:
  If $F_k$ has a vertex of degree at most $\frac{3r - 4}{3r - 1} \cdot k$ we delete that vertex to obtain $F_{k-1}$; otherwise, if $\delta(F_k) > \frac{3r-4}{3r-1}$ and let $s = k$.  Since $F_s \subseteq F$ is $H$-free, \autoref{thm:edge-critical-min-degree} implies that $F_s$ is $r$-partite.  Hence,
  \[
    \begin{split}
      \frac{r-1}{2r} \cdot s^2 &\ge e(F_s) \ge e(F) - \sum_{k=s+1}^n \frac{3r - 4}{3r - 1} \cdot k
      = e(F) - \frac{3r - 4}{3r - 1} \cdot \frac{(n-s)(n+s+1)}{2} \\
      & \ge \frac{12r^2 -16r + 1}{8r(3r-1)} \cdot n^2 + \frac{r-1}{2r} \cdot n - \frac{3r-4}{3r-1} \cdot \frac{n^2-s^2+n}{2} \\
      & > \frac{12r^2 -16r + 1}{8r(3r-1)} \cdot n^2 - \frac{3r - 4}{3r - 1} \cdot \frac{n^2 - s^2}{2}.
    \end{split}
  \]
  This implies that $\frac{s^2}{2r(3r-1)} > \frac{n^2}{8r(3r-1)}$ and so $s > \frac{n}{2}$.
  Let $m \coloneqq e(G - V(F_s))$ and note that
  \[
    m \ge (n-s) \cdot \delta(G) - \binom{n-s}{2} \ge \left(\frac{12r^2 - 16r + 1}{4(r-1)(3r-1)} \cdot n + 1\right) \cdot (n-s) - \frac{(n-s)^2}{2}.
  \]
  By \autoref{lemma:r-partite-subgraph},
  \[
    \begin{split}
      e(\Gamma) &\ge e(F_s) + \frac{r-1}{r} \cdot m \ge e(F) - \frac{3r - 4}{3r - 1} \cdot \frac{(n-s)(n+s+1)}{2} + \frac{r-1}{r} \cdot m \\
      &\ge e(F) - \frac{3r - 4}{3r - 1} \cdot \frac{n^2 - s^2}{2} + \frac{r - 1}{r}\left(\frac{12r^2 - 16r + 1}{4(r-1)(3r-1)} \cdot n(n-s) - \frac{(n-s)^2}{2}\right) \\
      &= e(F) + \frac{(n-s)(2s-n)}{4r(3r-1)}.
    \end{split}
  \]
  Since $e(F) \ge e(\Gamma)$ and $s > n/2$, it must be that $s=n$ and thus $\delta(F) = \delta(F_n) > \frac{3r - 4}{3r - 1} \cdot n$.  Consequently, \autoref{thm:edge-critical-min-degree} implies that $F$ is $r$-partite.
\end{proof}

\end{document}